\documentclass[11pt, letter]{article}
                           
\usepackage{amsfonts,amssymb, latexsym, amsmath, amsthm,mathrsfs, verbatim,geometry}
\usepackage{graphics}
\usepackage{slashed}
\usepackage{url,color}
\usepackage{enumerate}
\usepackage{esint}
\usepackage{pifont}
\usepackage{upgreek}
\usepackage{times}
\usepackage{calligra}
\usepackage{graphicx}
\usepackage{caption} 
\usepackage{accents}
\usepackage{tikz}
\usetikzlibrary{calc}
\usetikzlibrary{arrows.meta}

\usepackage{mathtools}

\allowdisplaybreaks

\newtheorem{theorem}{Theorem}[section]

\newtheorem{proposition}[theorem]{Proposition}

\newtheorem{lemma}[theorem]{Lemma}
\newtheorem{remark}[theorem]{Remark}
 
\renewcommand{\theequation}{\arabic{section}.\arabic{equation}}

\def\bcr{\begin{color}{red}} 
\def\ec{\end{color}}

\title{Global solutions to the compressible Euler equations with heat transport by convection around Dyson's isothermal affine solutions}

\author{Calum Rickard\footnote{Department of Mathematics, University of Southern California, Los Angeles, USA}}

\date{}
\begin{document}

\maketitle

\abstract{Global solutions to the compressible Euler equations with heat transport by convection in the whole space are shown to exist through perturbations  of  Dyson's isothermal affine solutions \cite{10.2307/24902147}.  This  setting presents new difficulties because of the vacuum at infinity behavior of the density. In particular, the perturbation  of isothermal motion  introduces a Gaussian function into our stability analysis and a novel finite propagation result is proven to handle potentially unbounded terms arising from the presence of the Gaussian. Crucial stabilization-in-time effects of the background motion are mitigated through the use of this finite propagation result however and a careful use of the heat transport formulation in conjunction with new time weight manipulations are used to establish global existence. The heat transport by convection offers unique physical insights into the model and mathematically, we use a controlled spatial perturbation in the analysis of this feature of our system which leads us to exploit source term estimates as part of our techniques.}



 
\section{Introduction}\label{S:FORM}
We consider compressible Euler equations  with heat transport by convection  for ideal gases in three space dimensions
\begin{alignat}{2} 
\partial_t \rho + \text{div}(\rho\mathbf{u}) & = 0, \label{E:CONTINUITY} \\
\rho(\partial_t \mathbf{u} + \mathbf{u} \cdot \nabla \mathbf{u}) + \nabla p & = 0, \label{E:MOM} \\
\alpha (\partial_t T + \mathbf{u} \cdot \nabla T) +T \, \text{div}(\mathbf{u}) & = 0, \label{E:EE}
\end{alignat}
where $\mathbf{u}:\mathbb{R}_+ \times \mathbb{R}^3\rightarrow\mathbb{R}^3 $ is the fluid velocity vector field, $\rho :\mathbb{R}_+ \times \mathbb{R}^3 \rightarrow \mathbb{R}_+ $ is the density, $T:\mathbb{R}_+ \times \mathbb{R}^3\rightarrow \mathbb{R}_+ $ is the temperature, $p:\mathbb{R}_+ \times \mathbb{R}^3\rightarrow\mathbb{R}_+ $ is the pressure and $\alpha > 0$ is the heat capacity at constant volume  \cite{kondepudi2014modern}, a physical constant. Equations (\ref{E:CONTINUITY}) and (\ref{E:MOM}) express the usual conservation of mass and momentum respectively. Equation (\ref{E:EE}) expresses the conservation of energy in terms of the temperature $T$ in the form of heat transport by convection. This formulation follows from the ideal gas assumption which lets us express the internal energy $e$ as a linear function of temperature: $e=\alpha T$ \cite{kondepudi2014modern}. 

Our equation of state is the usual equation of state for an ideal gas given in terms of independent unknowns $\rho$ and $T$ \cite{kardar2007statistical,kondepudi2014modern}
\begin{equation}
p(\rho,T)=\rho T. \label{E:EEOS}
\end{equation}
Together, equations (\ref{E:CONTINUITY})-(\ref{E:EEOS}) describe the compressible flow of an inviscid, non-conducting and  calorically perfect \cite{anderson2006hypersonic},  ideal gas. The equations (\ref{E:CONTINUITY})-(\ref{E:EE}) under consideration can also be derived in a kinetic theory framework from the Boltzmann equation \cite{guo2010acoustic,saint2009hydrodynamic} where the temperature occurs naturally as an unknown and the specific heat capacity $\alpha=\frac{3}{2}$ appears modeling the monoatomic gas.

Collectively, we will study the Cauchy problem in the whole space for the Euler system with heat transport (\ref{E:CONTINUITY})-(\ref{E:EEOS}).

Notably we are interested in a regime that is almost isothermal. Following Dyson \cite{10.2307/24902147}, we consider a system isothermal if the temperature function is space independent. We will consider a wider class of potentially space dependent temperature profiles by perturbing around a background space independent temperature. Our model admits isothermal solutions in the sense of Dyson and allows the temperature to vary in time. Mathematically, we are deviating slightly from the notion of the isothermal Euler models which do not consider temperature but instead focus on the equation of state $p=\rho$, see \cite{dong2020blowup,jenssen2020multi} for example. However with the equation of state (\ref{E:EEOS}) $p=\rho T$, and $T$ close to being space independent, our framework shares some mathematical similarities with such previously considered models.

Beyond the special Dyson solution \cite{10.2307/24902147} discussed below in Section \ref{S:AFF}, to the best of our knowledge there are no known previous global existence results for  almost isothermal or isothermal Euler, with or without heat conduction.  The main goal of this article is to construct open sets of initial data that lead to global solutions to the physically important  almost isothermal  Euler system with heat transport.

\subsection{Existence Theories for the Euler System}

Before we introduce the special Dyson solution, we briefly review some known results for the Euler equations relevant to us.  Firstly the Euler equations are hyperbolic and the existence of $C^1$ local-in-time positive density solutions follows from the theory of symmetric hyperbolic systems \cite{kato1975cauchy,majda1984compressible}. However smooth solutions are expected to breakdown in finite time: the classical result of Sideris \cite{sideris1985} shows that singularities must form if the density is a strictly positive constant outside of a bounded set. Makino-Ukai- Kawashima \cite{MUK1986} proved that singularities form for compactly supported smooth solutions moving into vacuum. A detailed description of shock formation for irrotational fluids starting with smooth initial data is given by Christodoulou-Miao~\cite{MiCr}. For a general framework in this direction see the works of Speck and Luk-Speck~\cite{luk2018shock,Sp}. Very recently Buckmaster-Shkoller-Vicol \cite{buckmaster2020shock} give a constructive proof of shock formation from an open set of initial conditions leading to vorticity formation. Further it is only known in one space dimension that the isentropic ($p=\rho^\gamma$ for a constant $\gamma >1$) Euler system allows for a globally defined notion of a unique weak solution \cite{Chen1997,DiPerna1983,LiPeSo}.  

While compression contributes to the breakdown of solutions, expansion provides a mechanism that can produce global solutions. The so-called affine motions, found across different works by Ovsyannikov \cite{ovsyannikov1956new}, Dyson \cite{10.2307/24902147} and Sideris \cite{MR3634025}, are special expanding global solutions obtained using a separation-of-variables ansatz for the Lagrangian flow map,  see Section \ref{S:AFF}. In particular, Dyson \cite{10.2307/24902147} obtained special affine solutions to the isothermal system with heat transport, where a space independent initial temperature was considered. A lot of progress has been made in the vacuum free boundary problem in this direction: notably the Sideris affine solutions \cite{MR3634025} are special solutions in this setting. Then in the isentropic setting, global stability of background affine solutions has been proven by Had\v zi\'c-Jang \cite{1610.01666} for $\gamma \in (1,\frac53]$ and then extended to the full range $\gamma>1$ by Shkoller-Sideris~\cite{shkoller2017global}. For the nonisentropic (variable entropy) vacuum free boundary problem, Rickard-Had\v zi\'c-Jang \cite{rickard2019global} recently established global existence through perturbations around a rich class of background nonisentropic affine motions. Parmeshwar-Had\v zi\'c-Jang \cite{PHJ2019} showed the global existence of expanding solutions in the vacuum setting with small densities without relying on a perturbation. Small density global solutions in the whole space were obtained by Serre \cite{Se1997} and Grassin \cite{grassin1998global} for a special class of initial data by the perturbation of expansive wave solutions to the vectorial Burgers equation with linearly growing velocities at infinity - a related idea was used in the work of Rozanova~\cite{Ro}. We remark that our result does not require smallness of the density, considers a different pressure law to nonisentropic models and allows for isothermal solutions in the sense of Dyson \cite{10.2307/24902147}.

Finally, there is limited work on the mathematical isothermal Euler models which do not consider temperature but instead focus on the equation of state $p=\rho$. Akin to Sideris' work on nonisentropic Euler \cite{sideris1985}, Dong \cite{dong2020blowup} proved a blow up result  and a finite propagation result when $p=\rho$.  Jenssen-Tsikkou \cite{jenssen2020multi} provide a construction of blow-up solutions in the radial setting, and these are shown to be weak solutions to the original $p=\rho$ system.

\subsection{Isothermal Affine Motion}\label{S:AFF}

A special global-in-time solution to the isothermal Euler system with heat transport in the whole space was given by Dyson \cite{10.2307/24902147}. Crucially, Dyson makes following the space independence assumption on the initial affine temperature profile
\begin{equation}\label{E:INITIALAFFINETEMPERATUREPROFILE}
 T_A(0,x) \equiv \overline{T} > 0, 
\end{equation}
With this assumption in hand, start by writing
\begin{equation}\label{E:AFFINE}
x(y,t)=A(t)y.
\end{equation}
Then by computing $\frac{d}{dt} x = \mathbf{u}_A$,
\begin{equation}\label{E:VELOCITY}
\mathbf{u}_A(x,t)=A'(t)A(t)^{-1}x.
\end{equation}
Then $\text{div}(\mathbf{u}_A)=\text{tr}(A'(t)A(t)^{-1})$. Hence by Jacobi's formula $(\det A(t))'=\det A(t) \text{div}(\mathbf{u}_A)$, we obtain from (\ref{E:EE})
\begin{equation}\label{E:EETA1}
\alpha \partial_t T_A + T_A \frac{(\det A(t))'}{\det A(t)} = 0
\end{equation}
Then we have
\begin{equation}
\partial_t ( \ln \left[T_A^\alpha \det A(t) \right]) =0.
\end{equation}
Then using (\ref{E:INITIALAFFINETEMPERATUREPROFILE})
\begin{equation}\label{E:EETA2}
T_A=\overline{T} [\det A(0)]^{\frac{1}{\alpha}} (\det A(t))^{-\frac{1}{\alpha}}.
\end{equation}
Hence $T_A$ does not depend on space for all time. Then, after plugging (\ref{E:AFFINE}) into the momentum and continuity equations, one can effectively separate variables and discover the associated density solution of the Euler equations
\begin{equation}
\rho_A(t,x)=\frac{e^{-\frac{1}{2} |A(t)^{-1}x|^2}}{ \det A(t)},\label{E:AFFINEDENSITY}
\end{equation}
and fundamental system of ODEs for $A(t)$
\begin{align}
A''(t)&=\overline{T} [\det A(0)]^{\frac{1}{\alpha}} (\det A(t))^{-\frac{1}{\alpha}}  A(t)^{-\top}, \label{E:AODE} \\
(A(0),A'(0)) &\in \text{GL}^+(3) \times \mathbb M^{3\times3}. \label{E:FUNDAMENTALSYSTEMFIRST}
\end{align}
In the above $\mathbb{M}^{3\times3}$ denotes the set of $3 \times 3$ matrices over $\mathbb{R}$ and $\text{GL}^+(3)=\{A \in \mathbb{M}^{3\times3} : \det A > 0\}$.  Notably (\ref{E:AODE}) is the same ODE system discovered by Sideris \cite{MR3634025}.  We let $A \in C(\mathbb{R},\text{GL}^+(3)) \cap C^\infty(\mathbb{R},\mathbb{M}^3)$ be the global solution of this system of ODEs.

\begin{remark}
The density $\rho_A$ is a Gaussian function modulated by the matrix $A(t)$ \eqref{E:AFFINEDENSITY}. This is unique to the isothermal whole space problem and in particular we note $\rho_A$  dispereses  to $0$ as $|x| \rightarrow \infty$, that is, exhibits vacuum at infinity behavior. This is different to the form of $\rho_A$ considered in the vacuum free boundary perturbation problem \cite{1610.01666,rickard2019global,shkoller2017global} where $\rho_A$ achieves a vacuum state locally through the modulation of a distance function $1-|y|^2$. Specifically, in the isentropic problem \cite{1610.01666,shkoller2017global}, $\rho_A$ is fixed to a particular distance function whereas in the nonisentropic problem \cite{rickard2019global}, $\rho_A$ has infinite dimensional freedom associated with  its  distance function. The isothermal momentum $\rho_A \mathbf{u}_A$ decays  exponentially  in space because of the Gaussian function while the affine velocity is linearly growing in space.
\end{remark}



To conclude our characterization of isothermal affine motion, we provide precise asymptotics-in-time for $A(t)$.

\begin{lemma}\label{L:AASYMPTOTICS}
Consider the initial value problem~\eqref{E:AODE}--\eqref{E:FUNDAMENTALSYSTEMFIRST} and note $\overline{T} [\det A(0)]^{\frac{1}{\alpha}}>0$. For $\alpha \geq \frac{3}{2}$, the unique solution $A(t)$ to the fundamental system~\eqref{E:AODE}--\eqref{E:FUNDAMENTALSYSTEMFIRST} has the property
\begin{equation}\label{E:DETATCUBED1}
\det A(t) \sim 1 + t^3, \quad t \geq 0 
\end{equation}
Furthermore in this case, there exist matrices $A_0,A_1,M(t)$ such that
\begin{align}\label{E:AASYMP} 
A(t) = A_0 + t A_1 + M(t), \quad t \geq 0.
\end{align}
where $A_0,A_1$ are time-independent and $M(t)$ satisfies the bounds
\begin{align}\label{E:MASYMP} 
\|M(t)\| = o_{t \rightarrow \infty}(1+t), \ \ \|\partial_t M(t)\| \lesssim (1+t)^{-\frac{3}{\alpha}}.
\end{align}
For $0 < \alpha < \frac{3}{2}$, given matrices $A_0, A_1$ with $A_1$ positive definite, there exists a unique solution $A(t)$ to the fundamental system~\eqref{E:AODE}--\eqref{E:FUNDAMENTALSYSTEMFIRST} such that (\ref{E:DETATCUBED1}), (\ref{E:AASYMP}) and (\ref{E:MASYMP}) hold. 
\end{lemma}
\begin{proof}
For all $\alpha>0$, we use Theorem 3 and Lemma 6 from \cite{MR3634025} to obtain the results. Now note
\begin{equation}
\begin{cases}
0 < \alpha < \frac{3}{2} & \Rightarrow \  \frac{1}{\alpha}+1 > \frac53 \\
\alpha \geq \frac{3}{2} & \Rightarrow \  1 < \frac{1}{\alpha}+1 \leq \frac53,
\end{cases}
\end{equation}
and $1-(\frac{1}{\alpha}+1)=-\frac{1}{\alpha}$. Then for $\alpha \geq \frac{3}{2}$, we additionally use Lemma A.1 from \cite{1610.01666}. For $0 < \alpha < \frac{3}{2}$, we additionally use Lemma 1 from \cite{shkoller2017global}.
\end{proof}
In this paper we restrict our attention to the class of  isothermal  affine solutions expanding linearly in each coordinate direction: namely we require
\begin{equation}\label{E:DETATCUBED2}
\det A(t) \sim 1 + t^3, \quad t \geq 0.
\end{equation}
By Lemma \ref{L:AASYMPTOTICS}, for $\alpha \geq \frac{3}{2}$ this is not a restriction at all in fact since $A(t)$ will immediately satisfy (\ref{E:DETATCUBED2}). For $0 < \alpha < \frac{3}{2}$, Lemma \ref{L:AASYMPTOTICS} shows there exists a rich class of $A(t)$ satisfying (\ref{E:DETATCUBED2}).

\begin{remark}
With expanding affine solutions in hand, we give a physical interpretation of our affine temperature profile $T_A=\overline{T} [\det A(0)]^{\frac{1}{\alpha}} (\det A(t))^{-\frac{1}{\alpha}}.$ Since $\det A(t) \sim 1 + t^3$, our space independent affine temperature $T_A \rightarrow 0$ as $t \rightarrow \infty$, that is, the isothermal gas becomes cooler for large time. This is a consequence of the heat conduction equation (\ref{E:EE}). 
\end{remark}

We denote the set of affine motions under consideration by $\mathscr{S}$. To recap,  the set $\mathscr{S}$ is parametrized by the quadruple
\begin{equation}    
(A(0),A'(0),\bar{T}) \in \text{GL}^+ (3) \times \mathbb{M}^{3 \times 3} \times \mathbb{R}_+.
\end{equation}
With our set of isothermal affine motions $\mathscr{S}$ in hand, the goal of this paper is to establish the global-in-time stability of the isothermal Euler system with heat transport (\ref{E:CONTINUITY})-(\ref{E:EEOS}) by perturbing around the expanding affine motions. 

\section{Formulation and Main Global Existence Result}\label{S:FORM}  

\subsection{Lagrangian Coordinates}\label{S:LAGR}    

In order to analyze the stability problem for affine motions, we will use the Lagrangian formulation  that elucidates the perturbation of the background affine motion compared to the Eulerian formulation. In particular, the Lagrangian formulation allows us to separate variables in a sense which crucially allows us to take full advantage of the time stabilizing mechanism provided by the expanding background motion and on the other hand, isolate the spatial Gaussian profile which will require careful treatment. Notably the heat transport equation is necessary in providing sufficient decay to close estimates and establish global stability. 

We first define the flow map $\zeta$ as follows     
\begin{align}
\partial_t \zeta (t,y) &= \mathbf{u}(t,\zeta(t,y)), \\
\zeta(0,y)&=\zeta_0(y),
\end{align}
where $\zeta_0$ is a sufficiently smooth diffeomorphism to be specified. We introduce the notation
\begin{align} 
\mathscr{A}_\zeta := [D \zeta]^{-1} \quad &\text{(Inverse of the Jacobian matrix),}\label{E:SCRAZETA} \\ 
\mathscr{J}_\zeta := \det[D \zeta] \quad &\text{(Jacobian determinant),}\label{E:SCRJZETA} \\
v:=\mathbf{u} \circ \zeta \quad &\text{(Lagrangian velocity)}, \\
f:=\rho \circ \zeta \quad &\text{(Lagrangian density)}, \\
\mathcal{T}:=T \circ \zeta \quad &\text{(Langrangian temperature)}, \\
a_\zeta:= \mathscr{J}_\zeta  \mathscr{A}_\zeta \quad &\text{(Cofactor matrix)}.
\end{align}
In this framework material derivatives reduce to pure time derivatives and in particular, the temperature equation (\ref{E:EE}) is reformulated as
\begin{equation}\label{E:EELAGRANGIAN1}
\alpha \partial_t \mathcal{T} + \mathcal{T} [\mathscr{A}_\zeta]_i^j v^i,_j = 0
\end{equation}
By the standard calculation $\partial_t \mathscr{J}_\zeta = \mathscr{J}_\zeta [\mathscr{A}_\zeta]_i^j v^i,_j$
\begin{equation}\label{E:EELAGRANGIANT}
\alpha \frac{\partial_t \mathcal{T}}{\mathcal{T}} + \frac{\partial_t \mathscr{J}_\zeta}{\mathscr{J}_\zeta} = 0.
\end{equation}
Then we have
\begin{equation}
\partial_t ( \ln \left[\mathcal{T}^\alpha \mathscr{J}_\zeta \right]) =0,
\end{equation}
which implies
\begin{equation}\label{E:LAGRANGIANT}
\mathcal{T}=T_{0}(\zeta_0(y)) \left[\frac{\mathscr{J}_\zeta (0,y)}{\mathscr{J}_\zeta}\right]^{\frac{1}{\alpha}}.
\end{equation}
Furthermore is well-known \cite{Coutand2012,doi:10.1002/cpa.21517} that the conservation of mass equation (\ref{E:CONTINUITY}) gives
\begin{equation}\label{E:LD}
f(t,y)=(\mathscr{J}_\zeta(t,y))^{-1} \rho_0(\zeta_0(y))\mathscr{J}_{\zeta}(0,y).
\end{equation}
Finally using the equation of state  for an ideal gas  $p=\rho T$ the momentum equation (\ref{E:MOM}) is reformulated as
\begin{equation}\label{E:LMOM}
f \partial_{tt} \zeta_i + [\mathscr{A}_\zeta]_i^k (f \mathcal{T})_{,k}=0.
\end{equation}
Here we use coordinates $i=1,2,3$ with the Einstein summation convention and the notation $F,_k$ to denote the $k^{th}$ partial derivative of $F$.

Next introduce the following notation
\begin{equation}
w(y):=\rho_0(\zeta_0(y))\mathscr{J}_\zeta(0,y). \label{E:WNOTATION}
\end{equation}
Then using $[\mathscr{A}_\zeta]_i^k =  \mathscr{J}_\zeta^{-1} [a_\zeta]_i^k$ and the formula for $\mathcal{T}$ (\ref{E:LAGRANGIANT}) we obtain
\begin{equation}\label{E:LIE}
w \partial_{tt} \zeta_i + [a_\zeta]_i^k \left(w \mathscr{J}_{\zeta}^{-1} T_{0}(\zeta_0(y)) \left[\frac{\mathscr{J}_\zeta (0,y)}{\mathscr{J}_\zeta}\right]^{\frac{1}{\alpha}} \right)_{,k}=0,
\end{equation}
Using the Piola identity $([a_\zeta]_i^k)_{,k} = 0$ we rewrite (\ref{E:LIE}) as
\begin{equation}\label{E:LAGRANGIANPREAFF}
w \partial_{tt} \zeta_i  + \left(w [\mathscr{A}_\zeta]_i^k T_{0}(\zeta_0(y)) \left[\frac{\mathscr{J}_\zeta (0,y)}{\mathscr{J}_\zeta}\right]^{\frac{1}{\alpha}} \right)_{,k}=0.
\end{equation}
Affine motions described in Section \ref{S:AFF} can be realized as special solutions of (\ref{E:LAGRANGIANPREAFF}) of the form $\zeta(t,y)=A(t)y$ if  we make the space independence assumption on our initial affine temperature
\begin{equation}
 T_A(0,x) \equiv \overline{T} > 0, 
\end{equation} 
Now in this case, $\zeta(t,y)=A(t)y$, $\mathscr{A}_\zeta^{\top}=A(t)^{-\top}$ and $\mathscr{J}_\zeta = \det A$. Hence the ansatz transforms (\ref{E:LAGRANGIANPREAFF}) into
\begin{equation}\label{E:POSTAFFANSATZ}
w A_{tt} y + \overline{T} [\det A(0)]^{\frac{1}{\alpha}} (\det A)^{-\frac{1}{\alpha}}  A^{-\top} \nabla( w  ) = 0,
\end{equation}
We have that $w$ is independent of $t$ and hence (\ref{E:POSTAFFANSATZ}) will hold if we require
\begin{align}
A_{tt}&=\overline{T} [\det A(0)]^{\frac{1}{\alpha}} (\det A)^{-\frac{1}{\alpha}}  A^{-\top} A^{-\top} , \label{E:AFFREQ1} \\
w  y &=  -\nabla (w).  \label{E:AFFREQ2}
\end{align}
At this stage, we demand
\begin{equation}\label{E:WDEMAND}
w(y)=w(|y|)=e^{-\frac{|y|^2}{2}}.
\end{equation}
We observe that (\ref{E:AFFREQ1})-(\ref{E:WDEMAND}) are nothing but the affine solutions described in Section \ref{S:AFF}, and produce the set of affine motions $\mathscr{S}$ under consideration. Fix an element of $\mathscr{S}$.
\begin{remark}
Through (\ref{E:WNOTATION}), the initial data $\rho_0$ for our problem is chosen such that (\ref{E:WDEMAND}) is satisfied.
\end{remark}
With an affine motion from $\mathscr{S}$ fixed, we define the modified flow map $\eta:=A^{-1}\zeta$. Then $\mathscr{A}_\zeta^{\top}=A^{-\top}\mathscr{A}_\eta^\top$, $\mathscr{J}_\zeta = (\det A) \mathscr{J}_\eta$ where $\mathscr{A}_\eta^\top$, $\mathscr{J}_\eta$ are the $\eta$ equivalents of (\ref{E:SCRAZETA}), (\ref{E:SCRJZETA}) respectively. Now from (\ref{E:LAGRANGIANPREAFF}) we have 
\begin{align}
&w (\partial_{tt}\eta_{i} + 2 [A^{-1}]_{i \ell} \partial_t A_{\ell j} \partial_t \eta_{j} + [A^{-1}]_{i \ell} \partial_{tt} A_{\ell j} \eta_j) \notag \\
&\quad +[\det A(0)]^{\frac{1}{\alpha}} (\det A)^{-\frac{1}{\alpha}} [A^{-1}]_{i\ell} [A^{-1}]_{j \ell} \left( w [\mathscr{A}_\eta]_j^k T_{0}(\zeta_0(y)) \left[\frac{\mathscr{J}_\eta (0,y)}{\mathscr{J}_\eta}\right]^{\frac{1}{\alpha}} \right)_{,k}=0.
\end{align}
Via (\ref{E:AFFREQ1}) we rewrite the above equation as      
\begin{align}\label{E:ETAPRETAU}
&w ( \partial_{tt} \eta_i + 2 [A^{-1}]_{i \ell} \partial_t A_{\ell j} \partial_t \eta_{j}) + w \overline{T} [\det A(0)]^{\frac{1}{\alpha}} (\det A)^{-\frac{2}{3}-\frac{1}{\alpha}}  \Lambda_{i \ell} \eta_{\ell} \notag \\
&\quad +[\det A(0)]^{\frac{1}{\alpha}} (\det A)^{-\frac{2}{3}-\frac{1}{\alpha}} \left( w \Lambda_{ij} [\mathscr{A}_\eta]_j^k T_{0}(\zeta_0(y)) \left[\frac{\mathscr{J}_\eta (0,y)}{\mathscr{J}_\eta}\right]^{\frac{1}{\alpha}} \right)_{,k}=0,
\end{align}
where the notation $\Lambda:=(\det A)^{\frac{2}{3}} A^{-1} A^{-\top}$ has been introduced.

Writing $A=\mu O$ where $\mu:=(\det A)^{\frac{1}{3}}$ and $O \in \text{SL}(3)$, we have $A^{-1}A_{t} = \mu^{-1}\mu_{t} I + O^{-1}O_{t}$. For ease of notation set $\overline{C}:=\overline{T}[\det A(0)]^{\frac{1}{\alpha}} >0$. Taking out of a factor of $\overline{T}$ from the last term of (\ref{E:ETAPRETAU}) the $\eta$ equation is
\begin{align}
& w ( \partial_{tt} \eta_i +  2 \frac{\mu_t}{\mu} \partial_t \eta_{i} + 2 \Gamma_{ij} \partial_t \eta_{j}) + \overline{C} w \mu^{-2-\frac{3}{\alpha}} \Lambda_{i \ell} \eta_{\ell} + \notag \\
& + \overline{C} \mu^{-2-\frac{3}{\alpha}} \left( w \Lambda_{ij} [\mathscr{A}_\eta]_j^k \left[\frac{T_{0}(\zeta_0(y))}{\overline{T}}\right] \left[\frac{\mathscr{J}_\eta (0,y)}{\mathscr{J}_\eta}\right]^{\frac{1}{\alpha}} \right)_{,k}=0, \label{E:POSTMUINTROEQN}
\end{align}
where we have defined $\Gamma:=O^{-1}O_t.$ Note $\eta(y) \equiv y$ corresponds to affine motion. Introducing the perturbation
\begin{equation}
\uptheta(\tau,y):=\eta(\tau,y)-y,
\end{equation}
and using (\ref{E:AFFREQ2}), equation (\ref{E:POSTMUINTROEQN}) can be written in terms of $\uptheta$
\begin{align}\label{E:FIRSTTHETA}
&w( \partial_{tt} \uptheta_i + 2 \frac{\mu_t}{\mu} \partial_t \uptheta_{i} + 2 \Gamma_{ij} \partial_t \uptheta_{j} ) + \overline{C} w \mu^{-2-\frac{3}{\alpha}} \Lambda_{i \ell} \uptheta_\ell \notag \\ 
& \quad + \overline{C} \mu^{-2-\frac{3}{\alpha}} \left(w  \Lambda_{ij} \left( [\mathscr{A}_\eta]_j^k \left[ \frac{T_{0}(\zeta_0(y))}{\overline{T}} \right] \left[ \frac{\mathscr{J}_\eta (0,y)}{\mathscr{J}_{\eta}} \right]^{\frac{1}{\alpha}} -\delta_j^k \right) \right)_{,k}=0.
\end{align}
Now we discuss how we choose the class of general initial data $T_0$ for our original problem. In particular we suppose
\begin{equation}\label{E:ORIGINALT0}
T_{0}(\zeta_0(y))=[\mathscr{J}_{\eta}(0,y)]^{-\frac{1}{\alpha}}\overline{T}(1+\beta(y)),
\end{equation}
where $\beta:\mathbb{R}^3 \rightarrow \mathbb{R}$ is a smooth compactly supported function in the unit ball in $H^{k}(\mathbb{R}^3)$ with the following smallness condition
\begin{equation}\label{E:BETADEMAND}
\| \beta \|^{ 2 }_{H^{k}(\mathbb{R}^3)} \leq \lambda,
\end{equation}
for $k \in \mathbb{Z}_{\geq 0}$ taken sufficiently large and $\lambda > 0$ taken sufficiently small, to be specified later by Theorems \ref{T:LWPGAMMALEQ5OVER3} and \ref{T:MAINTHEOREMGAMMALEQ5OVER3}.
\begin{remark}
Through (\ref{E:ORIGINALT0}), we are allowing our initial temperature to vary spatially in a controlled way from the space independent affine temperature. As we will see,  $1+\beta$ will appear in our high order energy and $\beta$ itself will contribute as a source term in our estimates. This motivates the smallness assumption (\ref{E:BETADEMAND}). 
\end{remark}
Now from (\ref{E:ORIGINALT0})
\begin{equation}
\frac{T_0 (\zeta_0(y))[\mathscr{J}_{\eta}(0,y)]^{\frac{1}{\alpha}}}{\overline{T}}=1+\beta(y).
\end{equation}
Then the last term of (\ref{E:FIRSTTHETA}) is
\begin{align}
&\overline{C} \mu^{-2-\frac{3}{\alpha}} (w  \Lambda_{ij} ( [\mathscr{A}_\eta]_j^k \mathscr{J}_{\eta}^{-\frac{1}{\alpha}} +\beta [\mathscr{A}_\eta]_j^k \mathscr{J}_{\eta}^{-\frac{1}{\alpha}} - \delta_j^k ) )_{,k} \notag \\
&=\overline{C} \mu^{-2-\frac{3}{\alpha}} ( w  \Lambda_{ij} ( 1+\beta ) ( [\mathscr{A}_\eta]_j^k \mathscr{J}_{\eta}^{-\frac{1}{\alpha}} - \delta_j^k ) )_{,k} + \overline{C} \mu^{-2-\frac{3}{\alpha}}  \Lambda_{ik} (w \beta)_{,k} \, ,
\end{align}
where the second term on the right hand side will act as a source term in our estimates. Thus from (\ref{E:FIRSTTHETA}) we have
\begin{align}
&\partial_{tt} \uptheta_i + 2 \frac{\mu_t}{\mu} \partial_t \uptheta_{i} + 2 \Gamma_{ij} \partial_t \uptheta_{j} + \overline{C}  \mu^{-2-\frac{3}{\alpha}} \Lambda_{i \ell} \uptheta_\ell \notag \\
&+ \frac{\overline{C} \mu^{-2-\frac{3}{\alpha}}}{w} ( w  \Lambda_{ij} ( 1+\beta ) ( [\mathscr{A}_\eta]_j^k \mathscr{J}_{\eta}^{-\frac{1}{\alpha}} - \delta_j^k ) )_{,k} + \frac{\overline{C} \mu^{-2-\frac{3}{\alpha}}}{w}  \Lambda_{ik} (w \beta)_{,k}=0. \label{E:THETATA}
\end{align}
Make the time of change variable
\begin{equation}
\frac{d \tau}{dt}=\frac{1}{\mu}.
\end{equation}
Then we can formulate (\ref{E:THETATA}) as 
\begin{align}
&\frac{1}{\mu^2} \partial_{\tau \tau} \uptheta_i + \frac{\mu_{\tau}}{\mu^3} \partial_\tau \uptheta_i  + \frac{2}{\mu^2} \Gamma^*_{ij} \partial_{\tau} \uptheta_j +\overline{C} \mu^{-2-\frac{3}{\alpha}} \Lambda_{i \ell} \uptheta_\ell\notag \\
& + \frac{\overline{C} \mu^{-2-\frac{3}{\alpha}}}{w} ( w  \Lambda_{ij} ( 1+\beta ) ( [\mathscr{A}_\eta]_j^k \mathscr{J}_\eta^{-\frac{1}{\alpha}} - \delta_j^k ) )_{,k} + \frac{\overline{C} \mu^{-2-\frac{3}{\alpha}}}{w}  \Lambda_{ik} (w \beta)_{,k}=0. \label{E:PRETHETAEQN}
\end{align}
where we define $\Gamma^*=O^{-1}O_\tau$ which implies $\Gamma = \frac{1}{\mu} \Gamma^*$.

To take advantage of additional time decay arising from the  heat transport, let

\begin{equation}
0<\sigma<\min\left(\frac{3}{\alpha},2\right), \quad \delta:=\frac{3}{\alpha}-\sigma > 0.
\end{equation}

Multiply  (\ref{E:PRETHETAEQN})  by $\mu^{2+\sigma}$ to finally obtain
\begin{align}
&\mu^{\sigma} \partial_{\tau \tau} \uptheta_i + \mu_{\tau} \mu^{-1+\sigma} \partial_\tau \uptheta_i + 2 \mu^{\sigma} \Gamma^*_{ij} \partial_{\tau} \uptheta_j + \overline{C} \mu^{-\delta} \Lambda_{i \ell} \uptheta_\ell \notag \\
&+\frac{\overline{C} \mu^{-\delta}}{w} ( w  \Lambda_{ij} ( 1+\beta ) ( [\mathscr{A}_\eta]_j^k \mathscr{J}^{-\frac{1}{\alpha}} - \delta_j^k ) )_{,k} + \frac{\overline{C} \mu^{-\delta}}{w}  \Lambda_{ik} (w \beta)_{,k}=0. \label{E:THETAEQNLINEARENERGYFUNCTION}
\end{align}
We consider (\ref{E:THETAEQNLINEARENERGYFUNCTION}) with the initial conditions 
\begin{equation}\label{E:THETAICGAMMALEQ5OVER3}
\uptheta(0,y)=\uptheta_0(y), \quad \uptheta_\tau(0,y)=\mathbf{V}(0,y)=\mathbf{V}_{0}(y).
\end{equation}
Above we have introduced the notation $\mathbf{V}:=\partial_\tau \uptheta$ which will be used interchangeably.

\begin{remark}[Eulerian initial density $\rho_0$ and temperature $T_0$]\label{R:IDIT}
The Eulerian initial density $\rho_0$ and temperature $T_0$ are connected to the background affine motion via
\begin{align*}
\rho_0(x)&=w((\eta_0 \circ \zeta_A(0))^{-1}(x))\det[D(\zeta_0^{-1}(x))]^{-1}, \\
T_0(x)&=\det[D(\eta_0^{-1}(x))]^{-\frac{1}{\alpha}}\overline{T}[1+\beta((\eta_0 \circ \zeta_A(0))^{-1}(x))]
\end{align*}
where the composed maps are defined by $\eta_0(y):=A^{-1}(0)\zeta_0(y)$ and $\zeta_A(0)(y):=A(0)y$.
\end{remark}

\subsection{Notation}\label{S:NOTATION}
For ease of notation first set
\begin{equation}
\mathscr{A}:=\mathscr{A}_\eta; \quad \mathscr{J}:=\mathscr{J}_\eta.
\end{equation}
Using $\mathscr{A} [D\eta] = \text{{\bf Id}}$, we have the differentiation formulae for $\mathscr{A}$ and $\mathscr{J}$
\begin{equation}\label{E:AJDIFFERENTIATIONFORMULAE}
\partial \mathscr{A}^k_i = - \mathscr{A}^k_\ell \partial \eta^\ell,_s \mathscr{A}^s_i \  ; \quad \partial \mathscr{J} = \mathscr{J} \mathscr{A}^s_\ell\partial\eta^\ell,_s
\end{equation}
for $\partial=\partial_\tau$ or $\partial=\partial_i$, $i=1,2,3$.

Let $\mathbf{F}:\Omega \rightarrow \mathbb{R}^3$ and $f:\Omega \rightarrow \mathbb{R}$ be an arbitrary vector field and function respectively. First define the gradient and divergence along the flow map $\eta$ respectively
\begin{equation}
[\nabla_\eta \mathbf{F}]^i_r:=\mathscr{A}^s_r \mathbf{F}^i_{,s}; \quad \text{div}_\eta \mathbf{F}:=\mathscr{A}^s_\ell \mathbf{F}^\ell_{,s}.
\end{equation}
For curl estimates, introduce the anti-symmetric curl and cross product matrices respectively
\begin{equation}\label{E:CURLCROSSPRODMATRICES}
[\text{Curl}_{\Lambda\mathscr{A}}\mathbf{F}]^i_j :=\Lambda_{jm}\mathscr{A}^s_m\mathbf{F}^i,_s- \Lambda_{im}\mathscr{A}^s_m\mathbf{F}^j,_s; \quad [\Lambda\mathscr{A}\nabla f \times \mathbf{F}]^i_j :=\Lambda_{jm}\mathscr{A}^s_m f_{,s} \mathbf{F}^i- \Lambda_{im}\mathscr{A}^s_m f_{,s} \mathbf{F}^j.
\end{equation}

We work with $L^2$ based norms. Define
\begin{equation} 
\|\cdot\|:=\|\cdot\|_{L^2(\mathbb{R}^3)}.
\end{equation}
Cartesian derivative operators will be used. For $\nu \in \mathbb{Z}^3_{\geq 0}$ let\begin{equation}
\partial^\nu:=\partial_{y_1}^{\nu_1}\partial_{y_2}^{\nu_2}\partial_{y_3}^{\nu_3}.
\end{equation}

For our gradient energy contribution which arises directly from our problem, we need to diagonalize then positive symmetric matrix $\Lambda=(\det A)^{\frac{2}{3}} A^{-1} A^{-\top} \in \text{SL}(3)$ as follows   
\begin{equation}\label{E:LAMBDADECOMP}
\Lambda=P^{\top}QP, \quad P \in \text{SO(3)}, \quad Q=\text{diag}(d_1,d_2,d_3), \quad d_i > 0 \text{ eigenvalues of } \Lambda.
\end{equation}
Then the following quantity will appear in the energy which arises directly from the problem
\begin{equation}\label{E:NNU}
\mathscr N_\nu : =  P\:\nabla_\eta \partial^\nu\uptheta\:P^{\top}.
\end{equation}

Finally define the important $\mu$ related quantities
\begin{equation}\label{E:MU1MU0DEFINITIONS}
\mu_1:=\lim_{\tau \rightarrow \infty} \frac{\mu_\tau(\tau)}{\mu(\tau)}, \quad \mu_0:=\frac{\sigma}{2}\mu_1,
\end{equation}
 where we recall $0<\sigma<\min\left(\frac{3}{\alpha},2\right)$. 

\subsection{High-order Quantities}\label{S:HOQGAMMALEQ5OVER3}
Let $N \in \mathbb{N}$. To measure the size of the deviation $\uptheta$, we define the high-order weighted Sobolev norm as follows 
\begin{align}
\mathcal{S}^N(\tau) &:= \sup_{0\leq \tau' \leq \tau} \Big\{  \sum_{|\nu| \leq N} \Big( \mu^{\sigma} \| \partial^\nu \mathbf{V}\|^2 + \| \partial^\nu \uptheta \|^2 \Big) +  \sum_{|\nu| \leq N-1} \Big(\|\nabla_\eta \partial^\nu\uptheta\|^2+\|\text{div}_\eta \partial^\nu \uptheta \|^2 \Big) \notag \\
& \qquad \qquad  +\sum_{|\nu|=N} \Big( \mu^{-\delta} \|\nabla_\eta \partial^\nu\uptheta\|^2+ \mu^{-\delta} \|\text{div}_\eta \partial^\nu \uptheta \|^2 \Big) \Big\}.  \label{E:SNNORMGAMMALEQ5OVER3} 
\end{align}
Modified curl terms arise during energy estimates which are not a priori controlled by the norm $\mathcal{S}^N(\tau)$. These are measured via the following high-order quantity 
\begin{equation}\label{E:BNNORMGAMMALEQ5OVER3}
\mathcal{B}^N[\mathbf{V}](\tau)=\sup_{0 \leq \tau' \leq \tau} \Big\{ \sum_{|\nu| \leq N-1} \| \text{Curl}_{\Lambda \mathscr{A}} \partial^\nu \mathbf{V}\|^2+ \sum_{|\nu| = N } \mu^{-\delta} \| \text{Curl}_{\Lambda \mathscr{A}} \partial^\nu \mathbf{V}\|^2 \Big\},
\end{equation}
with $\mathcal{B}^N[\uptheta]$ defined in the same way: $\uptheta$ replaces $\mathbf{V}$ in (\ref{E:BNNORMGAMMALEQ5OVER3}).

\subsection{Main Theorem}\label{S:MAINTHEOREM}
\textbf{Local Well-Posedness.} Before giving our main theorem, we give the local well-posedness of our system.
\begin{theorem}\label{T:LWPGAMMALEQ5OVER3}
Fix $N\geq 4$.  Then there  are $\varepsilon_0>0$, $\lambda_0 > 0$ and $T^* > 0$  such that for every $\varepsilon \in (0,\varepsilon_0]$, $\lambda \in (0,\lambda_0]$ , pair of compactly supported initial data for (\ref{E:THETAEQNLINEARENERGYFUNCTION}) $(\uptheta_0,{\bf V}_0)$ satisfying 
\begin{equation}\label{E:LANGRANGIANDATASATISFY}
\mathcal{S}^N(\uptheta_0, {\bf V}_0) + \mathcal{B}^N({\bf V}_0) \leq \varepsilon, \quad \text{supp} \,\uptheta_0 \subseteq B_1(\mathbf{0}), \text{supp}\,\mathbf{V}_0 \subseteq B_1(\mathbf{0}),
\end{equation}
 spatial temperature perturbation $\beta$ appearing in (\ref{E:THETAEQNLINEARENERGYFUNCTION}) satisfying
\begin{equation}
\| \beta \|^2_{H^{N+1}(\mathbb{R}^3)} \leq \lambda, \quad \text{supp} \, \beta \subseteq B_1(\mathbf{0}) ,
\end{equation}
 there exists a unique solution $(\uptheta(\tau), {\bf V}(\tau)):\Omega \rightarrow \mathbb R^3\times \mathbb R^3$ to (\ref{E:THETAEQNLINEARENERGYFUNCTION})-(\ref{E:THETAICGAMMALEQ5OVER3}) for all $\tau\in [0,T^*]$. The solution has the property $\mathcal{S}^N (\uptheta, {\bf V})(\tau) + \mathcal{B}^N[{\bf V}](\tau)  \lesssim  \varepsilon$ for each $\tau\in[0,T^*]$. Furthermore, the map $[0,T^*]\ni\tau\mapsto\mathcal{S}^N(\tau)\in\mathbb R_+$ is continuous.
\end{theorem}

\begin{proof}[Sketch of proof] The construction of a local Eulerian solution from generic initial data is given in Appendix \ref{A:LWP}. With an initial flow map specified by our Lagrangian initial data satisfying (\ref{E:LANGRANGIANDATASATISFY}), we define the Eulerian initial data $\rho_0$ and $T_0$, see Remark \ref{R:IDIT}, and choose $\mathbf{u}_0$ to have the same regularity. From the theory of symmetric hyperbolic systems \cite{kato1975cauchy}, the associated local Eulerian solution then preserves the regularity of its initial data. Then using the Picard Iteration Theorem for ODEs to solve for the flow map, we obtain a perturbation $\uptheta$, and associated ${\bf V}$, which solves (\ref{E:THETAEQNLINEARENERGYFUNCTION})-(\ref{E:THETAICGAMMALEQ5OVER3}). The $\mathcal{S}^N (\uptheta, {\bf V})(\tau) + \mathcal{B}^N[{\bf V}](\tau)$ bound is established through regularity obtained from the Eulerian solution and ODE theory in conjunction with elliptic regularity, desirable curl equations, and finally a standard div-curl estimate.
%
%
%
%
%
\end{proof}

\textbf{A priori assumptions.} Finally before our main theorem, make the following a priori assumptions on our local solutions from Theorem \ref{T:LWPGAMMALEQ5OVER3}
\begin{align}\label{E:APRIORI}
&\| \mathscr{A}-\textbf{Id} \|_{L^{\infty}(\mathbb{R}^3)} < \frac{1}{3}, \quad \| D\uptheta \|_{L^{\infty}(\mathbb{R}^3)}  < \frac{1}{3}, \quad \| \mathscr{J}-\textbf{Id} \|_{L^{\infty}(\mathbb{R}^3)}  < \frac{1}{3}, \notag \\
&\mathcal{S}^N(\tau) < 1/3, \quad \| D \mathbf{V} \|_{L^\infty (\mathbb{R}^3)} \leq C, \quad \| D \mathbf{V}_\tau \|_{L^\infty (\mathbb{R}^3)} \leq C,
\end{align}
for all $\tau \in [0,T^*]$. 

We are now ready to give our main theorem.   

\begin{theorem}\label{T:MAINTHEOREMGAMMALEQ5OVER3}
Fix $N\geq 4$. Consider a fixed triple
\begin{equation}    
(A(0),A'(0),\overline{T}) \in \text{GL}^+ (3) \times \mathbb{M}^{3 \times 3} \times \mathbb{R}_+,
\end{equation}
parametrizing an  isothermal  affine motion from the set $\mathscr{S}$ so that $\det A(t) \sim 1 + t^3, t \geq 0$. Then there  are  $\varepsilon_0>0$  and $\lambda_0 > 0$  such that for every $\varepsilon \in (0,\varepsilon_0]$, $\lambda \in (0,\lambda_0]$ , pair of compactly supported initial data for (\ref{E:THETAEQNLINEARENERGYFUNCTION}) $(\uptheta_0,{\bf V}_0)$ satisfying 
\begin{equation}
\mathcal{S}^N(\uptheta_0, {\bf V}_0) + \mathcal{B}^N({\bf V}_0) \leq \varepsilon, \quad \text{supp} \,\uptheta_0 \subseteq B_1(\mathbf{0}), \text{supp}\,\mathbf{V}_0 \subseteq B_1(\mathbf{0}),
\end{equation}
 spatial temperature perturbation $\beta$ appearing in (\ref{E:THETAEQNLINEARENERGYFUNCTION}) satisfying
\begin{equation}
\| \beta \|^2_{H^{N+1}(\mathbb{R}^3)} \leq \lambda, \quad \text{supp} \, \beta \subseteq B_1(\mathbf{0}) ,
\end{equation}
 there exists a global-in-time solution, $(\uptheta,{\bf V})$, to the initial value problem (\ref{E:THETAEQNLINEARENERGYFUNCTION})-(\ref{E:THETAICGAMMALEQ5OVER3})  and a constant $C>0$ such that
\begin{equation}
\mathcal S^N(\uptheta,{\bf V})(\tau)  \le C (\varepsilon+\lambda), \quad \mathcal{B}^N[{\bf V}](\tau) \leq C (\varepsilon +\lambda)(1+\tau^2) e^{-2\mu_0 \tau}, \ \ 0\le\tau<\infty, \label{E:GLOBALBOUND} 
\end{equation}
where we recall $\mu_1=\lim_{\tau \rightarrow \infty} \frac{\mu_\tau(\tau)}{\mu(\tau)}, \mu_0=\frac{\sigma}{2}\mu_1$ and $0<\sigma<\min\left(\frac{3}{\alpha},2\right)$.

\end{theorem}
 We believe Theorem \ref{T:MAINTHEOREMGAMMALEQ5OVER3} is the first global existence result for the  almost isothermal  Euler system with heat transport.

Henceforth we assume we are working with a unique local solution $(\uptheta, {\bf V}):\mathbb{R}^3 \rightarrow \mathbb R^3\times \mathbb R^3$ to (\ref{E:THETAEQNLINEARENERGYFUNCTION})-(\ref{E:THETAICGAMMALEQ5OVER3}) such that $\mathcal{S}^N(\uptheta, {\bf V}) +\mathcal{B}^N[{\bf V}] < \infty$ and $\text{supp} \,\uptheta_0 \subseteq B_1(\mathbf{0})$, $\text{supp}\,\mathbf{V}_0 \subseteq B_1(\mathbf{0})$ on $[0,T^*]$ with $T^*>0$ fixed: Theorem \ref{T:LWPGAMMALEQ5OVER3} ensures the existence of such a solution, and furthermore we assume this local solution satisfies the a priori assumptions (\ref{E:APRIORI}).

To prove our main result, we apply high order energy estimates. A similar methodology to \cite{rickard2019global} enables us to handle exponentially growing-in-time coefficients. The exponentially growing time weights take advantage of the stabilizing effect of the expanding background affine motion and were used crucially in \cite{rickard2019global} to close estimates. As seen in our high order quantities, only spatial derivative operators are used so as to keep intact the exponential structure of time weights which function as stabilizers allowing us to close estimates.

The primary difficulty in the isothermal setting is the presence of the Gaussian function 
\begin{equation}
w=e^{-\frac{|y|^2}{2}}.
\end{equation}
In particular, as seen in our high order quantities, we cannot include $w$ as a weight function since then we will have no hope to control unavoidable lower order terms using any kind of weighted Sobolev embedding. However if we do not use $w$ as a weight function, we instead must contend with potentially unbounded $y$ terms arising from the following unavoidable calculation resulting from the Gaussian form of $w$
\begin{equation}
\frac{w,_k}{w}=-y_k.
\end{equation}
Therefore without any further analysis, since we are working in the whole space, there is no chance to close estimates containing clearly unbounded $y$ terms. Furthermore since such $y$ terms arise from the nonlinear pressure term, they persist throughout the levels of higher order derivatives so we must deal with them at each stage.  There is no such difficulty in the vacuum free boundary problem since in that setting $w$ behaves like a distance function on a bounded domain instead. 

The key to overcoming this issue is that we are able to establish a finite propagation result for our equation. This will be proven immediately below in Theorem \ref{T:FINITEPROPAGATIONTHEOREM}. We will show that starting from compactly supported initial data as specified in Theorems (\ref{T:LWPGAMMALEQ5OVER3})-(\ref{T:MAINTHEOREMGAMMALEQ5OVER3}), the support of our solution grows  at most  linearly in $\tau$ from $B_1(\mathbf{0})$, the support of our initial data. Therefore $y$ terms can be estimated by a function linear in $\tau$.

This finite propagation result is applied in conjunction with the crucial exponentially growing time weights and hence it will be seen both in the energy estimates and curl estimates that we will need to be careful with any expression involving $y$ to establish sufficient time decay for our norm. In particular in the curl estimates, novel algebraic manipulations of time weights are required but fortunately, we are still able to close estimates because we have sufficient time decay from our equation.

On this note, it is in particular the temperature formulation which provides us with slightly more time decay than the isothermal setting without temperature and this allows us to close estimates. Our spatial perturbation of the temperature through $\beta$ (\ref{E:ORIGINALT0}) on the other hand means we have to deal with source terms in our estimates which arise from the last term of (\ref{E:THETAEQNLINEARENERGYFUNCTION}). Through the compact support and smallness of $\beta$, we can use similar methods to \cite{PHJ2019} to overcome these source terms and still establish global existence.

Furthermore, from the time decay provided by the temperature formulation, many of our terms will contain time weights with negative powers. Thus to control such terms without this decay, we apply a coercivity estimate technique from \cite{rickard2019global}. This technique employs the fundamental theorem of calculus to express $\uptheta$ in terms of $\mathbf{V}$ and initial data with the coercivity estimates given in Lemma \ref{L:COERCIVITY}. 

As mentioned above, we prove the fundamental finite propagation result immediately below in Section \ref{S:FINITEPROPAGATION}. Then in Section \ref{S:ENERGYESTIMATES} we prove our higher order energy estimates and in Section \ref{S:CURLESTIMATES} we establish high order curl estimates. Finally in Section \ref{S:CONTINUITY}, we prove our Main Theorem \ref{T:MAINTHEOREMGAMMALEQ5OVER3} using a continuity argument.

\section{Finite Propagation}\label{S:FINITEPROPAGATION}
We first state our finite propagation theorem and then prove the local energy and curl estimate results required to prove the theorem, with the final proof of the theorem given at the end of the Section.

\begin{theorem}[Finite Propagation]\label{T:FINITEPROPAGATIONTHEOREM}
Let $(\uptheta, {\bf V}):\Omega \rightarrow \mathbb R^3\times \mathbb R^3$ be a unique local solution to (\ref{E:THETAEQNLINEARENERGYFUNCTION})-(\ref{E:THETAICGAMMALEQ5OVER3}) on $[0,T^*]$ for $T^*>0$ fixed with $\text{supp} \,\uptheta_0 \subseteq B_1(\mathbf{0})$, $\text{supp}\,\mathbf{V}_0 \subseteq B_1(\mathbf{0})$ and assume $(\uptheta, {\bf V})$ satisfies the a priori assumptions (\ref{E:APRIORI}). Fix $N\geq 4$.  Suppose $\beta$ in (\ref{E:THETAEQNLINEARENERGYFUNCTION}) satisfies $\| \beta \|^2_{H^{N+1}(\mathbb{R}^3)} \leq \lambda$ and $\text{supp} \, \beta \subseteq B_1(\mathbf{0})$ where  $\lambda > 0$ is fixed.  Then there exists $K>0$ such that
\begin{equation}
(\uptheta(\tau,y),\mathbf{V}(\tau,y))=(0,0,0),
\end{equation}
for $\tau \in [0,T^*]$ and $|y| > 1+K\tau$.
\end{theorem}
Local energy estimates are used to prove this theorem. To this end, we first define the required local energy and curl quantities. For all of the following definitions, we fix $y \in \mathbb{R}^3$, $\tau \in [0,T^*]$ and $K>0$, and with $\tau$ fixed, then let $s \in [0,\tau]$. First define the cone $C_s$ 
\begin{equation}
C_s:=\{ (x,s') : |x-y| \leq K (\tau-s'), 0 \leq s' \leq s \}.
\end{equation}
Cross-sections of $C_s$ are
\begin{equation}\label{E:CROSSSECTIONDEFINE}
U(s'):= \{ x: |x-y| \leq K(\tau-s') \} \text{ for } s' \in [0,s].
\end{equation}
We then define our local energy and boundary energy, where we recall the definitions of $d_i$ (\ref{E:LAMBDADECOMP}) and $\mathscr{N}_\nu$ (\ref{E:NNU}),
\begin{align}
e(s)&:=\frac{1}{2} \int_{U(s)} \left(w \mu^{\sigma} \left\langle \Lambda^{-1} \mathbf{V}, \mathbf{V} \right\rangle +\overline{C} w \mu^{-\delta} |\uptheta|^2 + w \overline{C} \mu^{-\delta} (1+\beta) \mathscr{J}^{-\frac{1}{\alpha}}\Big[\sum_{i,j=1}^3 d_id_j^{-1}( [\mathscr{N}_0]^j_i )^2 \Big] \right.  \notag \\
& + \left. + \frac{\overline{C}}{\alpha} w \mu^{-\delta} (1+\beta)(\text{div}\, \uptheta)^2 \right) (s,x)  \, dx \notag \\
e_\partial(s)&:=\frac{1}{2} \int_{\partial U(s)}  \left(w \mu^{\sigma} \left\langle \Lambda^{-1} \mathbf{V}, \mathbf{V} \right\rangle +\overline{C} w \mu^{-\delta} |\uptheta|^2 + w \overline{C} \mu^{-\delta} (1+\beta) \mathscr{J}^{-\frac{1}{\alpha}} \Big[\sum_{i,j=1}^3 d_id_j^{-1} ( [\mathscr{N}_0]^j_i )^2 \Big] \right. \notag \\
&\left. + \frac{\overline{C}}{\alpha}w \mu^{-\delta} (1+\beta)(\text{div}\, \uptheta)^2 \right)(s,x)  \, dS(x). 
\end{align}
As is the case in high order estimates, modified curl energy arise during energy estimates which are not a priori controlled by the norm $\mathcal{S}^N(\tau)$. The quantities we will need to control are the local $\uptheta$ and $\mathbf{V}$ curl energies as well as the boundary $\uptheta$ curl energy
\begin{align}
b_\uptheta(s)&:=\int_{U(s)} \left( \mu^{-\delta} w  (1+\beta)| \text{Curl}_{\Lambda \mathscr{A}} \uptheta |^2 \right)(s,x) \, dx \notag \\
b_\mathbf{V}(s)&:=\int_{U(s)} \left( \mu^{-\delta} w (1+\beta) | \text{Curl}_{\Lambda \mathscr{A}} \mathbf{V} |^2 \right)(s,x) \, dx \notag \\
b_{\uptheta \partial}(s)&:=\int_{\partial U(s)} \left( \mu^{-\delta} w (1+\beta) | \text{Curl}_{\Lambda \mathscr{A}} \uptheta |^2\right)(s,x) \, dS(x)
\end{align}
\begin{remark}
We include $w$ at this level unlike our higher order quantities which is allowable because we do not need high order embeddings here. In fact it is necessary because we do not have higher order derivative control for our local quantities.
\end{remark}

We now prove our local energy estimate result.
\begin{proposition}\label{P:LOCALENERGYESTIMATE}
Let $(\uptheta, {\bf V}):\Omega \rightarrow \mathbb R^3\times \mathbb R^3$ be a unique local solution to (\ref{E:THETAEQNLINEARENERGYFUNCTION})-(\ref{E:THETAICGAMMALEQ5OVER3}) on $[0,T^*]$ for $T^*>0$ fixed with $\text{supp} \,\uptheta_0 \subseteq B_1(\mathbf{0})$, $\text{supp}\,\mathbf{V}_0 \subseteq B_1(\mathbf{0})$ and assume $(\uptheta, {\bf V})$ satisfies the a priori assumptions (\ref{E:APRIORI}). Fix $N\geq 4$.  Suppose $\beta$ in (\ref{E:THETAEQNLINEARENERGYFUNCTION}) satisfies $\| \beta \|^2_{H^{N+1}(\mathbb{R}^3)} \leq \lambda$ and $\text{supp} \, \beta \subseteq B_1(\mathbf{0})$ where  $\lambda > 0$ is fixed.  Fix $y \in \mathbb{R}^3$ and $\tau \in [0,T^*]$. Then there exists $K>0$ in (\ref{E:CROSSSECTIONDEFINE}) such that for $s \in [0,\tau]$, 
\begin{equation}\label{E:PREENERGYINEQUALITYFINITEPROP}
e(s) \lesssim e(0)+b_{\uptheta}(0)+ \int_{U(0)}  |\beta|^2 + |D\beta|^2 dx + b_{\uptheta}(s)+\int_0^s b_{\mathbf{V}} (s') ds'+\int_0^s e(s') ds'.
\end{equation}
\end{proposition}
\begin{proof}
\textbf{Energy estimates} First, since we include $w$ in our local energy, rewrite our $\uptheta$ equation (\ref{E:THETAEQNLINEARENERGYFUNCTION}) as follows
\begin{align}\label{E:THETAEQNLINEARENERGYFUNCTION2}
&w ( \mu^{\sigma} \partial_{\tau \tau} \uptheta_i + \mu_{\tau} \mu^{-1+\sigma} \partial_\tau \uptheta_i + 2 \mu^{\sigma} \Gamma^*_{ij} \partial_{\tau} \uptheta_j + \overline{C} \mu^{-\delta} \Lambda_{i \ell} \uptheta_\ell ) \notag \\
&+ \overline{C} \mu^{-\delta} (w \Lambda_{ij}  (1+\beta) (\mathscr{A}_j^k \mathscr{J}^{-\frac{1}{\alpha}}  - \delta_j^k))_{,k} +\overline{C} \mu^{-\delta} \Lambda_{ik} (w \beta)_{,k}=0.
\end{align}
Multiply (\ref{E:THETAEQNLINEARENERGYFUNCTION2}) by $\Lambda_{i m}^{-1} \partial_{\tau} \uptheta^m$, and integrate over $C_s$
\begin{align}\label{E:POSTMULTIPLYTHETA}
&\int_0^s \int_{U(s')} w ( \mu^\sigma \partial_{\tau \tau} \uptheta_i +  \mu^{-1+\sigma} \mu_{\tau} \partial_{\tau} \uptheta_i + 2 \mu^{\sigma} \Gamma^*_{ij} \partial_{\tau} \uptheta_j + \overline{C}  \Lambda_{i \ell} \mu^{-\delta} \uptheta_\ell ) \Lambda^{-1}_{i m} \partial_{\tau} \uptheta^m dx ds' \notag \\
&+ \int_0^s \int_{U(s')}  \overline{C} \mu^{-\delta} ( w  \Lambda_{ij} (1+\beta) (\mathscr{A}_j^k \mathscr{J}^{-\frac{1}{\alpha}} -\delta_j^k ) ),_k  \Lambda^{-1}_{i m} \partial_{\tau} \uptheta^m dx ds' \notag \\
&+ \int_0^s \int_{U(s')} \overline{C} \mu^{-\delta} \Lambda_{ik} (w \beta)_{,k} \Lambda^{-1}_{i m} \partial_{\tau} \uptheta^m dx ds' =0.
\end{align}
Recognizing the perfect time derivative structure of the first integral in (\ref{E:POSTMULTIPLYTHETA}), we rewrite it as
\begin{align}\label{E:FIRSTINTEGRALPOSTMULTIPLY}
&\int_0^s \int_{U(s')} \frac{1}{2} \frac{d}{d \tau} ( w \mu^\sigma \left\langle \Lambda^{-1} \mathbf{V}, \mathbf{V} \right\rangle + \overline{C} w \mu^{-\delta} |\uptheta|^2 ) dx ds'  \notag \\
&+ \int_0^s \int_{U(s')} \left(-\frac{\sigma \mu^{\sigma-1} \mu_{\tau}}{2}+\mu^{-1+\sigma} \mu_{\tau} \right) w \left\langle \Lambda^{-1} \mathbf{V}, \mathbf{V} \right\rangle dx ds' \notag \\
&+ \int_0^s \int_{U(s')} -\frac{\mu^{\sigma}}{2} w \left\langle \partial_{\tau} \Lambda^{-1} \mathbf{V}, \mathbf{V} \right\rangle + 2 \mu^{\sigma} w \left\langle \Lambda^{-1} \mathbf{V}, \Gamma^* \mathbf{V} \right\rangle dx ds' + \delta \int_0^s \int_{U(s')} \mu^{-\delta-1} \mu_{\tau} w |\uptheta|^2 dx ds'.
\end{align}
For the first integral in (\ref{E:FIRSTINTEGRALPOSTMULTIPLY}), apply the Differentiation Formula for Moving Regions, which interchanges the derivative and integrals, and then the Fundamental Theorem of Calculus to write it as,
\begin{align}
&\int_{U(s)} \frac{1}{2} \left( w \mu^\sigma \left\langle \Lambda^{-1} \mathbf{V}, \mathbf{V} \right\rangle+\overline{C} w \mu^{-\delta} |\uptheta|^2 \right)(s,x) \, dx  - \int_{U(0)} \frac{1}{2} \left( w \mu^\sigma \left\langle \Lambda^{-1} \mathbf{V}, \mathbf{V} \right\rangle+\overline{C} w \mu^{-\delta} |\uptheta|^2 \right) (0,x) \, dx \notag \\
&+\int_0^s \int_{\partial U(s')} K \frac{1}{2} \left( w \mu^\sigma \left\langle \Lambda^{-1} \mathbf{V}, \mathbf{V} \right\rangle+\overline{C} w \mu^{-\delta} |\uptheta|^2 \right) \, dS ds'.
\end{align}
The second integral in (\ref{E:FIRSTINTEGRALPOSTMULTIPLY}) is positive, 
$$-\frac{\sigma \mu^{\sigma-1} \mu_\tau}{2}+\mu^{-1+\sigma} \mu_\tau=\frac{(2-\sigma)\mu^{\sigma-1} \mu_\tau}{2} \text{ and } 0<\sigma<2,$$
so we leave it as is. 

For the third integral in (\ref{E:FIRSTINTEGRALPOSTMULTIPLY}), note with $A=\mu O$, $\mu=(\det A)^{\frac{1}{3}}, O \in \text{SL}(3)$, we have
\begin{equation} 
\partial_\tau \Lambda^{-1}-4 \Lambda^{-1}\Gamma^\ast = \partial_\tau(O^\top O) - 4 O^\top OO^{-1}O_\tau = - \partial_\tau\Lambda^{-1}  + 2(O^\top_\tau O - O^\top O_\tau).
\end{equation}
Since $O^\top_\tau O - O^\top O_\tau$ is anti-symmetric, we can reduce the third integral in (\ref{E:FIRSTINTEGRALPOSTMULTIPLY}) to
\begin{equation}
\int_0^s \int_{U(s')} \frac{\mu^\sigma}{2} w \left\langle \partial_\tau \Lambda^{-1} \mathbf{V}, \mathbf{V} \right\rangle dx ds'
\end{equation}
which is then bounded by $\int_0^s e(s') ds$ via $\partial_\tau\Lambda^{-1}=-\Lambda^{-1}(\partial_\tau \Lambda)\Lambda^{-1}$ and (\ref{E:LAMBDABOUNDSGAMMALEQ5OVER3}). The fourth integral in (\ref{E:FIRSTINTEGRALPOSTMULTIPLY}) is positive so we leave it as is. For the second integral in (\ref{E:POSTMULTIPLYTHETA}), integrate by parts in $x$ and apply the identities (\ref{E:AJIDENTITYENERGY})-(\ref{E:LAMBDAIDENTITYENERGY})
\begin{align}
&-\int_0^s \int_{U(s')} \overline{C} \mu^{-\delta} w \Lambda_{ij} (1+\beta)( \mathscr{A}_j^k \mathscr{J}^{-\frac{1}{\alpha}} -\delta_j^k)\Lambda^{-1}_{i m} \partial_{\tau} \uptheta^m,_k dx ds' \notag \\
&+\int_0^s \int_{\partial U(s')} \overline{C} \mu^{-\delta} w \Lambda_{ij} (1+\beta)( \mathscr{A}_j^k \mathscr{J}^{-\frac{1}{\alpha}} -\delta_j^k)\Lambda^{-1}_{i m} \partial_{\tau} \uptheta^m \mathbf{n}^k dS ds' \notag \\
&= \int_0^s \int_{U(s')}  \overline{C} \mu^{-\delta} w  (1+\beta) \mathscr{J}^{-\frac{1}{\alpha}} \Lambda_{\ell j}[\nabla_\eta\uptheta]^i_j  \Lambda_{im}^{-1} [\nabla_\eta \partial_{\tau} \uptheta ]^m_\ell dx ds' \notag \\
&+\int_0^s \int_{U(s')} \overline{C} \mu^{-\delta} w   (1+\beta) \mathscr{J}^{-\frac{1}{\alpha}}   [\text{Curl}_{\Lambda\mathscr{A}} \uptheta]^\ell_i  \Lambda_{im}^{-1} [\nabla_\eta \partial_{\tau} \uptheta ]^m_\ell dx ds' \notag \\
&+ \int_0^s \int_{U(s')} \overline{C} \mu^{-\delta} w   (1+\beta) \mathscr{J}^{-\frac{1}{\alpha}}   \mathscr{A}^j_l\uptheta^l,_m \mathscr{A}^k_\ell \uptheta^\ell,_j \partial_{\tau} \uptheta^m,_k   dx ds' \notag \\
&-\int_0^s \int_{U(s')}  \overline{C} \mu^{-\delta} w  (1+\beta)(\mathscr{J}^{-\frac{1}{\alpha}} -1) \partial_{\tau} \uptheta^k,_k  dx ds' \notag \\
&+\int_0^s \int_{\partial U(s')} \overline{C} \mu^{-\delta} w \Lambda_{ij} (1+\beta) ( \mathscr{A}_j^k \mathscr{J}^{-\frac{1}{\alpha}} -\delta_j^k ) \Lambda^{-1}_{i m} \partial_{\tau} \uptheta^m \mathbf{n}^k dS ds', \label{E:POSTINTEGRATEBYPARTS}
\end{align}
where $\mathbf{n}$ is the outward facing unit normal to $U(s')$. Using Lemma \ref{L:KEYLEMMA}, the Differentiation Formula for Moving Regions and the Fundamental Theorem of Calculus, the first integral on the right hand side of (\ref{E:POSTINTEGRATEBYPARTS}) gives a gradient energy contribution with some remainder terms
\begin{align}
&\int_{U(s)}   \frac{\overline{C} \mu^{-\delta} w}{2}  (1+\beta) \mathscr{J}^{-\frac{1}{\alpha}}\sum_{i,j=1}^3 d_id_j^{-1}([\mathscr{N}_\nu]^j_i )^2 (s,x)dx \notag \\
&- \int_{U(0)}   \frac{\overline{C} \mu^{-\delta} w}{2}   (1+\beta) \mathscr{J}^{-\frac{1}{\alpha}}\sum_{i,j=1}^3 d_id_j^{-1}([\mathscr{N}_\nu]^j_i)^2 (0,x) dx \notag \\
&+\int_0^{s} \int_{\partial U(s')} K \frac{\overline{C} \mu^{-\delta} w}{2}  (1+\beta) \mathscr{J}^{-\frac{1}{\alpha}}\sum_{i,j=1}^3 d_id_j^{-1}([\mathscr{N}_\nu]^j_i )^2  dS ds' \notag \\
&+\int_0^{s} \int_{ U(s')} \overline{C} \mu^{-\delta} w  (1+\beta) \mathscr{J}^{-\frac{1}{\alpha}} \mathcal T_{0,0} dx ds' \notag \\
&+\delta \int_0^{s} \int_{ U(s')} \overline{C} \mu^{-\delta-1} \mu_{\tau} w (1+\beta) \mathscr{J}^{-\frac{1}{\alpha}} \sum_{i,j=1}^3 d_i d_j^{-1}([\mathscr{N}_\nu]^j_i )^2 dx ds' \notag \\
&+\frac{1}{2 \alpha}\int_0^{s} \int_{ U(s')} \overline{C} \mu^{-\delta}  w (1+\beta) \mathscr{J}^{-\frac{1}{\alpha}-1} \mathscr{J}_{\tau} \sum_{i,j=1}^3 d_i d_j^{-1}([\mathscr{N}_\nu]^j_i )^2 dx ds' \notag \\
&+ \int_0^{s} \int_{ U(s')} \overline{C} \mu^{-\delta} w (1+\beta) \mathscr{J}^{-\frac{1}{\alpha}} \Lambda_{\ell j} [\nabla_\eta \uptheta]_i^j \Lambda_{im}^{-1} [\nabla_\eta \uptheta]_p^m [\nabla_\eta \partial_{s'} \uptheta]_\ell^p dx ds'. \label{E:ZEROORDERIEXPGAMMALEQ5OVER3}
\end{align}
The same tools, and in addition (\ref{E:JIDENTITYENERGY}): $1-\mathscr{J}^{-\frac1\alpha}=\frac{1}{\alpha}\text{Tr}[D\uptheta ] + O(|D\uptheta|^2)$, applied to the second to last integral on the right hand side of (\ref{E:POSTINTEGRATEBYPARTS}) gives a divergence energy contribution with some remainder terms
\begin{align}
&\frac{1}{\alpha} \int_0^s \int_{U(s')} \overline{C} \mu^{-\delta} w (1+\beta)\uptheta^\ell,_\ell \partial_{\tau} \uptheta^k,_k dx ds'+\int_0^s \int_{U(s')} \overline{C} \mu^{-\delta} w (1+\beta)O(|D\uptheta|^2) \partial_{\tau} \uptheta^k,_k dx ds' \notag \\
&=\frac{1}{2 \alpha} \int_0^s \int_{U(s')} \overline{C} w \frac{d}{d \tau} \left(\mu^{-\delta} (1+\beta) (\text{div}\, \uptheta)^2 \right) dx ds' +\frac{\delta}{2 \alpha} \int_0^s \int_{U(s')} \overline{C} w \mu^{-\delta-1} \mu_{\tau} (1+\beta) (\text{div}\, \uptheta)^2  dx ds' \notag \\
&+\int_0^s \int_{U(s')} \overline{C} \mu^{-\delta} w (1+\beta) O(|D\uptheta|^2) \partial_{\tau} \uptheta^k,_k dx ds' \notag \\
&=\int_{U(s)} \frac{\overline{C}}{2 \alpha} \left( w \mu^{-\delta} (1+\beta) (\text{div}\, \uptheta)^2 \right)(s,x) \, dx  - \int_{U(0)} \frac{\overline{C}}{2 \alpha} \left( w \mu^{-\delta} (1+\beta) (\text{div}\, \uptheta)^2 \right) (0,x) \, dx \notag \\
&+\int_0^s \int_{\partial U(s')} K \frac{1}{2 \alpha} \left( \overline{C} w \mu^{-\delta} (1+\beta) (\text{div}\, \uptheta)^2 \right) \, dS ds' +\frac{\delta}{2 \alpha} \int_0^s \int_{U(s')} \overline{C} w \mu^{-\delta-1} \mu_{\tau} (1+\beta) (\text{div}\, \uptheta)^2  dx ds' \notag \\
&+\int_0^s \int_{U(s')} \overline{C} \mu^{-\delta} w (1+\beta) O(|D\uptheta|^2) \partial_{\tau} \uptheta^k,_k dx ds'.
\end{align}
For the boundary term on the right hand side of (\ref{E:POSTINTEGRATEBYPARTS}), apply (\ref{E:AJIDENTITYENERGY}):
$$\mathscr{A}_j^k \mathscr{J}^{-\frac{1}{\alpha}} -\delta_j^k=-\mathscr{A}_{\ell}^k[D \uptheta]^{\ell}_j \mathscr{J}^{-\frac{1}{\alpha}}-\delta_j^k(\frac{1}{\alpha}\text{Tr}[D \uptheta ] + O(|D\uptheta|^2),$$
and then using $\mu^{-\delta}=\mu^{\sigma-\frac{3}{\alpha}}=\mu^{\frac{\sigma}{2}}\mu^{\frac{\sigma}{2}-\frac{3}{\alpha}} \lesssim \mu^{\frac{\sigma}{2}}\mu^{\frac{\sigma-3/\alpha}{2}} = \mu^{\frac{\sigma}{2}}\mu^{\frac{-\delta}{2}}$, we have that there exists $C_1>0$ such that
\begin{align}
&\left| \overline{C} \mu^{-\delta} w \Lambda_{ij} (1+\beta) ( \mathscr{A}_j^k \mathscr{J}^{-\frac{1}{\alpha}} -\delta_j^k ) \Lambda^{-1}_{i m} \partial_{\tau} \uptheta^m \mathbf{n}^k \right| \notag \\
& \leq C_1 \left(  \frac{w \mu^{\sigma}}{2} \left\langle \Lambda^{-1} \mathbf{V}, \mathbf{V} \right\rangle + \frac{\mu^{-\delta} w \overline{C}}{2} (1+\beta) \mathscr{J}^{-\frac{1}{\alpha}} \Big[\sum_{i,j=1}^3 d_i d_j^{-1} \left( [\mathscr{N}_\nu]^j_i \right)^2 \Big] \right),
\end{align}
where we also use the a priori assumptions (\ref{E:APRIORI}) and the boundedness of $\beta$. Therefore at this stage we take $K > C_1$.

For the curl term on the right hand side of (\ref{E:POSTINTEGRATEBYPARTS}), first write as a perfect time derivative and then integrate by parts in $\tau$
\begin{align}
&\int_0^s \int_{U(s')} \overline{C} \frac{d}{d \tau} \left[ \mu^{-\delta} w (1+\beta) \mathscr{J}^{-\frac{1}{\alpha}}   [\text{Curl}_{\Lambda\mathscr{A}} \uptheta]^\ell_i  \Lambda_{im}^{-1} [\nabla_\eta \uptheta ]^m_\ell \right] dx ds' \notag \\ 
&- \int_0^s \int_{U(s')} \overline{C} \frac{d}{d \tau} \left[ \mu^{-\delta} w  (1+\beta) \mathscr{J}^{-\frac{1}{\alpha}}   [\text{Curl}_{\Lambda\mathscr{A}} \uptheta]^\ell_i  \Lambda_{im}^{-1} \mathscr{A}^m_p \right] \uptheta^p,_\ell dx ds' \notag \\
&= \int_{U(s)} \overline{C} \left(\mu^{-\delta} w   (1+\beta) \mathscr{J}^{-\frac{1}{\alpha}}   [\text{Curl}_{\Lambda\mathscr{A}} \uptheta]^\ell_i  \Lambda_{im}^{-1} [\nabla_\eta \uptheta ]^m_\ell \right)(x,s)  dx \notag \\
&- \int_{U(0)} \overline{C} \left( \mu^{-\delta} w   (1+\beta) \mathscr{J}^{-\frac{1}{\alpha}}   [\text{Curl}_{\Lambda\mathscr{A}} \uptheta]^\ell_i  \Lambda_{im}^{-1} [\nabla_\eta \uptheta ]^m_\ell\right)(0,s)  dx \notag \\
&+\int_0^s \int_{\partial U(s')} K \overline{C}  \mu^{-\delta} w (1+\beta) \mathscr{J}^{-\frac{1}{\alpha}}   [\text{Curl}_{\Lambda\mathscr{A}} \uptheta]^\ell_i  \Lambda_{im}^{-1} [\nabla_\eta \uptheta ]^m_\ell  d S d s' \notag \\
&- \int_0^s \int_{U(s')} \overline{C} \frac{d}{d \tau} \left[ \mu^{-\delta} w (1+\beta) \mathscr{J}^{-\frac{1}{\alpha}}   [\text{Curl}_{\Lambda\mathscr{A}} \uptheta]^\ell_i  \Lambda_{im}^{-1} \mathscr{A}^m_p \right] \uptheta^p,_\ell dx ds' \notag \\
&=(i)+(ii)+(iii)+(iv). \label{E:CURLIBPTAUFINITEPROP}
\end{align}
Integral $(i)$ on the right hand side of (\ref{E:CURLIBPTAUFINITEPROP}) is estimated using Young's inequality which gives for some fixed $0 < \kappa_1 \ll 1$
\begin{align}
|(i)| & \leq \int_{U(s)} \left( \left| \mu^{-\tfrac{\delta}{2}} w^{\tfrac{1}{2}} (1+\beta)^{1/2} [\text{Curl}_{\Lambda \mathscr{A}} \uptheta ]_i^\ell \right| \left| \mu^{-\tfrac{\delta}{2}} w^{\tfrac{1}{2}} (1+\beta)^{1/2} \overline{C} \mathscr{J}^{-\frac{1}{\alpha}} \Lambda_{im}^{-1} [\nabla_{\eta} \uptheta]^m_{\ell} \right| \right)(x,s)dx \notag \\
& \lesssim \frac{1}{\kappa_1} \int_{U(s)} \left(\mu^{-\delta} w (1+\beta) \left| \text{Curl}_{\Lambda \mathscr{A}}  \uptheta \right|^2 \right)(x,s) \, dx + \kappa_1 \int_{U(s)} \left(\mu^{-\delta} w (1+\beta) \left| D \uptheta \right|^2 \right)(x,s) dx \notag \\
& \lesssim \frac{1}{\kappa_1} b_{\uptheta} (s) + \kappa_1 e(s). \label{E:BTHETATERMTOESTIMATE}
\end{align}
Then $\kappa_1 e(s)$ is absorbed by the left hand side of our energy inequality. Integral $(ii)$ is bounded by initial data $e(0)+b_{\uptheta}(0).$ The boundary integral $(iii)$ on the right hand side of (\ref{E:CURLIBPTAUFINITEPROP}) is also estimated using Young's inequality which gives for some fixed $0 < \kappa_2 \ll 1$
\begin{align}
|(iii)| & \lesssim \frac{K}{\kappa_2} \int_0^s \int_{\partial U(s')} \mu^{-\delta} w (1+\beta) \left| \text{Curl}_{\Lambda \mathscr{A}} \uptheta \right|^2 dS ds' \notag\\
&+\kappa_2 K \int_0^s \int_{\partial U(s')} w \frac{\overline{C}}{2} \mu^{-\delta} (1+\beta) \mathscr{J}^{-\frac{1}{\alpha}}\Big[\sum_{i,j=1}^3 d_id_j^{-1}\left( \left( P\:\nabla_\eta \uptheta\:P^{\top}\right)^j_i \right)^2 \Big] dS ds' \notag \\
& \lesssim K \int_0^s b_{\uptheta \partial}(s') ds' + \kappa_2 K \int_0^s e_\partial(s') ds'. \label{E:BTHETAPARTIALTOESTIMATE}
\end{align}
The term $\kappa_2 K \int_0^s e_\partial(s') ds'$ can then be absorbed by the positive boundary energy contribution from above. For $K \int_0^s b_{\uptheta \partial}(s') ds'$, we note we can recover this term as follows
\begin{align}
&\int_0^s \int_{U(s')} w (1+\beta) \frac{d}{d \tau} \left[ \mu^{-\delta} | \text{Curl}_{\Lambda \mathscr{A}} \uptheta |^2 \right] dx ds' \notag \\
&=\int_{U(s)} w \mu^{-\delta} (1+\beta) | \text{Curl}_{\Lambda \mathscr{A}} \uptheta |^2 dx - \int_{U(0)} w \mu^{-\delta} (1+\beta) | \text{Curl}_{\Lambda \mathscr{A}} \uptheta |^2 dx  \notag \\
&+K \int_0^s \int_{\partial U(s')} w \mu^{-\delta} (1+\beta) | \text{Curl}_{\Lambda \mathscr{A}} \uptheta |^2 dS ds' \notag \\
&=b_{\uptheta}(s)-b_{\uptheta}(0)+ K \int_0^s b_{\uptheta \partial}(s') ds'. \label{E:RECOVERBOUNDARYCURL}
\end{align}
On the other hand note for the left hand side of (\ref{E:RECOVERBOUNDARYCURL})
\begin{align}
&\int_0^s \int_{U(s')} w (1+\beta) \frac{d}{d \tau} \left[ \mu^{-\delta} | \text{Curl}_{\Lambda \mathscr{A}} \uptheta |^2 \right] dx ds' =-\delta \int_0^s \int_{U(s')} w \mu^{-\delta-1} \mu_{\tau} (1+\beta) | \text{Curl}_{\Lambda \mathscr{A}} \uptheta |^2 dx ds' \notag \\
&+ 2 \int_0^s \int_{U(s')} w \mu^{-\delta} (1+\beta) \text{Curl}_{\Lambda \mathscr{A}} \uptheta \cdot \left[ \frac{d}{d \tau} \left[ \text{Curl}_{\Lambda \mathscr{A}} \uptheta \right] \right] dx ds' \notag \\
&=-\delta \int_0^s \int_{U(s')} w \mu^{-\delta-1} \mu_{\tau} (1+\beta) | \text{Curl}_{\Lambda \mathscr{A}} \uptheta |^2 dx ds' \notag \\
&+ 2 \int_0^s \int_{U(s')} w \mu^{-\delta} (1+\beta) \text{Curl}_{\Lambda \mathscr{A}} \uptheta \cdot  \text{Curl}_{\partial_{\tau} (\Lambda \mathscr{A})} \uptheta dx ds' \notag \\
&+ 2 \int_0^s \int_{U(s')} w \mu^{-\delta}(1+\beta)  \text{Curl}_{\Lambda \mathscr{A}} \uptheta \cdot  \text{Curl}_{\Lambda \mathscr{A}} \mathbf{V} dx ds'.
\end{align}
The first two terms on the right hand side above are remainder terms and the last term is estimated  bounded by a remainder term and $\int_0^s b_{\mathbf{V}}(s') ds'$ contribution
\begin{equation}
\int_0^s \int_{U(s')} w \mu^{-\delta} (1+\beta) | \text{Curl}_{\Lambda \mathscr{A}} \uptheta |^2 dx ds' + \int_0^s \int_{U(s')} w \mu^{-\delta} (1+\beta) | \text{Curl}_{\Lambda \mathscr{A}} \mathbf{V} |^2 dx ds'.
\end{equation}
For integral $(iv)$ 
\begin{align}
&(iv)=- \int_0^s \int_{U(s')} \overline{C} \frac{d}{d \tau} \left[ \mu^{-\delta} w (1+\beta) \mathscr{J}^{-\frac{1}{\alpha}}   [\text{Curl}_{\Lambda \mathscr{A}} \uptheta]^\ell_i  \Lambda_{im}^{-1} \mathscr{A}^m_p \right] \uptheta^p,_\ell dx ds' \notag \\
&=\delta \int_0^s \int_{U(s')} \overline{C} \mu^{-\delta-1} \mu_{\tau} w (1+\beta) \mathscr{J}^{-\frac{1}{\alpha}} [\text{Curl}_{\Lambda \mathscr{A}} \uptheta]^\ell_i \Lambda_{im}^{-1} \mathscr{A}^m_p \uptheta^p,_\ell dx ds' \notag \\
&-\int_0^s \int_{U(s')} \overline{C} \mu^{-\delta} w  (1+\beta) \frac{d}{d \tau} \left[ \mathscr{J}^{-\frac{1}{\alpha}}  \Lambda_{im}^{-1} \mathscr{A}^m_p \right] [\text{Curl}_{\Lambda \mathscr{A}} \uptheta]^\ell_i \uptheta^p,_\ell dx ds' \notag \\
&-\int_0^s \int_{U(s')} \overline{C} \mu^{-\delta} w (1+\beta) \mathscr{J}^{-\frac{1}{\alpha}} \Lambda_{im}^{-1} \mathscr{A}^m_p [\text{Curl}_{\partial_{\tau}(\Lambda \mathscr{A})} \uptheta]^\ell_i \uptheta^p,_\ell dx ds' \notag \\
&-\int_0^s \int_{U(s')} \overline{C} \mu^{-\delta} w (1+\beta) \mathscr{J}^{-\frac{1}{\alpha}} [\text{Curl}_{\Lambda \mathscr{A}} \mathbf{V}]^\ell_i \Lambda_{im}^{-1} \mathscr{A}^m_p \uptheta^p,_\ell dx ds',
\end{align}
where the last term on the right hand side of above is estimated by
\begin{equation}\label{E:BVTOESTIMATE}
\int_0^s b_{\mathbf{V}} (s') ds' + \int_0^s e(s') ds'.
\end{equation}
All remainder integrals either appear as positive contributions on the left hand side of our inequality or they are bounded by
\begin{equation}
\int_0^s e(s') ds'.
\end{equation}
It remains to estimate the third integral in (\ref{E:POSTMULTIPLYTHETA}) which will contribute source terms
\begin{align}
&\int_0^s \int_{U(s')} \overline{C} \mu^{-\delta} \Lambda_{ik} (w \beta)_{,k} \Lambda^{-1}_{i m} \partial_{\tau} \uptheta^m dx ds' \notag \\
&=\int_0^s \int_{U(s')} ( -\overline{C} \mu^{-\delta} w \beta \Lambda_{ik} x_k + \overline{C} \mu^{-\delta} w \Lambda_{ik} \beta_{,k} ) \Lambda^{-1}_{i m} \partial_{\tau} \uptheta^m dx ds'. \label{E:SOURCETERMSFINITEPROP}
\end{align}
where we have used $w,_k=-x_k w,_k$. For the first term on the right hand side of (\ref{E:SOURCETERMSFINITEPROP}), first recall $y \in \mathbb{R}^3$ and $\tau \in [0,T]$ are fixed. Then for $x \in U(s')$, first note $w(x) |x|^2 \lesssim 1$ because $w$ is a Gaussian function. Then
\begin{align}
&\left| \int_0^s \int_{U(s')} - \overline{C} \mu^{-\delta} w \beta \Lambda_{ik} x_k  \Lambda^{-1}_{i m} \partial_{\tau} \uptheta^m dx ds' \right|  \lesssim \int_0^s \left| \int_{U(s')} \mu^{\sigma/2} \mu^{-\delta} w \beta   \Lambda_{ik} x_k \Lambda^{-1}_{i m} \mathbf{V}^m  dx \right|  ds' \notag \\
& \lesssim  \int_0^s \left( \int_{U(s')} \mu^{\sigma} w |\mathbf{V}|^2 dx ds' \right)^{1/2} \left( \int_{U(s')} \mu^{-2 \delta} w |x|^2 |\beta|^2 dx \right)^{1/2} ds' \notag \\
& \lesssim \int_0^s [e(s')]^{1/2}  \left( \int_{U(s')} \mu^{-2 \delta} |\beta|^2 dx \right)^{1/2} ds' \notag \\
& \lesssim \int_0^s e(s') ds' + \int_0^s  \int_{U(s')} \mu^{-2 \delta} |\beta|^2 dx ds' \notag \\
& \lesssim \int_0^s e(s') ds' + \int_0^s \int_{U(0)} \mu^{-2 \delta} |\beta|^2 dx ds' \notag \\
& \lesssim \int_0^s e(s') ds' + \int_{U(0)} |\beta|^2 dx \int_0^s \mu^{-2 \delta} ds' \notag \\
& \lesssim \int_0^s e(s') ds' + \int_{U(0)} |\beta|^2 dx
\end{align}
where we have used $U(s) \subseteq U(s')$ for $s' \leq s$ and the integrability of negative powers of $\mu$. A similar argument for the second term on the left hand side of (\ref{E:SOURCETERMSFINITEPROP}), this time simply using $w \lesssim 1$, gives the bound
\begin{equation}
\left| \int_0^s \int_{U(s')}  \overline{C} \mu^{-\delta} w \Lambda_{ik} \beta_{,k} \Lambda^{-1}_{i m} \partial_{\tau} \uptheta^m dx ds' \right| \lesssim \int_0^s e(s') ds' + \int_{U(0)} |D \beta|^2 dx .
\end{equation}
This concludes the proof (\ref{E:PREENERGYINEQUALITYFINITEPROP}).
\end{proof}
Next we prove the local curl estimates required to control the curl quantities on the right hand side of (\ref{E:PREENERGYINEQUALITYFINITEPROP}). Before proving the result, note that we say $C(s')$ is integrable in $s'$ if $\int_0^\infty C(s') ds' \lesssim 1$ and we say $C(s',\tau,|y|)$ is integrable in $s'$ if $\int_0^\infty C(s',\tau,|y|) ds' \lesssim C(\tau,|y|).$
\begin{proposition}\label{P:LOCALCURLESTIMATES}
Let $(\uptheta, {\bf V}):\Omega \rightarrow \mathbb R^3\times \mathbb R^3$ be a unique local solution to (\ref{E:THETAEQNLINEARENERGYFUNCTION})-(\ref{E:THETAICGAMMALEQ5OVER3}) on $[0,T^*]$ for $T^*>0$ fixed with $\text{supp} \,\uptheta_0 \subseteq B_1(\mathbf{0})$, $\text{supp}\,\mathbf{V}_0 \subseteq B_1(\mathbf{0})$ and assume $(\uptheta, {\bf V})$ satisfies the a priori assumptions (\ref{E:APRIORI}). Fix $N\geq 4$.  Suppose $\beta$ in (\ref{E:THETAEQNLINEARENERGYFUNCTION}) satisfies $\| \beta \|^2_{H^{N+1}(\mathbb{R}^3)} \leq \lambda$ and $\text{supp} \, \beta \subseteq B_1(\mathbf{0})$ where  $\lambda > 0$ is fixed.  Fix $y \in \mathbb{R}^3$ and $\tau \in [0,T^*]$. Take $K>0$ from (\ref{P:LOCALENERGYESTIMATE}). Let $s \in [0,\tau]$. Let $s \in [0,s']$. Then
\begin{align}
b_{\mathbf{V}} (s') & \lesssim  C(s') b_{\mathbf{V}}(0) + \overline{C}(s',\tau,|y|)e(0) + C(\tau,|y|) \sup_{0 \leq s'' \leq s'} e(s'') + C(s',\tau,|y|) \int_{U(s')} | D \beta |^2 dx, \label{E:CURLVFINITEPROPINEQ} \\
b_{\uptheta} (s) &\lesssim b_{\uptheta}(0)+b_{\mathbf{V}}(0) + C(\tau,|y|)e(0)+C(\tau,|y|) \int_0^s\sup_{0 \leq s'' \leq s'} e(s'') ds'+C(\tau,|y|) \int_{U(0)} | D \beta |^2 dx, \label{E:CURLTHETAFINITEPROPINEQ}
\end{align}
where $C(s')$ and $C(s',\tau,|y|)$ are integrable in $s'$.
\end{proposition}
\begin{proof}
\textit{Proof of} (\ref{E:CURLVFINITEPROPINEQ}). From our velocity curl equation derivation Lemma \ref{L:CURLEQUATIONDERIVATION} (\ref{E:CURLVFINAL}) we have the desirable form for our local curl  estimate
\begin{align}\label{E:CURLVLOCAL}
&\text{Curl}_{\Lambda\mathscr{A}}{\bf V}  = \frac{1}{1+\alpha}\Lambda\mathscr{A} x \times \mathbf{V} + \frac{\alpha}{(1+\alpha)(1+\beta)} \Lambda\mathscr{A} \nabla \beta \times \mathbf{V} \notag \\
&+  \frac{\mu(0) \text{Curl}_{\Lambda \mathscr{A}} ({\bf V}(0))}{\mu} - \frac{\mu(0) \Lambda\mathscr{A} x \times \mathbf{V}(0)}{(1+\alpha) \mu} - \frac{\alpha \mu(0) \Lambda\mathscr{A} \nabla \beta \times \mathbf{V}(0)}{(1+\alpha)(1+\beta) \mu} \notag \\
& + \frac{1}{\mu}\int_0^{s'} \mu [\partial_{\tau}, \text{Curl}_{\Lambda\mathscr{A}}] {\bf V} ds'' - \frac{1}{(1+\alpha) \mu} \int_0^{s'} \mu [\partial_{\tau},\Lambda \mathscr{A} x \times ] \mathbf{V} d s'' \notag \\
&- \frac{\alpha}{(1+\alpha)(1+\beta) \mu} \int_0^{s'} \mu [\partial_{\tau},\Lambda \mathscr{A} \nabla \beta \times ] \mathbf{V} d s'' \notag \\
& - \frac{2}{\mu} \int_0^{s'} \mu \, \text{Curl}_{\Lambda\mathscr{A}}(\Gamma^\ast{\bf V}) ds'' + \frac{2}{(1+\alpha) \mu} \int_0^{s'} \mu \, \Lambda\mathscr{A} x \times (\Gamma^\ast \mathbf{V}) d s'' \notag \\
&+ \frac{2 \alpha}{(1+\alpha)(1+\beta) \mu} \int_0^{s'} \mu \, \Lambda\mathscr{A} \nabla \beta \times (\Gamma^\ast \mathbf{V}) d s'' \notag \\
&+\frac{\overline{C}}{(1+\alpha) \mu}\int_0^{s'} \mu^{1-\delta-\sigma} \Lambda x \times  \Lambda \uptheta d s'' -\frac{\overline{C}}{(1+\alpha) \mu} \int_0^{s'} \mu^{1-\delta-\sigma} \Lambda \mathscr{A}[D \uptheta]x \times \Lambda \eta d s'' \notag \\
&+\frac{\alpha \overline{C}}{(1+\alpha)(1+\beta) \mu}\int_0^{s'} \mu^{1-\delta-\sigma} \Lambda \mathscr{A} \nabla \beta \times \Lambda \uptheta d s'' -\frac{\alpha \overline{C}}{(1+\alpha)(1+\beta) \mu}\int_0^{s'} \mu^{1-\delta-\sigma} \Lambda \mathscr{A} \nabla \beta \times \Lambda x d s''
\end{align}
For the purpose of our local curl estimates define the $\| \cdot \|^2_{(w (1+\beta),L^2(U(s))}$ norm as follows
\begin{equation}
\| f \|^2_{w(1+\beta),L^2(U(s))}:=\int_{U(s)} \left(w (1+\beta) |f|^2 \right)(s,x) \, dx
\end{equation}
We take the $ \| \cdot \|^2_{w(1+\beta),L^2(U(s'))}$ norm, where $s' \in [0,s]$ , of (\ref{E:CURLVLOCAL}), multiply by $[\mu(s')]^{-\delta}$ and estimate the right hand side. As can be seen from (\ref{E:CURLVLOCAL}), many terms and hence estimates are similar. Therefore we give the key estimates below and remark similar arguments will hold for the other terms.

For the first term on the right hand side of (\ref{E:CURLVLOCAL}), $\frac{1}{1+\alpha}\Lambda\mathscr{A} x   \times \mathbf{V}$, first recall $y \in \mathbb{R}^3$ and $\tau \in [0,T]$ are fixed. Then for $x \in U(s')$ where $s' \in [0,s]$, first note
\begin{equation}\label{E:XCONTROL}
|x|^2 \lesssim |x-y|^2 + |y|^2 \leq K^2(\tau-s')^2+|y|^2 \leq K^2 \tau^2 + |y|^2.
\end{equation}
Then using the boundedness of $\beta$
\begin{equation}
[\mu(s')]^{-\delta} \| \frac{1}{1+\alpha}\Lambda\mathscr{A} x   \times \mathbf{V} \|^2_{w(1+\beta),L^2(U(s'))} \lesssim C(\tau,|y|) e(s').
\end{equation}
For the sixth term on the right hand side of (\ref{E:CURLVLOCAL}), $- \frac{1}{(1+\alpha) \mu} \int_0^{s'} \mu [\partial_{\tau},\Lambda \mathscr{A} x \times ] \mathbf{V} d s''$, 
\begin{equation}
[\partial_{\tau}, \text{Curl}_{\Lambda\mathscr{A}}] {\bf V}_j^k  =\partial_{\tau} \left(\Lambda_{jm}\mathscr{A}^s_m\right) {\bf V},_s^k - \partial_{\tau} \left(\Lambda_{km}\mathscr{A}^s_m\right) {\bf V},_s^j.
\end{equation}
As the other term can be estimated in the same way, we restrict our focus to the first term only
\begin{equation}\label{E:CURLERROREXPANSION}
\partial_{\tau} \left(\Lambda_{jm}\mathscr{A}^s_m\right) {\bf V},_s^k = \left(\partial_{\tau} \Lambda_{jm} \mathscr{A}^s_m + \Lambda_{jm}\partial_{\tau} \mathscr{A}^s_m \right){\bf V},_s^k = \partial_{\tau} \Lambda_{jm} \mathscr{A}^s_m {\bf V},_s^k +  \Lambda_{jm} \partial_{\tau} \mathscr{A}^s_m {\bf V},_s^k.
\end{equation}
Schematically consider the right hand side of (\ref{E:CURLERROREXPANSION})
\begin{equation}\label{E:CURLERROR1FINITEPROP}
\partial_{\tau} \Lambda  D\eta \,  D{\bf V} + \Lambda D{\bf V} \,  D{\bf V}
\end{equation}
For the first term in (\ref{E:CURLERROR1FINITEPROP}), one must first integrate by parts in $s''$
\begin{align}
\frac1\mu\int_0^{s'} \mu \partial_{\tau} \Lambda  D\eta \,  D{\bf V} d s'' = & \frac1\mu\left(\mu \partial_{\tau} \Lambda  D\eta \, D \uptheta\right)\big|^{s'}_0 \label{tau2}\\ 
&  - \frac1\mu\int_0^{s'} \mu(\frac{\mu_\tau}\mu\partial_{\tau} \Lambda D\eta+\partial_{\tau \tau} \Lambda  D\eta+\partial_{\tau} \Lambda D{\bf V}) \,  D \uptheta\, ds''. \notag
\end{align}
Then
\begin{equation}
\| \frac1\mu\left(\mu \partial_{\tau} \Lambda  D\eta \, D \uptheta\right)\big|^{s'}_0 \|^2_{w(1+\beta),L^2(U(s'))} \lesssim e(s')+ C(s')e(0),
\end{equation}
where we include $1+\beta$ in our energy here, and $C(s')$ is integrable in $s'$. Next
\begin{align}
& \| \frac1\mu \int_0^{s'} \mu_{\tau} \partial_{\tau} \Lambda D\eta D \uptheta d s'' \|^2_{w(1+\beta),L^2(U(s'))} \notag \\
& \lesssim \frac{1}{[\mu(s')]^2} \sup_{0 \leq s'' \leq s'}  \int_{U(s')} w(1+\beta) |D \uptheta|^2 \left| \int_0^{s'} \mu_{\tau} \partial_{\tau} \Lambda D \eta d s'' \right|^2 dx \notag \\
& \lesssim \frac{1}{[\mu(s')]^2} \sup_{0 \leq s'' \leq s'} \int_{U(s')} w (1+\beta) |D \uptheta |^2 dx \left| \int_0^{s'} \mu_{\tau} d s'' \right|^2 \notag \\
& \lesssim \sup_{0 \leq s'' \leq s'} \int_{U(s')} w(1+\beta) |D \uptheta |^2 dx \notag \\
& \lesssim \sup_{0 \leq s'' \leq s'} \int_{U(s'')} w(1+\beta) |D \uptheta |^2 dx \notag \\
& \lesssim \sup_{0 \leq s'' \leq s'} e(s'').
\end{align}
Similar arguments apply to the remaining terms in (\ref{tau2}). For the second term in (\ref{E:CURLERROR1FINITEPROP}), we also first integrate by parts in $s''$
\begin{align}
\frac1\mu\int_0^{s'} \mu \Lambda D{\bf V} \,  D{\bf V} d s'' = & \frac1\mu\left(\mu  \Lambda  D{\bf V} \, D \uptheta\right)\big|^{s'}_0 \label{tau3}\\ 
&  - \frac1\mu\int_0^{s'} \mu(\frac{\mu_{\tau}}\mu \Lambda D{\bf V}+\partial_{\tau} \Lambda  D{\bf V}+ \Lambda D{\bf V}_{\tau}) \,  D \uptheta\, ds''. \notag 
\end{align}
Then, noting our a priori assumptions (\ref{E:APRIORI}) used to control $D{\bf V}_{\tau}$, a similar argument to the first term in (\ref{E:CURLERROR1FINITEPROP}) give the same bounds here. Therefore 
\begin{equation}
[\mu(s')]^{-\delta} \|\frac{1}{(1+\alpha) \mu} \int_0^{s'} \mu [\partial_{\tau},\Lambda \mathscr{A} x \times ] \mathbf{V} d s'' \|^2_{w(1+\beta),L^2(U(s'))} \lesssim \sup_{0 \leq s'' \leq s'} e(s'')+C(s')e(0).
\end{equation}
For the last term on the right hand side of (\ref{E:CURLVLOCAL}), $-\frac{\alpha \overline{C}}{(1+\alpha)(1+\beta) \mu}\int_0^{s'} \mu^{1-\delta-\sigma} \Lambda \mathscr{A} \nabla \beta \times \Lambda x d s''$, first note using Lemma \ref{L:USEFULTAULEMMAGAMMALEQ5OVER3}
\begin{align}\label{E:INTEGRABLEINSPRIME}
\frac{1}{\mu^2} | \int_0^{s'} \mu^{1-\delta-\sigma} d s''|^2 &\lesssim e^{-2 \mu_1 s'} \Big| \int_0^{s'} e^{(1-\delta-\sigma)\mu_1 s''} d s'' \Big|^2 \notag \\
&\lesssim e^{-2 \mu_1 s' }(e^{2(1-\delta-\sigma)\mu_1 s'}+1) \notag \\
& \lesssim e^{-2 \mu_0 s'},
\end{align}
which is integrable in $s'$. Then with this and using that $U(s') \subset U(s)$ for $s \leq s'$ as well as the boundedness of $\frac{1}{1+\beta}$ and (\ref{E:XCONTROL})
 \begin{align}
&[\mu(s')]^{-\delta} \| \frac{\alpha \overline{C}}{(1+\alpha)(1+\beta) \mu}\int_0^{s'} \mu^{1-\delta-\sigma} \Lambda \mathscr{A} \nabla \beta \times \Lambda x d s'' \|^2_{w(1+\beta),L^2(U(s'))} \notag \\
&\lesssim C(s',\tau,|y|) \int_{U(s')} | D \beta |^2 dx \notag \\
&\lesssim  C(s',\tau,|y|) \int_{U(0)} | D \beta |^2 dx.
\end{align}
\textit{Proof of} (\ref{E:CURLTHETAFINITEPROPINEQ}). From our perturbation curl equation derivation Lemma \ref{L:CURLEQUATIONDERIVATION} (\ref{E:CURLTHETAFINAL}) we have the desirable form for our local curl estimate
\begin{align}\label{E:CURLTHETALOCAL}
&\text{ Curl}_{\Lambda\mathscr{A}}{\bf \uptheta}  = \text{ Curl}_{\Lambda \mathscr{A}}([\uptheta(0)]) \notag \\
&+\mu(0) \text{ Curl}_{\Lambda \mathscr{A}} ({\bf V}(0)) \int_0^s \frac{1}{\mu(s')} d s' - \frac{\mu(0) \Lambda\mathscr{A} x \times \mathbf{V}(0)}{1+\alpha} \int_0^s \frac{1}{\mu(s')} d s' \notag \\
&- \frac{\alpha \mu(0) \Lambda\mathscr{A} \nabla \beta \times \mathbf{V}(0)}{(1+\alpha)(1+\beta)}\int_0^{s} \frac{1}{\mu(s')} d s' + \int_0^s  [\partial_{\tau}, \text{ Curl}_{\Lambda\mathscr{A}}] {\uptheta} ds'  \notag \\
&+ \frac{1}{1+\alpha} \int_0^s \Lambda\mathscr{A} x \times \mathbf{V} ds' + \frac{\alpha}{(1+\alpha)(1+\beta)} \int_0^s \Lambda\mathscr{A} \nabla \beta \times \mathbf{V} ds' \notag \\
&+\int_0^s \frac{1}{\mu(s')} \int_0^{s'} \mu(s'') [\partial_{\tau}, \text{ Curl}_{\Lambda\mathscr{A}}] {\bf V} ds'' d s' -\frac{1}{1+\alpha} \int_0^{s} \frac{1}{\mu(s')} \int_0^{s'} \mu(s '') [\partial_{\tau},\Lambda \mathscr{A} x \times ] \mathbf{V} ds '' d s ' \notag \\
&- \frac{\alpha}{(1+\alpha)(1+\beta)} \int_0^{s} \frac{1}{\mu(s')} \int_0^{s'} \mu(s '') [\partial_{\tau},\Lambda \mathscr{A} \nabla \beta \times ] \mathbf{V} d s'' \notag \\
& - \int_0^s \frac{2}{\mu(s')} \int_0^{s '} \mu(s'') \, \text{ Curl}_{\Lambda\mathscr{A}}(\Gamma^\ast{\bf V}) ds'' d s' + \frac{2}{1+\alpha} \int_0^{s} \frac{1}{\mu(s')} \int_0^{s '} \mu(s'') \, \Lambda\mathscr{A} x \times (\Gamma^\ast \mathbf{V}) ds '' d s ' \notag \\
&+ \frac{2 \alpha}{(1+\alpha)(1+\beta)} \int_0^{s} \frac{1}{\mu(s')} \int_0^{s '} \mu(s'') \, \Lambda\mathscr{A} \nabla \beta \times (\Gamma^\ast \mathbf{V}) ds '' d s ' \notag \\
& +\frac{\overline{C}}{1+\alpha} \int_0^s \frac{1}{\mu(s')} \int_0^{s'} \mu(s'')^{1-\delta-\sigma}  \Lambda x \times  \Lambda \uptheta d s '' d s ' \notag \\
& -\frac{\overline{C}}{1+\alpha} \int_0^s \frac{1}{\mu(s')} \int_0^{s'} \mu(s '')^{1-\delta-\sigma} \Lambda \mathscr{A}[D \uptheta]x \times \Lambda \eta d s '' d s ' \notag \\
&+\frac{\alpha \overline{C}}{(1+\alpha)(1+\beta)} \int_0^s \frac{1}{\mu} \int_0^{s'} \mu^{1-\delta-\sigma} \Lambda \mathscr{A} \nabla \beta \times \Lambda \uptheta d s'' ds' \notag \\
&-\frac{\alpha \overline{C}}{(1+\alpha)(1+\beta)} \int_0^{s} \frac{1}{\mu} \int_0^{s'} \mu^{1-\delta-\sigma} \Lambda \mathscr{A} \nabla \beta \times \Lambda x  d s'' ds'.
\end{align}
We take the $ \| \cdot \|^2_{w(1+\beta),L^2(U(s))}$ norm of (\ref{E:CURLTHETALOCAL}), multiply by $[\mu(s)]^{-\delta}$ and estimate the right hand side. As can be seen from (\ref{E:CURLTHETALOCAL}), many terms and hence estimates are similar. Therefore we give the key estimates below and remark similar arguments will hold for the other terms.

For the fifth term on the right hand side of (\ref{E:CURLTHETALOCAL}), $\int_0^s  [\partial_{\tau}, \text{ Curl}_{\Lambda\mathscr{A}}] {\uptheta} ds'$, noting $U(s) \subset U(s')$ for $s' \leq s$,
\begin{align}
& [\mu(s)]^{-\delta}  \| \int_0^s  [\partial_{\tau}, \text{ Curl}_{\Lambda\mathscr{A}}] {\uptheta} ds' \|^2_{w(1+\beta),L^2(U(s))} \notag \\
&\lesssim [\mu(s)]^{-\delta} \int_{U(s)} w(1+\beta) \left| \int_0^s  [\partial_{\tau}, \text{ Curl}_{\Lambda\mathscr{A}}] {\uptheta} ds' \right|^2 dx \notag \\
&\lesssim [\mu(s)]^{-\delta} \int_{U(s)} w(1+\beta) (\int_0^s 1 ds')(\int_0^s \left| [\partial_{\tau}, \text{ Curl}_{\Lambda\mathscr{A}}] {\uptheta} \right|^2 ds') dx \notag ]\\
&\lesssim \int_0^s \int_{U(s)} w(1+\beta) \left| [\partial_{\tau}, \text{ Curl}_{\Lambda\mathscr{A}}] {\uptheta}  \right|^2 dx ds' \notag \\
&\lesssim \int_0^s \int_{U(s')} w(1+\beta) \left| [\partial_{\tau}, \text{ Curl}_{\Lambda\mathscr{A}}] {\uptheta}  \right|^2 dx ds' \notag \\
&\lesssim \int_0^s e(s') ds'.
\end{align}
For the eighth term on the right hand side of (\ref{E:CURLTHETALOCAL}), $\int_0^s \frac{1}{\mu(s')} \int_0^{s'} \mu(s'') [\partial_{\tau}, \text{ Curl}_{\Lambda\mathscr{A}}] {\bf V} ds'' d s'$, first note
\begin{align}
& \| \int_0^s \frac{1}{\mu(s')} \int_0^{s'} \mu(s'') [\partial_{\tau}, \text{ Curl}_{\Lambda\mathscr{A}}] {\bf V} ds'' d s' \|^2_{w(1+\beta),L^2(U(s))} \notag  \\
&= \int_{U(s)}  w (1+\beta) \left| \int_0^s \frac{1}{[\mu(s')]^{1/2}} \frac{1}{[\mu(s')]^{1/2}}  \int_0^{s'} \mu(s'') [\partial_{\tau}, \text{ Curl}_{\Lambda\mathscr{A}}] {\bf V}  ds'' d s' \right|^2 dx \notag \\
&\lesssim \int_{U(s)} w (1+\beta) [ (\int_0^s \frac{1}{\mu(s')} d s') \int_0^s \frac{1}{\mu(s')} (\int_0^{s'} \mu(s'') [\partial_{\tau}, \text{ Curl}_{\Lambda\mathscr{A}}] {\bf V}  ds'')^2 d s'] dx \notag \\
& \lesssim \int_0^s \int_{U(s)} w (1+\beta) \frac{1}{\mu(s')} (\int_0^{s'} \mu(s'') [\partial_{\tau}, \text{ Curl}_{\Lambda\mathscr{A}}] {\bf V}  ds'')^2  dx ds' \notag \\
& \lesssim \int_0^s \int_{U(s')} w (1+\beta) \frac{1}{\mu(s')}  (\int_0^{s'} \mu(s'') [\partial_{\tau}, \text{ Curl}_{\Lambda\mathscr{A}}] {\bf V}  ds'')^2  dx ds'
\end{align}
where we note $U(s) \subset U(s')$ for $s' \leq s$. Hence applying the argument above used for the sixth term on the right hand side of (\ref{E:CURLVLOCAL})
\begin{align}
& [\mu(s)]^{-\delta} \| \int_0^s \frac{1}{\mu(s')} \int_0^{s'} \mu(s'') [\partial_{\tau}, \text{ Curl}_{\Lambda\mathscr{A}}] {\bf V} ds'' d s' \|^2_{w(1+\beta),L^2(U(s))} \notag \\
& \lesssim \int_0^s C(s') \sup_{0 \leq s'' \leq s'}e(s'')+C(s')(s')e(0) ds' \notag \\
&\lesssim e(0)+\int_0^s  \sup_{0 \leq s'' \leq s'} e(s'') ds' .
\end{align}
Finally, using (\ref{E:INTEGRABLEINSPRIME}), for the last term on the right hand side of (\ref{E:CURLTHETALOCAL}) we have
\begin{align}
&[\mu(s)]^{-\delta} \| \frac{\alpha \overline{C}}{(1+\alpha)(1+\beta)} \int_0^{s} \frac{1}{\mu} \int_0^{s'} \mu^{1-\delta-\sigma} \Lambda \mathscr{A} \nabla \beta \times \Lambda x  d s'' ds' \|^2_{w(1+\beta),L^2(U(s))} \notag \\
& \lesssim C(\tau,|y|) \int_{U(s)} | D \beta |^2 dx \notag \\
& \lesssim C(\tau,|y|)  \int_{U(0)} | D \beta |^2 dx.
\end{align} 
This concludes the proof of our local curl estimate Proposition \ref{P:LOCALCURLESTIMATES}. \end{proof}

We are now ready to prove our Finite Propagation Theorem \ref{T:FINITEPROPAGATIONTHEOREM}.
\begin{proof}[Proof of Theorem \ref{T:FINITEPROPAGATIONTHEOREM}.] Combining the results of the energy estimate and curl estimate Propositions above, namely the inequalities  (\ref{E:PREENERGYINEQUALITYFINITEPROP}), (\ref{E:CURLVFINITEPROPINEQ}) and (\ref{E:CURLTHETAFINITEPROPINEQ}), we have established the following energy inequality
\begin{align}
& e(s) \lesssim C(\tau,|y|)e(0)+b_{\uptheta}(0) + b_{\mathbf{V}}(0)+ \int_{U(0)}  |\beta|^2 dx \notag \\
& \qquad \qquad + C(\tau,|y|) \int_{U(0)} |D\beta|^2 dx  + C(\tau,|y|) \int_0^s \sup_{0 \leq s'' \leq s'} \{ e(s'') \}  ds'. \label{E:ENERGYINEQUALITYFINITEPROP1}
\end{align}
Then taking the supremum over $0 \leq s' \leq s$,
\begin{align}
& \sup_{0 \leq s' \leq s} \{ e(s') \}  \lesssim C(\tau,|y|)e(0)+b_{\uptheta}(0) + b_{\mathbf{V}}(0)+ \int_{U(0)}  |\beta|^2 dx  \notag \\ 
& \qquad \qquad + C(\tau,|y|) \int_{U(0)} |D\beta|^2 dx  + C(\tau,|y|) \int_0^s \sup_{0 \leq s'' \leq s'} \{ e(s'') \}  ds'. \label{E:PREENERGYINEQUALITYFINITEPROP2}
\end{align}
Since $C(\tau,|y|)e(0)+b_{\uptheta}(0) + b_{\mathbf{V}}(0)+ \int_{U(0)}  |\beta|^2 + |D\beta|^2 dx$ is constant in $s$, applying Gronwall's inequality, we have for some $C^*>0$,
\begin{align}
0 \leq \sup_{0 \leq s' \leq s} \{ e(s') \}) &\leq C^* \left[C(\tau,|y|)e(0)+b_{\uptheta}(0) + b_{\mathbf{V}}(0)+  \int_{U(0)}  |\beta|^2 dx \right. \notag \\
&\qquad \left.+ C(\tau,|y|) \int_{U(0)} |D\beta|^2 dx \right] \exp\left[\int_0^s C^* ds' \right].
\end{align}
Hence
\begin{align}
0 \leq \sup_{0 \leq s' \leq s} \{ e(s') \}) & \leq C^* \left[C(\tau,|y|)e(0)+b_{\uptheta}(0) + b_{\mathbf{V}}(0)+  \int_{U(0)}  |\beta|^2 dx \right. \notag \\
&\qquad \left. + C(\tau,|y|) \int_{U(0)} |D\beta|^2 dx \right] e^{C^* s}.
\end{align}
If $|y| > 1+K\tau$, then $|x| > 1$ for $x \in U(0)$. Thus 
\begin{equation}
e(0)=b_{\uptheta}(0)=b_{\mathbf{V}}(0)=0 \text { since } \text{supp} \,\uptheta_0 \subseteq B_1(\mathbf{0}), \text{supp}\,\mathbf{V}_0 \subseteq B_1(\mathbf{0}).
\end{equation}
Furthermore
\begin{equation}
\int_{U(0)}  |\beta|^2 dx = \int_{U(0)} |D\beta|^2 dx=0 \text{ since } \text{supp}\,\beta \subseteq B_1(\mathbf{0}).
\end{equation}
Hence $e(s)=0$ for $|y|>1+K\tau$. Thus $\uptheta(s,y)=\mathbf{V}(s,y)=0$ for $|y| > 1 +K \tau$. As $s \in [0,\tau]$ is arbitrary,
\begin{equation}
(\uptheta(\tau,y),\mathbf{V}(\tau,y))=(0,0),
\end{equation}
for $\tau \in [0,T]$ and $|y| > 1+K\tau$.
\end{proof}

\section{High Order Energy Estimates}\label{S:ENERGYESTIMATES}
With our finite propagation theorem in hand, we are now ready to proceed with high order energy estimates. To this end, we first introduce the two energy based high order quantities which arise directly from the problem. Denoting the usual dot product on $\mathbb{R}^3$ by $\langle \cdot , \cdot \rangle$, introduce the high-order energy functional
\begin{align}
& \mathcal{E}^N(\tau)= \frac12 \sum_{|\nu|\le N}
\int_{\mathbb{R}^3} \Big[\mu \left\langle \Lambda^{-1} \partial^\nu\mathbf{V},\,\partial^\nu \mathbf{V}\right\rangle +\overline{C} \mu^{-\delta} \left\langle\Lambda^{-1}\partial^\nu\mathbf{\uptheta},\,\partial^\nu\mathbf{\uptheta}\right\rangle\Big] 
\notag \\
& \ \ \ \  \ \ \ \ \ \ \ \ \ \ \ \ \ \ \ \ \ \ \ \ +\overline{C} \mu^{-\delta}(1+\beta)\mathscr{J}^{-\frac1\alpha}\Big[ \sum_{i,j=1}^3 d_id_j^{-1}([\mathscr{N}_{\nu}]^j_{i})^2 + \tfrac1\alpha\left(\text{div}_\eta\partial^\nu \uptheta\right)^2\Big]  \,dy, \label{E:EDEFGAMMAGREATER5OVER3}
\end{align}
and the dissipation functional
\begin{align}
&\mathcal{D}^N (\mathbf{\uptheta})=\overline{C} \frac{\delta}{2}\mu^{-\delta}\frac{\mu_\tau}{\mu} \int_{\mathbb{R}^3} \sum_{|\nu|\le N} |\partial^\nu \uptheta|^2 + (1+\beta)\mathscr{J}^{-\frac1\alpha} \Big[ \sum_{i,j=1}^3 d_id_j^{-1}([\mathscr{N}_{\nu}]^j_{i})^2 + \tfrac1\alpha\left(\text{div}_\eta\partial^\nu \uptheta\right)^2\Big]  \Big)\,dy.
\end{align}
We also introduce a similar term to $\mathcal{E}^N$ which does not include top order quantities but will be controlled through our coercivity Lemma \ref{L:COERCIVITY}
\begin{align}
& \mathcal{C}^{N-1}(\tau) = \frac12\sum_{|\nu|\le N-1}     
\int_\Omega \, \overline{C} \left\langle\Lambda^{-1}\partial^\nu \uptheta,\,\partial^\nu \uptheta \right\rangle \notag \\
& \ \ \ \  \ \ \ \ \ \ \ \ \ \ \ \ \ \ \ \ \ \ \ \ +\overline{C}(1+\beta)\mathscr{J}^{-\frac1\alpha}\Big[ \sum_{i,j=1}^3 d_id_j^{-1}\left((\mathscr{N}_{\nu})^j_{i}\right)^2 + \tfrac1\alpha\left(\text{div}_\eta\partial^\nu \uptheta\right)^2\Big] \,dy. \label{E:CDEFGAMMAGREATER5OVER3}
\end{align}

\begin{remark}
Our coercivity Lemma \ref{L:COERCIVITY} given in Appendix \ref{A:COERCIVITY} allows us to control terms without time weights with negative powers using our norm $\mathcal{S}^N(\tau)$ and initial data $\mathcal{S}^N(0)$. Such terms do not appear immediately from our equation (\ref{E:THETAEQNLINEARENERGYFUNCTION}). Notably, Lemma \ref{L:COERCIVITY} will let us include $\mathcal{C}^{N-1}$ in our energy inequality. However we cannot use Lemma \ref{L:COERCIVITY} to control top order quantities since that would require control of $N+1$ derivatives of $\mathbf{V}$ which we do not have. This is why we have the particular structure of $\mathcal{S}^N$ where we separate top order terms from lower order terms, and secondly why we only include $N-1$ derivatives in $\mathcal{C}^{N-1}$. 
\end{remark}

Finally, before we prove our main energy inequality, we give the norm-energy equivalence as follows.
\begin{lemma}\label{L:NORMENERGYGAMMAGREATER5OVER3} 
Let $(\uptheta, {\bf V}):\Omega \rightarrow \mathbb R^3\times \mathbb R^3$ be a unique local solution to (\ref{E:THETAEQNLINEARENERGYFUNCTION})-(\ref{E:THETAICGAMMALEQ5OVER3}) on $[0,T^*]$ for $T^*>0$ fixed with $\text{supp} \,\uptheta_0 \subseteq B_1(\mathbf{0})$, $\text{supp}\,\mathbf{V}_0 \subseteq B_1(\mathbf{0})$ and assume $(\uptheta, {\bf V})$ satisfies the a priori assumptions (\ref{E:APRIORI}). Fix $N\geq 4$.  Suppose $\beta$ in (\ref{E:THETAEQNLINEARENERGYFUNCTION}) satisfies $\| \beta \|^2_{H^{N+1}(\mathbb{R}^3)} \leq \lambda$ and $\text{supp} \, \beta \subseteq B_1(\mathbf{0})$ where  $\lambda > 0$ is fixed.  Then there are constants $C_1,C_2>0$ so that
\begin{align}\label{E:NORMENERGYEQUIVALENCE}
C_1\mathcal{S}^N (\tau) \le \sup_{0\le\tau'\le\tau}\{\mathcal{E}^N(\tau')+\mathcal{C}^{N-1}(\tau')\} \le C_2(\mathcal{S}^N(\tau)+\mathcal{S}^N(0)).
\end{align} 
\end{lemma} 
\begin{proof}
Recall the definition $\mathcal{S}^N$ (\ref{E:SNNORMGAMMALEQ5OVER3}), the decomposition of $\Lambda$ (\ref{E:LAMBDADECOMP}) and the definition of the conjugate $\mathscr{N}_{\nu}$ (\ref{E:NNU}). Then (\ref{E:NORMENERGYEQUIVALENCE}) is a straightforward application of Lemma \ref{L:USEFULTAULEMMAGAMMALEQ5OVER3} to give bounds on $\Lambda$ and associated matrix quantities and the boundedness of $\beta$ given by (\ref{E:BETADEMAND}), in conjunction with Lemma \ref{L:COERCIVITY} to control terms without time weights with negative powers, which are included in $\mathcal{C}^{N-1}(\tau')$, by $\mathcal{S}^N(\tau)+\mathcal{S}^N(0)$.
\end{proof}
We are now ready to prove our central high order energy inequality which will be essential in the proof of our main result Theorem \ref{T:MAINTHEOREMGAMMALEQ5OVER3}. 
\begin{proposition}\label{P:ENERGYESTIMATE}
Let $(\uptheta, {\bf V}):\Omega \rightarrow \mathbb R^3\times \mathbb R^3$ be a unique local solution to (\ref{E:THETAEQNLINEARENERGYFUNCTION})-(\ref{E:THETAICGAMMALEQ5OVER3}) on $[0,T^*]$ for $T^*>0$ fixed with $\text{supp} \,\uptheta_0 \subseteq B_1(\mathbf{0})$, $\text{supp}\,\mathbf{V}_0 \subseteq B_1(\mathbf{0})$ and assume $(\uptheta, {\bf V})$ satisfies the a priori assumptions (\ref{E:APRIORI}). Fix $N\geq 4$.  Suppose $\beta$ in (\ref{E:THETAEQNLINEARENERGYFUNCTION}) satisfies $\| \beta \|^2_{H^{N+1}(\mathbb{R}^3)} \leq \lambda$ and $\text{supp} \, \beta \subseteq B_1(\mathbf{0})$ where  $\lambda > 0$ is fixed.  Then for all $\tau \in [0,T^*]$, we have the following inequality for some $0<\kappa\ll 1$,
\begin{align}
& \mathcal E^N(\tau)+\mathcal{C}^{N-1}(\tau)+\int_0^\tau \mathcal D^N(\tau')\,d\tau' \lesssim  \mathcal{S}^N(0) +  \lambda  + \mathcal{B}^N[\uptheta](\tau)  + \kappa\mathcal{S}^N(\tau) \notag \\ 
& \qquad \qquad + \int_0^\tau  (1+\tau')(\mathcal{S}^N(\tau'))^{\frac12} (\mathcal{B}^N[{\bf V}](\tau'))^{\frac12}\,d\tau'  + \int_0^\tau (1+\tau')e^{-\mu_0\tau'} \mathcal{S}^N(\tau') d\tau'. \label{E:ENERGYMAINGAMMAGREATER5OVER3}
\end{align}
\end{proposition}
\begin{proof}
First write (\ref{E:THETAEQNLINEARENERGYFUNCTION}) as follows
\begin{align}\label{E:THETAEQNEXPANDPRESSUREZERO}
&\mu^{\sigma} \partial_{\tau \tau} \uptheta_i + \mu_{\tau} \mu^{-1+\sigma} \partial_\tau \uptheta_i + 2 \mu^{\sigma} \Gamma^*_{ij} \partial_{\tau} \uptheta_j + \overline{C} \mu^{-\delta} \Lambda_{i \ell} \uptheta_\ell + \overline{C} \mu^{-\delta} ( \Lambda_{ij}(1+\beta) ( \mathscr{A}_j^k \mathscr{J}^{-\frac{1}{\alpha}} - \delta_j^k))_{,k} \notag \\
&-\overline{C} \mu^{-\delta} y_k (1+\beta) \Lambda_{ij} (\mathscr{A}_j^k \mathscr{J}^{-\frac{1}{\alpha}} - \delta_j^k)+\overline{C} \mu^{-\delta} \Lambda_{ik} \beta_{,k}  -\overline{C} \mu^{-\delta} \beta \Lambda_{ik} y_k =0.
\end{align}
\textbf{Zeroth order estimate}
Multiply (\ref{E:THETAEQNEXPANDPRESSUREZERO}) by $\Lambda^{-1}_{im} \partial_\tau \uptheta^m$ and integrate over $\int_0^\tau \int_{\mathbb{R}^3} \, dy d \tau'$,
\begin{align}\label{E:THETAEQNPOSTINNERPRODUCTZERO}
&\int_0^\tau \int_{\mathbb{R}^3} \left( \mu^{\sigma} \partial_{\tau \tau} \uptheta_i + \mu_{\tau} \mu^{-1+\sigma} \partial_\tau \uptheta_i + 2 \mu^{\sigma} \Gamma^*_{ij} \partial_{\tau} \uptheta_j + \overline{C} \mu^{-\delta} \Lambda_{i \ell} \uptheta_\ell \right)  \Lambda^{-1}_{im} \partial_\tau \uptheta^m \, dy d \tau' \notag \\
&+\int_0^\tau \int_{\mathbb{R}^3} \overline{C} \mu^{-\delta} ( \Lambda_{ij} (1+\beta) ( \mathscr{A}_j^k \mathscr{J}^{-\frac{1}{\alpha}} - \delta_j^k))_{,k} \Lambda^{-1}_{im} \partial_\tau \uptheta^m \, dy d \tau' \notag \\
&- \int_0^\tau \int_{\mathbb{R}^3} \overline{C} \mu^{-\delta} y_k (1+\beta) \Lambda_{ij} (\mathscr{A}_j^k \mathscr{J}^{-\frac{1}{\alpha}} - \delta_j^k) \Lambda^{-1}_{im} \partial_\tau \uptheta^m dy d \tau' \notag \\
&+ \int_0^\tau \int_{\mathbb{R}^3} \left( \overline{C} \mu^{-\delta} \Lambda_{ik} \beta_{,k}  -\overline{C} \mu^{-\delta} \beta \Lambda_{ik} y_k \right) \Lambda^{-1}_{im} \partial_\tau \uptheta^m dy d \tau' =0.
\end{align}
The zeroth order estimate for the first and second integrals in (\ref{E:THETAEQNPOSTINNERPRODUCTZERO}) follows by a similar argument to the finite propagation energy estimate argument without boundary contributions and we do not isolate a divergence contribution: instead we control $1-\mathscr{J}^{-\frac1\alpha}$ using (\ref{E:JIDENTITYENERGY}). For the third integral in (\ref{E:THETAEQNPOSTINNERPRODUCTZERO}) apply (\ref{E:AJIDENTITYENERGY})
\begin{align}
&-\int_0^\tau \int_{\mathbb{R}^3} \overline{C} \mu^{-\delta} y_k (1+\beta) \Lambda_{ij} (\mathscr{A}_j^k \mathscr{J}^{-\frac{1}{\alpha}} - \delta_j^k) \Lambda^{-1}_{im} \partial_\tau \uptheta^m \, dy d \tau' \notag \\
&=\int_0^\tau \int_{\mathbb{R}^3} \overline{C} \mu^{-\delta} y_k (1+\beta) \Lambda_{ij} \mathscr{A}^k_\ell \uptheta^\ell,_j \mathscr{J}^{-\frac{1}{\alpha}} \Lambda^{-1}_{im} \partial_\tau \uptheta^m \, dy d \tau'  \notag \\
&+\int_0^\tau \int_{\mathbb{R}^3} \overline{C} \mu^{-\delta}  (1+\beta) \Lambda_{ik} y_k (\frac{1}{\alpha}\text{Tr}[D \uptheta ] + O(|D\uptheta|^2)) \Lambda^{-1}_{im} \partial_\tau \uptheta^m \, dy d \tau' \label{E:THIRDINTEGRALZEROORDEREXPAND}
\end{align}
Applying the Finite Propagation Theorem \ref{T:FINITEPROPAGATIONTHEOREM} to control $y$, and the boundedness of $\beta$, the first integral on the right hand side of (\ref{E:THIRDINTEGRALZEROORDEREXPAND}) is bounded by
\begin{align}
&\int_0^\tau \int_{\mathbb{R}^3} ((1+K \tau') \mu^{-\delta/2}) \mu^{-\delta/2} \uptheta^i,_j \partial_\tau \uptheta^m \, dy d \tau' \notag \\
& \lesssim \int_0^\tau  \left(\int_{\mathbb{R}^3} \mu^{-\delta} |\uptheta^i,_j|^2 \, dy \right)^{1/2} \left( \int_{\mathbb{R}^3} |\partial_\tau \uptheta^m|^2 \, dy\right)^{1/2} d \tau' \notag \\
& \lesssim \int_0^\tau e^{-\mu_0 \tau'} \mathcal{S}^N d \tau'
\end{align}
A similar argument holds for the second integral on the right hand side of the first term on the right hand side of (\ref{E:THIRDINTEGRALZEROORDEREXPAND}). For the fourth integral in \ref{E:THETAEQNPOSTINNERPRODUCTZERO})
\begin{align}
&\int_0^\tau \int_{\mathbb{R}^3} \left( \overline{C} \mu^{-\delta} \Lambda_{ik} \beta_{,k}  - \overline{C} \mu^{-\delta} \beta \Lambda_{ik} y_k \right) \Lambda^{-1}_{im} \partial_\tau \uptheta^m dy d \tau'  \notag \\
&=  \int_0^\tau \int_{\mathbb{R}^3} \overline{C} \mu^{-\delta} \Lambda_{ik} \beta_{,k}  \Lambda^{-1}_{im} \partial_\tau \uptheta^m dy d \tau' - \int_0^\tau \int_{\mathbb{R}^3} \overline{C} \mu^{-\delta}  \beta \Lambda_{ik} y_k \Lambda^{-1}_{im} \partial_\tau \uptheta^m dy d \tau'. \label{E:FOURTHINTEGRALZEROORDEREXPAND}
\end{align}
Using the compact support of $\beta$ to bound $y_k$ and then the smallness of $\| \beta\|_{L^2(\mathbb{R}^3)}$, the second integral on the right hand side of (\ref{E:THIRDINTEGRALZEROORDEREXPAND}) is bounded by
\begin{align}
&\int_0^\tau \int_{\mathbb{R}^3} \overline{C}  \beta \partial_\tau \uptheta^m dy d \tau'  \\
& \lesssim \int_0^\tau \left(\int_{\mathbb{R}^3} \beta^2 dx \right)^{1/2} \left(\int_{\mathbb{R}^3} |\partial_\tau \uptheta^m|^2 dx \right)^{1/2}  \notag \\
& \lesssim \int_0^\tau \lambda^{1/2} e^{-\mu_0 \tau'} (\mathcal{S}^N)^{1/2} d \tau' \notag \\
& \lesssim \lambda \int_0^\tau e^{-\mu_0 \tau'} d \tau' + \int_0^\tau e^{-\mu_0 \tau'} \mathcal{S}^N d \tau' \notag \\
& \lesssim \lambda + \int_0^\tau e^{-\mu_0 \tau'} \mathcal{S}^N d \tau'.
\end{align}
A similar argument directly using the smallness of $\| D \beta\|_{L^2(\mathbb{R}^3)}$ gives that the third integral on the right hand side (\ref{E:THIRDINTEGRALZEROORDEREXPAND}) is bounded by
\begin{equation}
\lambda + \int_0^\tau e^{-\mu_0 \tau'} \mathcal{S}^N d \tau'.
\end{equation}
To complete the zeroth order estimate, we obtain the full expression for $\mathcal{E}^N$ in (\ref{E:ENERGYMAINGAMMAGREATER5OVER3}) by adding the following formula 
 \begin{equation}\label{E:ZEROORDERFULLENGAMMALEQ5OVER3}
\int_0^\tau \frac{1}{2}\frac{d}{d\tau} \frac1\alpha \int_{\mathbb{R}^3}  \mathscr{J}^{-\frac1\alpha} |\text{div}_\eta \uptheta |^2 dy d\tau' =  \int_0^\tau \frac{1}{2\alpha} \int_{\mathbb{R}^3}  \left( \partial_\tau (\mathscr{J}^{-\frac1\alpha} )|\text{div}_\eta \uptheta |^2
 + \mathscr{J}^{-\frac1\alpha} \partial_\tau(|\text{div}_\eta \uptheta |^2) \right) dy d \tau',
\end{equation}
with the right-hand side in turn bounded by $e^{-\mu_0 \tau} \mathcal{S}^N(\tau)$ by the $H^2 (\mathbb{R}^3) \hookrightarrow L^\infty (\mathbb{R}^3)$ embedding. Then the left-hand side of (\ref{E:ZEROORDERFULLENGAMMALEQ5OVER3}) completes $\mathcal{E}^N$ in (\ref{E:ENERGYMAINGAMMAGREATER5OVER3}) and contributes to $\mathcal{S}^N(0)$ by the fundamental theorem of calculus.
\\ \\
\textbf{High order estimates}
Fix $\nu \in \mathbb{Z}^3_{\geq 0}$ with $|\nu| \geq 1$. Apply $\partial^\nu$ to (\ref{E:THETAEQNEXPANDPRESSUREZERO})
\begin{align}\label{E:DERIVATIVETHETAEQNEXPANDPRESSURE}
&\mu^{\sigma} \partial_{\tau \tau} \partial^\nu \uptheta_i + \mu_{\tau} \mu^{-1+\sigma} \partial^\nu \partial_\tau \uptheta_i + 2 \mu^{\sigma} \Gamma^*_{ij} \partial_{\tau} \partial^\nu \uptheta_j + \overline{C} \mu^{-\delta} \Lambda_{i \ell} \partial^\nu \uptheta_\ell \notag \\
&+ \overline{C} \mu^{-\delta} \Lambda_{ij} ( \partial^\nu[(1+\beta)( \mathscr{A}_j^k \mathscr{J}^{-\frac{1}{\alpha}} - \delta_j^k )])_{,k} -\overline{C} \mu^{-\delta} \Lambda_{ij} \partial^\nu [y_k (1+\beta) (\mathscr{A}_j^k \mathscr{J}^{-\frac{1}{\alpha}} - \delta_j^k)] \notag \\
&+\overline{C} \mu^{-\delta} \Lambda_{ik} (\partial^\nu \beta)_{,k}  - \overline{C} \mu^{-\delta} \Lambda_{ik} \partial^\nu( y_k \beta ) =0.
\end{align}
Multiply (\ref{E:DERIVATIVETHETAEQNEXPANDPRESSURE}) by $\Lambda^{-1}_{im}\partial_\tau \partial^\nu \uptheta^m$ and integrate over $\mathbb{R}^3$ and then from $0$ to $\tau$,
\begin{align}
&\int_0^\tau \frac12\frac{d}{d\tau}\left(\mu^\sigma \int_{\mathbb{R}^3} \langle\Lambda^{-1}\partial^\nu {\bf V},\,\partial^\nu {\bf V}\rangle \,dy+\overline{C} \mu^{-\delta} \int_{\mathbb{R}^3} \, |\partial^\nu \uptheta |^2 
\,dy \right) d \tau' \notag \\
& +\int_0^\tau \frac{\overline{C}}{2}\mu^{-\delta} \frac{\mu_\tau}{\mu}\int_{\mathbb{R}^3} |\partial^\nu \uptheta |^2 dy d \tau' - \int_0^\tau \frac{\mu^{\sigma}}{2}\int_{\mathbb{R}^3}
\, \left\langle\left[\partial_\tau\Lambda^{-1}-4 \Lambda^{-1}\Gamma^*\right] \partial^\nu {\bf V},\partial^\nu {\bf V}\right\rangle dy d \tau' \notag \\
& + \overline{C} \int_0^\tau \mu^{-\delta} \int_{\mathbb{R}^3} \Lambda_{ij} ( \partial^\nu[(1+\beta)( \mathscr{A}_j^k \mathscr{J}^{-\frac{1}{\alpha}} - \delta_j^k )])_{,k} \Lambda^{-1}_{im} \partial_\tau \partial^\nu \uptheta^m \,dy d \tau' \notag \\
&- \overline{C} \int_0^\tau \mu^{-\delta} \int_{\mathbb{R}^3} \Lambda_{ij} \partial^\nu [y_k (1+\beta) (\mathscr{A}_j^k \mathscr{J}^{-\frac{1}{\alpha}} - \delta_j^k)] \Lambda^{-1}_{im} \partial^\nu \uptheta^m_\tau \,dy d \tau' \notag \\
&+\overline{C} \int_0^\tau \mu^{-\delta} \int_{\mathbb{R}^3} \Lambda_{ik} (\partial^\nu \beta)_{,k} \Lambda^{-1}_{im} \partial^\nu \uptheta^m_\tau \, dy d\tau' - \overline{C} \int_0^\tau \mu^{-\delta}  \int_{\mathbb{R}^3} \Lambda_{ik} \partial^\nu( y_k \beta ) \Lambda^{-1}_{im} \partial^\nu \uptheta^m_\tau \, dy d\tau'=0. \label{E:HIGHORDERPOSTIPGAMMAGREATER5OVER3}
\end{align}
By the same reasoning as the zeroth order case, the first four integrals on the left hand side of (\ref{E:HIGHORDERPOSTIPGAMMAGREATER5OVER3}) contribute to the energy inequality (\ref{E:ENERGYMAINGAMMAGREATER5OVER3}) in the same way. For the fourth to last integral on the left hand side of (\ref{E:HIGHORDERPOSTIPGAMMAGREATER5OVER3}), first note via the Leibniz rule
\begin{equation}
\partial^\nu [(1+\beta)(\mathscr{A}_j^k \mathscr{J}^{-\frac{1}{\alpha}} - \delta_j^k)]=(1+\beta) \partial^{\nu} (\mathscr{A}_j^k \mathscr{J}^{-\frac{1}{\alpha}}) + \sum_{0 \leq |\nu'| \leq |\nu|-1} c_{\nu'} \partial^{\nu-\nu'} (\beta) \partial^{\nu'} (\mathscr{A}_j^k \mathscr{J}^{-\frac{1}{\alpha}}).
\end{equation}
Then the fourth to last integral on the left hand side of (\ref{E:HIGHORDERPOSTIPGAMMAGREATER5OVER3}) is
\begin{align}\label{E:FOURTHTOLASTINTEGRALEXPAND}
&\overline{C} \int_0^\tau \mu^{-\delta} \int_{\mathbb{R}^3} \Lambda_{ij} ( \partial^\nu[(1+\beta)( \mathscr{A}_j^k \mathscr{J}^{-\frac{1}{\alpha}} - \delta_j^k )])_{,k} \Lambda^{-1}_{im} \partial_\tau \partial^\nu \uptheta^m \,dy d \tau' \notag \\
&= \overline{C} \int_0^\tau \mu^{-\delta} \int_{\mathbb{R}^3} \Lambda_{ij} [(1+\beta) \partial^{\nu} (\mathscr{A}_j^k \mathscr{J}^{-\frac{1}{\alpha}})]_{,k} \Lambda^{-1}_{im} \partial_\tau \partial^\nu \uptheta^m \,dy d \tau' \notag \\
&+  \overline{C} \int_0^\tau \mu^{-\delta} \int_{\mathbb{R}^3} \Lambda_{ij} \left(\sum_{0 \leq |\nu'| \leq |\nu|-1} c_{\nu'} \partial^{\nu-\nu'} (\beta) \partial^{\nu'} (\mathscr{A}_j^k \mathscr{J}^{-\frac{1}{\alpha}})\right)_{,k} \Lambda^{-1}_{im} \partial_\tau \partial^\nu \uptheta^m \,dy d \tau'. 
\end{align}
For the first integral on the right hand side of (\ref{E:FOURTHTOLASTINTEGRALEXPAND}), first note using our differentiation formulae for $\mathscr{A}$ and $\mathscr{J}$ (\ref{E:AJDIFFERENTIATIONFORMULAE})
\begin{equation}\label{E:DERIVATIVEAJFORMULA}
\partial_{y_i} (\mathscr{A}^k_j\mathscr{J}^{-\frac1\alpha}) = - \mathscr{J}^{-\frac1\alpha}\mathscr{A}^k_\ell\mathscr{A}^s_j ( \partial_{y_i} \uptheta^\ell),_s  -\tfrac1\alpha \mathscr{J}^{-\frac1\alpha}\mathscr{A}^k_j\mathscr{A}^s_\ell (\partial_{y_i} \uptheta^\ell),_s.
\end{equation}
Then using this and applying integration by parts which holds due to finite propagation, the first integral on the right hand side of (\ref{E:FOURTHTOLASTINTEGRALEXPAND}) is
\begin{align}
& \overline{C} \int_0^\tau \mu^{-\delta} \int_{\mathbb{R}^3}  \Lambda_{ij} [ (1+\beta) \partial^\nu ( \mathscr{A}^k_j\mathscr{J}^{-\frac1\alpha} ) ],_k \Lambda^{-1}_{im} \partial_\tau \partial^\nu \uptheta^m \,dy d\tau' \notag \\
& = \overline{C} \int_0^\tau \mu^{-\delta} \int_{\mathbb{R}^3}  \Lambda_{ij} [(1+\beta)(- \mathscr{J}^{-\frac1\alpha}\mathscr{A}^k_\ell\mathscr{A}^s_j \partial^\nu \uptheta^\ell,_s -\tfrac1\alpha \mathscr{J}^{-\frac1\alpha}\mathscr{A}^k_j\mathscr{A}^s_\ell \partial^\nu \uptheta^\ell,_s \notag \\
& \qquad \qquad \qquad \qquad \qquad + \mathcal{C}^{\nu,k}_j(\uptheta))],_k \Lambda^{-1}_{im} \partial_\tau \partial^\nu \uptheta^m \,dy d\tau' \notag \\
&= -\overline{C} \int_0^\tau \mu^{-\delta} \int_{\mathbb{R}^3}  [ (1+\beta)(\mathscr{J}^{-\frac1\alpha}\mathscr{A}^k_\ell (\Lambda_{ij}\mathscr{A}^s_j \partial^\nu \uptheta^\ell,_s - \Lambda_{\ell j}\mathscr{A}^s_j \partial^\nu \uptheta^i,_s) + \mathscr{J}^{-\frac1\alpha}\mathscr{A}^k_\ell\Lambda_{\ell j}\mathscr{A}^s_j \partial^\nu \uptheta^i,_s  \notag \\
&+ \tfrac1\alpha\mathscr{J}^{-\frac1\alpha}\Lambda_{ij}\mathscr{A}^k_j\mathscr{A}^s_\ell  \partial^\nu \uptheta^\ell,_s )],_k \Lambda^{-1}_{im} \partial_\tau \partial^\nu \uptheta^m \,dy d \tau' +\overline{C} \int_0^\tau \mu^{-\delta} \int_{\mathbb{R}^3}  \Lambda_{ij} [(1 +\beta) \mathcal{C}^{\nu,k}_j ],_k \Lambda^{-1}_{im}\partial_\tau \partial^\nu \uptheta^m \,dy d \tau' \notag \\
&=\overline{C} \int_0^\tau \mu^{-\delta} \int_{\mathbb{R}^3}  (1+\beta)\mathscr{J}^{-\frac1\alpha} (\mathscr{A}^k_\ell  [\text{Curl}_{\Lambda\mathscr{A}}\partial^\nu\uptheta]^\ell_i + \mathscr{A}^k_\ell \Lambda_{\ell j}[\nabla_\eta \partial^\nu\uptheta]^i_j  \notag \\
&  \qquad \qquad \qquad \qquad \qquad + \tfrac1\alpha \Lambda_{ij}\mathscr{A}^k_j \text{div}_\eta (\partial^\nu \uptheta)) \Lambda^{-1}_{im}\partial_\tau (\partial^\nu \uptheta^m),_k \,dy d \tau' \notag \\
&+\overline{C} \int_0^\tau  \mu^{-\delta} \int_{\mathbb{R}^3}  \Lambda_{ij}  [(1+\beta) \mathcal{C}^{\nu,k}_j ],_k \Lambda^{-1}_{im} \partial_\tau \partial^\nu \uptheta^m \,dy d \tau' \notag \\
&=:I_1+I_2, \label{E:HIGHORDERLASTINTEGRALIBPGAMMAGREATER5OVER3}
\end{align}
where we define $\mathcal{C}^{\nu,k}_j$ to be the lower order terms resulting from differentiating the result of (\ref{E:DERIVATIVEAJFORMULA}). Note
\begin{equation}
\| \mathcal{C}^{\nu,k}_j \|^2 \lesssim \mathcal{S}^N, \quad \mu^{-\delta} \| D (\mathcal{C}^{\nu,k}_j) \|^2 \lesssim \mathcal{S}^N.
\end{equation}
Then using the boundedness of $\beta$ and $D \beta$
\begin{equation}
|I_2| \lesssim \int_0^\tau \| \partial_\tau \partial^\nu \uptheta \|  \| \mathcal{C}^{\nu,k}_j \| + \| \partial_\tau \partial^\nu \uptheta \| \mu^{-\delta} \| D (\mathcal{C}^{\nu,k}_j) \| d \tau' \lesssim \int_0^\tau e^{-\mu_0 \tau'} \mathcal{S}^N d \tau'.
\end{equation}
For $I_1$ rewrite the gradient and divergence terms, and integrate by parts in $\tau$ the curl term as follows
\begin{align}
&I_1 =  \int_0^\tau \mu^{-\delta} \int_{\mathbb{R}^3} (1+\beta) \mathscr{J}^{-\frac1\alpha} \mathscr{A}^k_\ell [\text{Curl}_{\Lambda\mathscr{A}}\partial^\nu\uptheta]^\ell_i \Lambda^{-1}_{im}\partial_\tau \left(\partial^\nu \uptheta^m\right),_k \,dy d\tau'  \notag \\
&+ \int_0^\tau \mu^{-\delta} \int_{\mathbb{R}^3} (1+\beta)  \mathscr{J}^{-\frac1\alpha} \left( \mathscr{A}^k_\ell \Lambda_{\ell j}[\nabla_\eta \partial^\nu\uptheta]^i_j + \tfrac1\alpha \Lambda_{ij}\mathscr{A}^k_j \text{div}_\eta (\partial^\nu \uptheta) \right) \Lambda^{-1}_{im}\partial_\tau \left(\partial^\nu \uptheta^m\right),_k \,dy d\tau'  \notag \\
&= \int_0^\tau \mu^{-\delta} \int_{\mathbb{R}^3}  (1+\beta)  \mathscr{J}^{-\frac1\alpha} \mathscr{A}^k_\ell \Lambda^{-1}_{im} [\text{Curl}_{\Lambda\mathscr{A}}\partial^\nu\uptheta]^\ell_i \partial_\tau \left(\partial^\nu \uptheta^m\right),_k \,dy d\tau'  \notag \\
&+ \int_0^\tau  \mu^{-\delta} \int_{\mathbb{R}^3} (1+\beta) \mathscr{J}^{-\frac1\alpha} \left( \Lambda_{\ell j}[\nabla_\eta \partial^\nu \uptheta]^i_j \Lambda^{-1}_{im}[\nabla_\eta\partial_\tau \partial^\nu \uptheta ]_\ell^m+\tfrac1\alpha \text{div}_\eta (\partial^\nu\uptheta) \text{div}_\eta (\partial_\tau \partial^\nu \uptheta) \right) \,dy d\tau' \notag \\ 
& =  \mu^{-\delta} \int_{\mathbb{R}^3} (1+\beta) \mathscr{J}^{-\frac1\alpha} \mathscr{A}^k_\ell\Lambda^{-1}_{im} [\text{Curl}_{\Lambda\mathscr{A}}\partial^\nu\uptheta]^\ell_i \left(\partial^\nu{\bf\uptheta}^m\right),_k \,dy \Big|^\tau_0 \notag \\
&  -\int_0^\tau \mu^{-\delta} \int_{\mathbb{R}^3} (1+\beta)  \mathscr{J}^{-\frac1\alpha} \mathscr{A}^k_\ell\Lambda^{-1}_{im}
[\text{Curl}_{\Lambda\mathscr{A}}\partial^\nu{\bf V}]^\ell_i \left(\partial^\nu{\bf\uptheta}^m\right),_k \,dy\,d\tau' \notag\\
&+\delta \int_0^\tau \mu^{-\delta-1} \mu_{\tau}  \int_{\mathbb{R}^3}  (1+\beta) \mathscr{J}^{-\frac1\alpha} \mathscr{A}^k_\ell\Lambda^{-1}_{im}
[\text{Curl}_{\Lambda\mathscr{A}}\partial^\nu{\bf \uptheta}]^\ell_i \left(\partial^\nu{\bf\uptheta}^m\right),_k \,dy\,d\tau' \notag \\
&  -\int_0^\tau \mu^{-\delta} \int_{\mathbb{R}^3} (1+\beta)  \partial_\tau\left(\mathscr{J}^{\frac1\alpha} \mathscr{A}^k_\ell\Lambda^{-1}_{im}\right)
[\text{Curl}_{\Lambda\mathscr{A}}\partial^\nu{\bf \uptheta}]^\ell_i \left(\partial^\nu{\bf\uptheta}^m\right),_k \,dy\,d\tau'  \notag \\
&  -\int_0^\tau \mu^{-\delta} \int_{\mathbb{R}^3} (1+\beta)  \mathscr{J}^{-\frac1\alpha} \mathscr{A}^k_\ell\Lambda^{-1}_{im}
[\text{Curl}_{\Lambda_\tau\mathscr{A}}\partial^\nu{\bf \uptheta}]^\ell_i \left(\partial^\nu{\bf\uptheta}^m\right),_k \,dy \,d\tau' \notag \\
&- \int_0^\tau \mu^{-\delta} \int_{\mathbb{R}^3}  (1+\beta) \mathscr{J}^{-\frac1\alpha} \mathscr{A}^k_\ell\Lambda^{-1}_{im}
[\text{Curl}_{\Lambda\mathscr{A}_\tau}\partial^\nu{\bf \uptheta}]^\ell_i \left(\partial^\nu{\bf\uptheta}^m\right),_k \,dy\,d\tau' \notag \\
&+\int_0^\tau \mu^{-\delta} \int_{\mathbb{R}^3}  (1+\beta) \mathscr{J}^{-\frac1\alpha} \left( \Lambda_{\ell j}[\nabla_\eta \partial^\nu \uptheta]^i_j \Lambda^{-1}_{im}\partial_\tau [\nabla_\eta \partial^\nu \uptheta ]_\ell^m+\tfrac1\alpha\text{div}_\eta (\partial^\nu \uptheta) \partial_\tau \text{div}_\eta ( \partial^\nu \uptheta) \right) dy d\tau' \notag \\
&-\int_0^\tau \mu^{-\delta} \int_{\mathbb{R}^3}(1+\beta) \mathscr{J}^{-\frac1\alpha} \left( \Lambda_{\ell j}[\nabla_\eta \partial^\nu \uptheta]^i_j \Lambda^{-1}_{im}\partial_\tau\mathscr{A}^k_\ell \left(\partial^\nu \uptheta^m\right),_k+ \tfrac1\alpha\text{div}_\eta (\partial^\nu\uptheta)  \partial_\tau \mathscr{A}^k_j \left(\partial^\nu \uptheta^j\right),_k  \right) dy d\tau' \\
&:=B_1+B_2+B_3+\mathcal{R}_1+\mathcal{R}_2+\mathcal{R}_3+E_1+\mathcal{R}_4. \label{E:HIGHORDERI1EXPANDGAMMAGREATER5OVER3} 
\end{align}
For $B_1$, for $0 < \kappa \ll 1$,
\begin{equation}
B_1 \lesssim \mathcal{S}^N(0)+\kappa \mathcal{S}^N(\tau) + \mathcal{B}^N[\uptheta](\tau),
\end{equation}
where we have introduced $\kappa$ through the Young inequality.
For $B_2$,
\begin{equation}
|B_2| \lesssim \int_0^\tau (\mathcal{S}^N)^{1/2} (\mathcal{B}^N[\mathbf V])^{1/2} d \tau'.
\end{equation}
For $B_3$, first apply the Fundamental Theorem of Calculus as follows
\begin{align}
&\text{Curl}_{\Lambda\mathscr{A}}\partial^\nu \uptheta (\tau',y)=\int_0^{\tau'} \partial_\tau \text{Curl}_{\Lambda\mathscr{A}}\partial^\nu \uptheta (\tau'',y) d \tau''+\text{Curl}_{\Lambda\mathscr{A}}\partial^\nu \uptheta (0,y) \notag \\
&=\int_0^{\tau'} \text{Curl}_{\Lambda\mathscr{A}} \partial^\nu \mathbf{V} (\tau'',y) d \tau''+\int_0^{\tau'} \text{Curl}_{\partial_{\tau}(\Lambda\mathscr{A})} \partial^\nu \uptheta (\tau'',y) d \tau''+\text{Curl}_{\Lambda\mathscr{A}}\partial^\nu \uptheta (0,y) \notag \\
&\lesssim \tau' \sup_{0 \leq \tau'' \leq \tau'} \{ \text{Curl}_{\Lambda\mathscr{A}} \partial^\nu \mathbf{V} \} + \tau' \sup_{0 \leq \tau'' \leq \tau'} \{ \text{Curl}_{\partial_{\tau}(\Lambda\mathscr{A})} \partial^\nu \uptheta \} + \text{Curl}_{\Lambda\mathscr{A}}\partial^\nu \uptheta (0,y).
\end{align}
Then using the boundedness of $\frac{\mu_\tau}{\mu}$
\begin{align}
|B_3| & \lesssim \int_0^\tau \mu^{-\delta} \int_{\mathbb{R}^3} ( \tau'  \sup_{0 \leq \tau'' \leq \tau'} \{ | [ \text{Curl}_{\Lambda\mathscr{A}} \partial^\nu \mathbf{V} ]^\ell_i | \} + \tau' \sup_{0 \leq \tau'' \leq \tau'} \{ | [\text{Curl}_{\partial_{\tau}(\Lambda\mathscr{A})} \partial^\nu \uptheta]^\ell_i | \}  \notag\\
&\qquad \qquad + | [\text{Curl}_{\Lambda\mathscr{A}}\partial^\nu \uptheta (0)]^\ell_i | ) \left(\partial^\nu{\bf\uptheta}^m\right),_k \,dy\,d\tau' \notag \\
& \lesssim \int_0^\tau \tau' (\mathcal{S}^N)^{1/2} (\mathcal{B}^N[\mathbf V])^{1/2} + \tau' e^{-\mu_0 \tau'} \mathcal{S}^N d \tau'  \notag \\
&+ \left( \kappa \sup_{0 \leq \tau \leq \tau'} \mu^{-\delta} \|\nabla_{\eta} \partial^\nu \uptheta \|^2 + \frac{1}{\kappa} \| \text{Curl}_{\Lambda\mathscr{A}}\partial^\nu \uptheta (0) \|^2 \right) \int_0^\tau \mu^{-\delta} d \tau' \notag \\
& \lesssim  \int_0^\tau \tau' (\mathcal{S}^N)^{1/2} (\mathcal{B}^N[\mathbf V])^{1/2} + \tau' e^{-\mu_0 \tau'} \mathcal{S}^N d \tau' + \kappa \mathcal{S}^N + \mathcal{S}^N(0).
\end{align}
For $E_1$
\begin{align}
E_1&=\frac12 \int_0^\tau \frac{d}{d\tau}\left\{\mu^{-\delta} \int_{\mathbb{R}^3}  (1+\beta) \mathscr{J}^{-\frac1\alpha}\left(\sum_{i,j=1}^3 d_id_j^{-1}(\mathscr N_\nu)^j_{i})^2 + \frac1\alpha\left(\text{div}_\eta \partial^\nu \uptheta\right)^2\right)\,dy\right\}\,d\tau' \notag\\
&+\int_0^\tau \frac{\delta}{2}\mu^{-\delta}\frac{\mu_\tau}{\mu} \int_{\mathbb{R}^3}   (1+\beta) \mathscr{J}^{-\frac1\alpha}\Big[\sum_{i,j=1}^3 d_id_j^{-1}\left((\mathscr{N}_{\nu})^j_{i}\right)^2+\tfrac1\alpha\left(\text{div}_\eta X_r^a\slashed\partial^\beta \uptheta\right)^2\Big] w^{a+\alpha+1}e^{\bar{S}} \, dy \, d\tau' \notag \\
&+ \int_0^\tau \mu^{-\delta} \int_{\mathbb{R}^3}  (1+\beta) \mathscr{J}^{-\frac1\alpha} \mathcal T_{\nu} dy \,d\tau' \notag \\
& - \frac12 \int_0^\tau \mu^{-\delta} \int_{\mathbb{R}^3} (1+\beta) \partial_\tau\left(\mathscr{J}^{-\frac1\alpha}\right)\left(\sum_{i,j=1}^3 d_id_j^{-1}(\mathscr N_{\nu})^j_{i})^2 + \frac1\alpha\left(\text{div}_\eta \partial^\nu \uptheta\right)^2\right)\,dy \,d\tau' \notag \\
&:=E_2+\mathcal{D}_1 + \mathcal{R}_5 + \mathcal{R}_6, \label{E:HIGHORDERE1EXPANDGAMMAGREATER5OVER3}
\end{align}
where we have used Lemma \ref{L:KEYLEMMA}. Then $E_2$ contributes to $\mathcal{E}^N(\tau)$ in (\ref{E:ENERGYMAINGAMMAGREATER5OVER3}), and also to $\mathcal{S}^N(0)$ by the fundamental theorem of calculus. Also $\mathcal{D}_1$ contributes to $\int_0^\tau \mathcal{D}^N (\tau ') \, d\tau'$ in (\ref{E:ENERGYMAINGAMMAGREATER5OVER3}). Now
\begin{equation}
| \mathcal{R}_i | \lesssim \int_0^\tau e^{-\mu_0 \tau'} \mathcal{S}^N d \tau',
\end{equation}
for $i=1,2,3,4,5,6$.

For the second integral on the right hand side of (\ref{E:FOURTHTOLASTINTEGRALEXPAND}) first note
\begin{align}\label{E:EXPANDLEIBNIZRULEENERGYESTIMATE}
&\left(\sum_{0 \leq |\nu'| \leq |\nu|-1} c_{\nu'} \partial^{\nu-\nu'} (\beta) \partial^{\nu'} (\mathscr{A}_j^k \mathscr{J}^{-\frac{1}{\alpha}})\right)_{,k} = \partial^{\nu} (\beta_{,k}) \mathscr{A}_j^k \mathscr{J}^{-\frac{1}{\alpha}} + \partial^{\nu} (\beta) (\mathscr{A}_j^k \mathscr{J}^{-\frac{1}{\alpha}})_{,k} \notag \\
&+ \left(\sum_{1 \leq |\nu'| \leq |\nu|-1} c_{\nu'} \partial^{\nu-\nu'} (\beta) \partial^{\nu'} (\mathscr{A}_j^k \mathscr{J}^{-\frac{1}{\alpha}})\right)_{,k}.
\end{align}
For the integral resulting from the first term on the right hand side of (\ref{E:EXPANDLEIBNIZRULEENERGYESTIMATE})
\begin{align}\label{E:SOURCETERMHIGHORDERENERGYESTIMATE}
&\left| \overline{C} \int_0^\tau \mu^{-\delta} \int_{\mathbb{R}^3} \Lambda_{ij} \partial^{\nu} (\beta_{,k}) \mathscr{A}_j^k \mathscr{J}^{-\frac{1}{\alpha}} \Lambda^{-1}_{im} \partial_\tau \partial^\nu \uptheta^m \,dy d \tau' \right| \lesssim \int_0^\tau \int_{\mathbb{R}^3}  (\partial^\nu \beta)_{,k} \partial_\tau \partial^\nu \uptheta^m dy d \tau' \notag \\
& \lesssim \int_0^\tau \left(\int_{\mathbb{R}^3} (\partial^\nu \beta)_{,k}^2 dx \right)^{1/2} \left(\int_{\mathbb{R}^3} |\partial^\nu \uptheta^m_\tau |^2 dx \right)^{1/2}  \notag \\
& \lesssim \int_0^\tau \lambda^{1/2} e^{-\mu_0 \tau'} (\mathcal{S}^N)^{1/2} d \tau' \notag \\
& \lesssim \lambda  \int_0^\tau e^{-\mu_0 \tau'} d \tau' + \int_0^\tau e^{-\mu_0 \tau'} \mathcal{S}^N d \tau' \notag \\
& \lesssim \lambda + \int_0^\tau e^{-\mu_0 \tau'} \mathcal{S}^N d \tau'.
\end{align}
For the remaining integrals from the last two terms on the right hand side of (\ref{E:EXPANDLEIBNIZRULEENERGYESTIMATE}) with both the $H^2 (\mathbb{R}^3) \hookrightarrow L^\infty (\mathbb{R}^3)$ embedding and boundedness of $\| \beta \|_{H^{N+1}(\mathbb{R}^3)}$, the resulting integrals are bounded by
\begin{equation}
\int_0^\tau e^{-\mu_0 \tau'} \mathcal{S}^N d \tau'.
\end{equation} 
For the third to last integral on the left hand side of (\ref{E:HIGHORDERPOSTIPGAMMAGREATER5OVER3}), first note:
\begin{align}
\partial^\nu [y_k(1+\beta)(\mathscr{A}_j^k \mathscr{J}^{-\frac{1}{\alpha}} - \delta_j^k)]&= y_k \partial^{\nu} [(1+\beta)(\mathscr{A}_j^k \mathscr{J}^{-\frac{1}{\alpha}} - \delta_j^k)]  \notag \\
&+ \sum_{|\nu'|=|\nu|-1} c_{\nu',k} \partial^{\nu'} [(1+\beta)(\mathscr{A}_j^k \mathscr{J}^{-\frac{1}{\alpha}} - \delta_j^k) ],
\end{align}
where $c_{\nu',k}$ are non-negative constants depending on $\nu',k$.  Then the third to last integral on the left hand side of (\ref{E:HIGHORDERPOSTIPGAMMAGREATER5OVER3}) is
\begin{align}\label{E:THIRDTOLASTINTEGRALPOSTDERIVATIVE}
&- \overline{C} \int_0^\tau \mu^{-\delta} \int_{\mathbb{R}^3} \Lambda_{ij} \partial^\nu [y_k (1+\beta) (\mathscr{A}_j^k \mathscr{J}^{-\frac{1}{\alpha}} - \delta_j^k)] \Lambda^{-1}_{im} \partial^\nu \uptheta^m_\tau \,dy d \tau' \notag \\
&= - \overline{C} \int_0^\tau \mu^{-\delta} \int_{\mathbb{R}^3} \Lambda_{ij} y_k \partial^{\nu} [(1+\beta)(\mathscr{A}_j^k \mathscr{J}^{-\frac{1}{\alpha}} - \delta_j^k)] \Lambda^{-1}_{im} \partial^\nu \uptheta^m_\tau \,dy d \tau' \notag \\
& - \overline{C} \int_0^\tau \mu^{-\delta} \int_{\mathbb{R}^3} \Lambda_{ij} \sum_{|\nu'|=|\nu|-1} c_{\nu',k} \partial^{\nu'} [(1+\beta)(\mathscr{A}_j^k \mathscr{J}^{-\frac{1}{\alpha}} - \delta_j^k) ]] \Lambda^{-1}_{im} \partial^\nu \uptheta^m_\tau \,dy d \tau'
\end{align}
For the first integral on the right hand side of (\ref{E:THIRDTOLASTINTEGRALPOSTDERIVATIVE}), use finite propagation and then a combination of the $H^2 (\mathbb{R}^3) \hookrightarrow L^\infty (\mathbb{R}^3)$ embedding, boundedness of $\| \beta \|_{H^{N+1}(\mathbb{R}^3)}$ and (\ref{E:EXPANDAJMINUSID}) to estimate
\begin{align}\label{E:FINITEPROPAGATIONHIGHORDER}
& \left| -\overline{C} \int_0^\tau \mu^{-\delta} \int_{\mathbb{R}^3} \Lambda_{ij} y_k \partial^{\nu} [(1+\beta)(\mathscr{A}_j^k \mathscr{J}^{-\frac{1}{\alpha}} - \delta_j^k)] \Lambda^{-1}_{im} \partial^\nu \uptheta^m_\tau \,dy d \tau' \right| \notag \\
& \leq \int_0^\tau \int_{\mathbb{R}^3} \overline{C} \frac{1+K \tau}{\mu^{\delta}} \Lambda_{ij} \partial^{\nu} [(1+\beta)(\mathscr{A}_j^k \mathscr{J}^{-\frac{1}{\alpha}} - \delta_j^k)] \Lambda^{-1}_{im} \partial^\nu \uptheta^m_\tau \, dy d \tau' \notag \\
& \lesssim \int_0^\tau e^{-\mu_0 \tau'} \mathcal{S}^N d \tau'.
\end{align}
Using a combination of the $H^2 (\mathbb{R}^3) \hookrightarrow L^\infty (\mathbb{R}^3)$ embedding, boundedness of $\| \beta \|_{H^{N+1,\infty}(\mathbb{R}^3)}$ and (\ref{E:EXPANDAJMINUSID}), the second integral on the right hand side of (\ref{E:THIRDTOLASTINTEGRALPOSTDERIVATIVE}) is straightforward to bound by 
\begin{equation}
\int_0^\tau e^{-\mu_0 \tau'} \mathcal{S}^N d \tau'.
\end{equation}
For the second to last integral on the right hand side (\ref{E:HIGHORDERPOSTIPGAMMAGREATER5OVER3}), a similar argument to (\ref{E:SOURCETERMHIGHORDERENERGYESTIMATE}) gives
\begin{equation}
\left| \overline{C} \int_0^\tau \mu^{-\delta} \int_{\mathbb{R}^3} \Lambda_{ik} (\partial^\nu \beta)_{,k} \Lambda^{-1}_{im} \partial^\nu \uptheta^m_\tau \, dy d\tau'  \right| \lesssim \lambda + \int_0^\tau e^{-\mu_0 \tau'} \mathcal{S}^N d \tau'.
\end{equation}
For the last integral on the left hand side of (\ref{E:HIGHORDERPOSTIPGAMMAGREATER5OVER3}) on the first note
\begin{equation}
\partial^\nu (y_k \beta) = y_k \partial^{\nu} \beta + \sum_{|\nu'|=|\nu|-1} c_{\nu',k} \partial^{\nu'} \beta.
\end{equation}
Then the last integral on the left hand side of (\ref{E:HIGHORDERPOSTIPGAMMAGREATER5OVER3}) is
\begin{align}\label{E:LASTINTEGRALPOSTDERIVATIVE}
&- \overline{C} \int_0^\tau \mu^{-\delta}  \int_{\mathbb{R}^3} \Lambda_{ik} \partial^\nu( y_k \beta ) \Lambda^{-1}_{im} \partial^\nu \uptheta^m_\tau \, dy d\tau' \notag \\
&= - \overline{C} \int_0^\tau \mu^{-\delta}  \int_{\mathbb{R}^3} \Lambda_{ik} y_k \partial^{\nu} \beta \Lambda^{-1}_{im} \partial^\nu \uptheta^m_\tau \, dy d\tau' - \overline{C} \int_0^\tau \mu^{-\delta}  \int_{\mathbb{R}^3} \Lambda_{ik} \sum_{|\nu'|=|\nu|-1} c_{\nu',k} \partial^{\nu'} \beta \Lambda^{-1}_{im} \partial^\nu \uptheta^m_\tau \, dy d\tau'. 
\end{align}
Use finite propagation in the same way as in (\ref{E:FINITEPROPAGATIONHIGHORDER}) and then a similar argument to (\ref{E:SOURCETERMHIGHORDERENERGYESTIMATE}) to bound the first integral on the right hand side of (\ref{E:LASTINTEGRALPOSTDERIVATIVE}) by
\begin{equation}
\lambda + \int_0^\tau e^{-\mu_0 \tau'} \mathcal{S}^N d \tau'.
\end{equation}
Finally use a similar argument to (\ref{E:SOURCETERMHIGHORDERENERGYESTIMATE}) to bound the second integral on the right hand side of (\ref{E:LASTINTEGRALPOSTDERIVATIVE}) by
\begin{equation}
\lambda + \int_0^\tau e^{-\mu_0 \tau'} \mathcal{S}^N d \tau'.
\end{equation}
This concludes the energy estimate.
\end{proof}

\section{High Order Curl Estimates}\label{S:CURLESTIMATES}
We now prove the high order curl estimates necessary to control the curl contributions on the right hand side of our energy inequality (\ref{E:ENERGYMAINGAMMAGREATER5OVER3}).
\begin{proposition}\label{P:CURLESTIMATES}
Let $(\uptheta, {\bf V}):\Omega \rightarrow \mathbb R^3\times \mathbb R^3$ be a unique local solution to (\ref{E:THETAEQNLINEARENERGYFUNCTION})-(\ref{E:THETAICGAMMALEQ5OVER3}) on $[0,T^*]$ for $T^*>0$ fixed with $\text{supp} \,\uptheta_0 \subseteq B_1(\mathbf{0})$, $\text{supp}\,\mathbf{V}_0 \subseteq B_1(\mathbf{0})$ and assume $(\uptheta, {\bf V})$ satisfies the a priori assumptions (\ref{E:APRIORI}). Fix $N\geq 4$.  Suppose $\beta$ in (\ref{E:THETAEQNLINEARENERGYFUNCTION}) satisfies $\| \beta \|^2_{H^{N+1}(\mathbb{R}^3)} \leq \lambda$ and $\text{supp} \, \beta \subseteq B_1(\mathbf{0})$ where  $\lambda > 0$ is fixed.  Then for all $\tau \in [0,T^*]$, we have the following inequalities for some $0<\kappa\ll 1$
\begin{align}    
&\mathcal{B}^N[{\bf V}](\tau) \lesssim 
e^{-2\mu_0\tau}\left(\mathcal{S}^N(0)+\mathcal{B}^N[{\bf V}](0)\right)+e^{-2\mu_0 \tau} \lambda+ (1+\tau^2)e^{-2\mu_0\tau}\mathcal{S}^N(\tau), \label{E:CURLVBOUNDGAMMAGREATER5OVER3} \\
& \mathcal{B}^N[\uptheta](\tau) \lesssim \mathcal{S}^N(0)+\mathcal{B}^N[{\bf V}](0) +\lambda+ \kappa \mathcal{S}^N(\tau) + \int_0^\tau e^{-\mu_0\tau'}  \mathcal{S}^N(\tau')\,d\tau'. \label{E:CURLTHETABOUNDGAMMAGREATER5OVER3}
\end{align}
\end{proposition}
\begin{proof}
\textit{Proof of} (\ref{E:CURLVBOUNDGAMMAGREATER5OVER3}).
We have derived the equation (\ref{E:CURLVFINAL}) for $\text{Curl}_{\Lambda \mathscr{A}} \mathbf{V}$ in Lemma \ref{L:CURLEQUATIONDERIVATION}. Apply $\partial^\nu$ to (\ref{E:CURLVFINAL})
\begin{align}
&\text{Curl}_{\Lambda\mathscr{A}}{\partial^\nu \bf V}  =-[\partial^\nu,\text{Curl}_{\Lambda \mathscr{A}}]\mathbf{V} +\frac{1}{1+\alpha}\partial^\nu(\Lambda\mathscr{A} y \times \mathbf{V}) + \frac{\alpha}{1+\alpha}\partial^\nu( (1+\beta)^{-1} (\Lambda\mathscr{A} \nabla \beta \times \mathbf{V})) \notag \\
& + \frac{\mu(0) \partial^\nu(\text{Curl}_{\Lambda \mathscr{A}} ({\bf V}(0)))}{\mu} - \frac{\mu(0)\partial^\nu( \Lambda\mathscr{A} y \times \mathbf{V}(0))}{(1+\alpha) \mu}+ \frac{\alpha \mu(0) \partial^\nu ( (1+\beta)^{-1}(\Lambda\mathscr{A} \nabla \beta \times \mathbf{V}(0))}{(1+\alpha) \mu} \notag \\
&  + \frac{1}{\mu}\int_0^\tau \mu \partial^\nu [\partial_\tau, \text{Curl}_{\Lambda\mathscr{A}}] {\bf V} d\tau' -\frac{1}{(1+\alpha)\mu}\int_0^\tau \mu \partial^\nu [\partial_\tau,\Lambda \mathscr{A} y \times ] \mathbf{V} d\tau ' \notag \\
&+ \frac{\alpha}{(1+\alpha) \mu} \int_0^{\tau} \mu \partial^\nu ((1+\beta)^{-1} [\partial_{\tau},\Lambda \mathscr{A} \nabla \beta \times ] \mathbf{V}) d \tau' \notag \\
& -  \frac{2}{\mu} \int_0^\tau \mu \,\partial^\nu( \text{Curl}_{\Lambda\mathscr{A}}(\Gamma^\ast{\bf V})) d\tau' + \frac{2}{(1+\alpha)\mu} \int_0^\tau \mu \, \partial^\nu(\Lambda\mathscr{A} y \times (\Gamma^\ast \mathbf{V})) d\tau ' \notag \\
&+ \frac{2 \alpha}{(1+\alpha) \mu} \int_0^{\tau} \mu \, \partial^\nu ((1+\beta)^{-1} \Lambda\mathscr{A} \nabla \beta \times (\Gamma^\ast \mathbf{V})) d \tau' \notag \\
& +\frac{\overline{C}}{2 \mu}\int_0^\tau \mu^{1-\delta-\sigma} \partial^\nu(\Lambda y \times  \Lambda \uptheta )d \tau ' -\frac{\overline{C}}{2 \mu}\int_0^\tau \mu^{1-\delta-\sigma} \partial^\nu(\Lambda \mathscr{A}[D \uptheta]y \times \Lambda \eta) d \tau ' \notag \\
& +\frac{\alpha \overline{C}}{(1+\alpha) \mu} \int_0^{\tau} \mu^{1-\delta-\sigma} \partial^\nu( (1+\beta)^{-1}(\Lambda \mathscr{A} \nabla \beta \times \Lambda \uptheta )) d \tau' \notag \\
&-\frac{\alpha \overline{C}}{(1+\alpha) \mu}\int_0^{\tau} \mu^{1-\delta-\sigma} \partial^\nu ( (1+\beta)^{-1} (\Lambda \mathscr{A} \nabla \beta \times \Lambda y)) d \tau'. \label{E:CURLVHIGHORDERDERV}
\end{align}
We take the $ \| \cdot \|^2$ norm of (\ref{E:CURLVHIGHORDERDERV}), and if $|\nu| = N$ then multiply by $[\mu(\tau)]^{-\delta}$, and then estimate the right hand side. As can be seen from (\ref{E:CURLVHIGHORDERDERV}), many terms and hence estimates are similar. Therefore we give the key estimates below and remark similar arguments will hold for the other terms.

For the second term on the right hand side of (\ref{E:CURLVHIGHORDERDERV}), $\frac{1}{1+\alpha}\partial^\nu(\Lambda\mathscr{A} y \times \mathbf{V})$, for $|\nu| \leq N-1$ use finite propagation,
\begin{equation}
\int_{\mathbb{R}^3} \frac{1}{(1+\alpha)^2} |\partial^\nu(\Lambda\mathscr{A} y \times \mathbf{V})|^2 \lesssim (1+\tau^2) e^{-2\mu_0\tau}\mathcal{S}^N
\end{equation}
For $|\nu| = N$, similarly
\begin{equation}
\mu^{-\delta} \| \frac{1}{1+\alpha}\partial^\nu(\Lambda\mathscr{A} y \times \mathbf{V}) \|^2 \lesssim (1+\tau^2) e^{-2\mu_0\tau}\mathcal{S}^N,
\end{equation}
including $\mu^{-\delta}$ in $\mathcal{S}^N$ if needed. 

For the third term on the right hand side of (\ref{E:CURLVHIGHORDERDERV}), use a combination of the $H^2 (\mathbb{R}^3) \hookrightarrow L^\infty (\mathbb{R}^3)$ embedding and boundedness of $\| \beta \|_{H^{N+1,\infty}(\mathbb{R}^3)}$
\begin{equation}
e^{-2\mu_0 \tau}\mathcal{S}^N \gtrsim 
\begin{dcases}
\|\frac{\alpha}{1+\alpha}\partial^\nu( (1+\beta)^{-1} (\Lambda\mathscr{A} \nabla \beta \times \mathbf{V}))\|^2  &\text{if } \ |\nu| \leq N-1,  \\
\mu^{-\delta} \|\frac{\alpha}{1+\alpha}\partial^\nu( (1+\beta)^{-1} (\Lambda\mathscr{A} \nabla \beta \times \mathbf{V}))\|^2 &\text{if } \ |\nu| = N.
\end{dcases}
\end{equation} 
For the fifth term on the right hand side of (\ref{E:CURLVHIGHORDERDERV}), we use the compact support of $\partial^\nu \partial_\tau \uptheta(0)$ for $\nu \leq N$ from finite propagation to in this case bound $|y| \lesssim 1$, and we conclude
\begin{equation}
e^{-2\mu_0 \tau}(\mathcal{S}^N(0)) \gtrsim 
\begin{dcases}
\| \frac{\mu(0)\partial^\nu( \Lambda\mathscr{A} y \times \mathbf{V}(0))}{(1+\alpha) \mu}\|^2  &\text{if } \ |\nu| \leq N-1,  \\
\mu^{-\delta} \|\frac{\mu(0)\partial^\nu( \Lambda\mathscr{A} y \times \mathbf{V}(0))}{(1+\alpha) \mu}\|^2 &\text{if } \ |\nu| = N.
\end{dcases}
\end{equation}
For eighth term on the right hand side of (\ref{E:CURLVHIGHORDERDERV}), $\frac{\alpha}{(1+\alpha) \mu} \int_0^{\tau} \mu \partial^\nu ((1+\beta)^{-1} [\partial_{\tau},\Lambda \mathscr{A} \nabla \beta \times ] \mathbf{V}) d \tau'$, first compute
\begin{equation}
\partial^\nu [\partial_\tau,\Lambda \mathscr{A} y \times]\partial_\tau \uptheta^k_j = \partial^\nu (\partial_\tau (\Lambda_{jm} \mathscr{A}^s_m) y_s \partial_\tau \uptheta^k-\partial_\tau(\Lambda_{km}\mathscr{A}^s_m) y_s \partial_\tau \uptheta^j).
\end{equation}
For the first term on the right hand side,
\begin{align}
&\partial_\tau (\partial_\tau \Lambda_{jm} \mathscr{A}^s_m y_s + \Lambda_{jm} \partial_\tau \mathscr{A}^s_m y_s) \partial_\tau \uptheta^k  = \partial_\tau \Lambda_{jm}((\partial^\nu \mathscr{A}^s_m) y_s \partial_\tau \uptheta^k + \mathscr{A}_m^s y_s (\partial^\nu \partial_\tau \uptheta^k)) \notag \\
& + \Lambda_{jm} (\partial^\nu (\partial_\tau \mathscr{A}^s_m) y_s \partial_\tau \uptheta^k + \partial_\tau \mathscr{A}^s_m y_s \partial^\nu \partial_\tau \uptheta^k) + \mathscr{R}.
\end{align}
We are denoting by $\mathscr{R}$ above favorable remainder terms. Schematically consider the first two terms on the right hand side of above
\begin{equation}\label{E:CURLERROR1}
\underbrace{\partial_\tau\Lambda \,\partial^\nu D\uptheta y {\bf V}}_{=:D_1} + \underbrace{\partial_\tau\Lambda  A \, y \partial^\nu {\bf V}}_{=:D_2} + \underbrace{\Lambda (\partial^\nu D{\bf V}) \, y {\bf V}}_{=:D_3} +\underbrace{\Lambda D{\bf V} \, y\partial^\nu {\bf V}}_{=:D_4}
\end{equation}
For $D_1$
\begin{align}
& \int_{\mathbb{R}^3} \frac{1}{\mu^2} \Big| \int_0^\tau \mu \partial_\tau\Lambda \,\partial^\nu D\uptheta y {\bf V} \,d\tau'\Big|^2\,dy \notag \\
& \lesssim e^{-2\mu_1\tau} \int_{\mathbb{R}^3} \Big|\int_0^\tau \left|e^{\mu_1\tau'}(1+K\tau') \partial_\tau\Lambda \,\partial^\nu D\uptheta  {\bf V} \right|\,d\tau'\Big|^2\,dy  \notag \\
& \lesssim e^{-2\mu_1\tau} \sup_{0\le\tau'\le\tau} \int_{\mathbb{R}^3}  |\partial^\nu D\uptheta|^2 \Big| \int_0^\tau  |e^{\mu_1 \tau'} (1+\tau') \partial_\tau\Lambda {\bf V} | \,d\tau'\Big|^2 \,dy \notag \\
& \lesssim  e^{-2\mu_1 \tau}  \sup_{0\le\tau'\le\tau} \left( \| \partial^\nu D\uptheta\|^2\right) \Big| \int_0^\tau e^{\mu_1 \tau'} (1+\tau') e^{-\mu_1 \tau'} e^{-\mu_0 \tau'} (\mathcal{S}^N)^{1/2} d \tau' \Big|^2 \notag \\
& \lesssim e^{-2\mu_0\tau} \mathcal{S}^N(\tau), \label{E:CURLEST1}
\end{align}
where we use the $H^2 (\mathbb{R}^3) \hookrightarrow L^\infty (\mathbb{R}^3)$ embedding on $\| \mathbf{V} \|_{L^\infty}$. For $D_2$,
\begin{align}
\int_{\mathbb{R}^3} \frac{1}{\mu^2} | \int_0^\tau \mu \partial_\tau \Lambda \mathscr{A} y \partial^\nu \partial_\tau \uptheta d \tau' |^2 & \lesssim \int_{\mathbb{R}^3} \frac{1}{\mu^2} (1+K\tau)^2 \sup_{0 \leq \tau \leq \tau'} | \partial^\nu \partial_\tau \uptheta |^2 | \int_0^\tau 1 d \tau'|^2 \notag \\
& \lesssim e^{-2\mu_1 \tau'} (1+K\tau)^2 e^{-2\mu_0 \tau} \mathcal{S}^N \tau^2 \notag \\
& \lesssim e^{-2\mu_0 \tau'} \mathcal{S}^N
\end{align}
For $D_3$, first integrate by parts in $\tau$,
\begin{align}
&\frac{1}{\mu} \int_0^\tau \mu \Lambda (\partial^\nu \partial_\tau D \uptheta) y \partial_\tau \uptheta d \tau' =  \frac1\mu\left(\mu \Lambda \mathbf{V} \,(\partial^\nu D \uptheta) y \right)\big|^\tau_0  \notag \\
 &- \frac1\mu\int_0^\tau \mu(\frac{\mu_\tau}\mu \Lambda \partial_\tau \uptheta y + \partial_\tau \Lambda \partial_\tau \uptheta y + \Lambda \partial_{\tau\tau} \uptheta y ) \,\partial^\nu D \uptheta\,d\tau' \notag \\
&=A+B+C+D.
\end{align}
Then
\begin{align}
\|A \| &\lesssim (1+K \tau) \| \partial_\tau \uptheta \|_{L^\infty} \| \partial^\nu D \uptheta \|+\mu^{-1} \mathcal{S}^N(0) \lesssim (1+\tau) e^{-\mu_0 \tau} (\mathcal{S}^N)^{1/2} \| \partial^\nu D \uptheta \| + e^{-\mu_0 \tau} \mathcal{S}^N(0), \notag \\
\| B \| & \lesssim \sup_{0 \leq \tau' \leq \tau} \| \partial^\nu D \uptheta \| e^{-\mu_1 \tau} \int_0^\tau (1+\tau') e^{\mu_1 \tau'-\mu_0 \tau'} \| \partial_\tau \uptheta \|_{L^\infty} d \tau', \notag \\
&\lesssim \sup_{0 \leq \tau' \leq \tau} \| \partial^\nu D \uptheta \| (\mathcal{S}^N)^{1/2} e^{-\mu_0\tau} (1+\tau) \notag \\
\| C \| &\lesssim \sup_{0 \leq \tau' \leq \tau} \| \partial^\nu D \uptheta \| (\mathcal{S}^N)^{1/2} e^{-\mu_0 \tau}.
\end{align}
Now for $D$ in the same way as for $B$ first note
\begin{equation}
\left| e^{-\mu_1 \tau} \int_0^\tau e^{\mu_1 \tau'-\mu_0 \tau'}  d \tau' \right| \lesssim e^{-\mu_0\tau}.
\end{equation}
Then using our equation for $\uptheta$ (\ref{E:THETAEQNLINEARENERGYFUNCTION}) to rewrite $\partial_{\tau \tau} \uptheta$ and noting each term gains a $\mu^{-\sigma}$ factor from $\mu^{\sigma} \partial_{\tau \tau} \uptheta_i$ in (\ref{E:THETAEQNLINEARENERGYFUNCTION}) and our a priori assumption $\mathcal{S}^N < 1/3$, we have
\begin{equation}
\| D \| \lesssim \sup_{0 \leq \tau' \leq \tau} \| \partial^\nu D \uptheta \| e^{-\mu_0 \tau} (1+\tau),
\end{equation}
For $D_4$,
\begin{align}
\int_{\mathbb{R}^3} \frac{1}{\mu^2} | \int_0^\tau \mu \Lambda D{\bf V} y \partial^\nu \partial_\tau \uptheta d \tau' |^2 & \lesssim \int_{\mathbb{R}^3} \frac{1}{\mu^2} (1+K\tau)^2 \| D {\bf V} \|_{L^\infty}^2 \sup_{0 \leq \tau \leq \tau'} | \partial^\nu \partial_\tau \uptheta |^2 | \int_0^\tau 1 d \tau'|^2 \notag \\
& \lesssim e^{-2\mu_1 \tau'} (1+K\tau)^2 e^{-2\mu_0 \tau} \mathcal{S}^N \tau^2 \notag \\
& \lesssim e^{-2\mu_0 \tau'} \mathcal{S}^N
\end{align}
Hence
\begin{equation}
(1+\tau^2)e^{-2\mu_0 \tau}\mathcal{S}^N \gtrsim 
\begin{dcases}
\|\frac{1}{(1+\alpha)\mu}\int_0^\tau \mu \partial^\nu [\partial_\tau,\Lambda \mathscr{A} y \times ] \mathbf{V} d\tau '\|^2  &\text{if } \ |\nu| \leq N-1,  \\
\mu^{-\delta} \|\frac{1}{(1+\alpha)\mu}\int_0^\tau \mu \partial^\nu [\partial_\tau,\Lambda \mathscr{A} y \times ] \mathbf{V} d\tau '\|^2 &\text{if } \ |\nu| = N.
\end{dcases}
\end{equation}
For the thirteenth term on the right hand side of (\ref{E:CURLVHIGHORDERDERV}), $\frac{\overline{C}}{2 \mu}\int_0^\tau \mu^{1-\delta-\sigma} \partial^\nu(\Lambda y \times  \Lambda \uptheta )d \tau$, first compute
\begin{equation}
\partial^\nu[(\Lambda y \times \Lambda \uptheta)]^i_j = \partial^\nu ( \Lambda_{js} y_s \Lambda_{ik} \uptheta^k - \Lambda_{is} y_s \Lambda_{jk} \uptheta^k).
\end{equation}
For the left term
\begin{equation}\label{E:H9DERIVATIVELEFTTERM}
\partial^\nu (\Lambda_{js} y_s \Lambda_{ik} \uptheta^k) = \Lambda_{js} \Lambda_{ik} y_s \partial^\nu + \sum_{|\nu'|=|\nu|-1} c_{v',s} \Lambda_{js} \Lambda_{ik} \partial^{\nu'} \uptheta^k
\end{equation}
Schematically consider the first term
\begin{equation}
\Lambda \Lambda y \partial^\nu \uptheta := C_1
\end{equation}
Then:
\begin{align}\label{E:C1BOUNDFIRSTGAMMALEQ5OVER3}  
\| \frac{\overline{C}}{2 \mu}\int_0^\tau \mu^{1-\delta-\sigma}  C_1 d \tau ' \|^2 &= \int_{\mathbb{R}^3} \frac{\overline{C}^2}{4 \mu^2} \Big| \int_0^\tau \mu^{1-\delta-\sigma} \Lambda \Lambda y \partial^\nu \uptheta d \tau' \Big|^2 d y \notag \\
&\lesssim(1+\tau^2) \frac{1}{\mu^2} \Big| \int_0^\tau \mu^{1-\delta-\sigma} d \tau' \Big|^2 \sup_{0 \leq \tau \leq \tau'} \|  \partial^\nu \uptheta \|^2 \notag \\
&\lesssim (1+\tau^2) U(\tau) \mathcal{S}^N(\tau),
\end{align}
where we define $U(\tau)=\frac{1}{\mu^2} | \int_0^\tau \mu^{1-\delta-\sigma} d \tau'|^2$. Now for $U$,
\begin{align}\label{E:UBOUNDFIRSTGAMMALEQ5OVER3}
U &\lesssim e^{-2 \mu_1 \tau} \Big| \int_0^\tau e^{(1-\delta-\sigma)\mu_1 \tau'} d \tau' \Big|^2 \notag \\
&\lesssim e^{-2 \mu_1\tau}(e^{2(1-\delta-\sigma)\mu_1\tau}+1) \notag \\
& \lesssim e^{-2 \mu_0\tau}
\end{align}   
Hence
\begin{equation}
\| \frac{\overline{C}}{2 \mu}\int_0^\tau \mu^{1-\delta-\sigma}  C_1 d \tau ' \|^2 \lesssim (1+\tau^2) e^{-2\mu_0\tau}\mathcal{S}^N
\end{equation}
Similar for the sum in (\ref{E:H9DERIVATIVELEFTTERM}) except bound is just $e^{-2\mu_0\tau}\mathcal{S}^N$ there. Hence
\begin{equation}
(1+\tau^2)e^{-2\mu_0 \tau}\mathcal{S}^N \gtrsim 
\begin{dcases}
\| \frac{\overline{C}}{2 \mu}\int_0^\tau \mu^{1-\delta-\sigma} \partial^\nu(\Lambda y \times  \Lambda \uptheta )d \tau ' \|^2  &\text{if } \ |\nu| \leq N-1  \\
\mu^{-\delta} \|\frac{\overline{C}}{2 \mu}\int_0^\tau \mu^{1-\delta-\sigma} \partial^\nu(\Lambda y \times  \Lambda \uptheta )d \tau '\|^2 &\text{if } \ |\nu| = N
\end{dcases}
\end{equation}
For the fourteenth term on the right hand side of (\ref{E:CURLVHIGHORDERDERV}), $-\frac{\overline{C}}{2 \mu}\int_0^\tau \mu^{1-\delta-\sigma} \partial^\nu(\Lambda \mathscr{A}[D \uptheta]y \times \Lambda \eta) d \tau '$, first rewrite 
\begin{align}
&[\Lambda \mathscr{A} [D \uptheta] y \times \Lambda \eta]^i_j =\Lambda_{jm}\mathscr{A}^s_\ell \uptheta^\ell,_m y_s \Lambda_{ik} \eta^k - \Lambda_{im} \mathscr{A}^s_\ell \uptheta^\ell,_m y_s \Lambda_{jk} \eta^k \notag \\ 
&=(\Lambda_{jm}\mathscr{A}^s_\ell \uptheta^\ell,_m y_s \Lambda_{ik} \uptheta^k - \Lambda_{im} \mathscr{A}^s_\ell \uptheta^\ell,_m y_s \Lambda_{jk} \uptheta^k) + (\Lambda_{jm}\mathscr{A}^s_\ell \uptheta^\ell,_m y_s \Lambda_{ik} y^k - \Lambda_{im} \mathscr{A}^s_\ell \uptheta^\ell,_m y_s \Lambda_{jk} y^k) \notag \\
&=[\Lambda \mathscr{A} [D \uptheta] y \times \Lambda \uptheta]^i_j + [\Lambda \mathscr{A} [D \uptheta] y \times \Lambda y]^i_j.
\end{align}
The $\Lambda \mathscr{A} [D \uptheta] y \times \Lambda \uptheta$ term is similar to other terms and analogous arguments apply to estimate this term. For the $\Lambda \mathscr{A} [D \uptheta] y \times \Lambda y$ term, first compute:
\begin{align}
&\partial^\nu (\Lambda_{jm}\mathscr{A}^s_\ell \uptheta^\ell,_m y_s \Lambda_{ik} y^k) = \underbrace{c \Lambda \Lambda (\partial^\nu (D \uptheta D \uptheta)) y y}_{=:G_1} \notag \\
&+ \sum_{|\nu'| = |\nu|-2} c (\partial^{\nu'} (D \uptheta D \uptheta) ) \Lambda \Lambda + \sum_{|\nu'| = |\nu|-1} c (\partial^{\nu'} (D \uptheta D \uptheta) ) y \Lambda \Lambda.
\end{align}
The two sums above are straightforward to handle using methods from above and $H^2 (\mathbb{R}^3) \hookrightarrow L^\infty (\mathbb{R}^3)$ embedding. For $G_1$, for $0 \leq |\nu| \leq N-1$,
\begin{align}\label{E:C1BOUNDFIRSTGAMMALEQ5OVER3}  
\| \frac{\overline{C}}{2 \mu}\int_0^\tau \mu^{1-\delta-\sigma}  G_1 d \tau ' \|^2 &= \int_{\mathbb{R}^3} \frac{\overline{C}^2}{4 \mu^2} \Big| \int_0^\tau c \mu^{1-\delta-\sigma} \Lambda \Lambda \partial^\nu (D \uptheta D \uptheta) y y d \tau' \Big|^2 d y \notag \\
&\lesssim \frac{1}{\mu^2} \Big| \int_0^\tau (1+K \tau')^2 \mu^{1-\delta-\sigma} d \tau' \Big|^2 \sup_{0 \leq \tau \leq \tau'} \|  \partial^\nu \uptheta \|^2 \notag \\
&\lesssim \frac{1}{\mu^2} \Big| \int_0^\tau (1+K \tau')^2 \mu^{1-\delta-\sigma} d \tau' \Big|^2 \sup_{0 \leq \tau \leq \tau'} \|  \partial^\nu \uptheta \|^2 \notag \\
&\lesssim  \frac{1}{\mu^2} \Big | \int_0^\tau \mu^{1-\sigma} d \tau' \Big|^2 \sup_{0 \leq \tau \leq \tau'} \|  \partial^\nu \uptheta \|^2 \notag \\
&\lesssim e^{-2\mu_1 \tau}(e^{2\mu_1\tau-\sigma \mu_1 \tau}+1) \mathcal{S}^N(\tau) \notag \\
&\lesssim e^{-2\mu_0 \tau} \mathcal{S}^N(\tau).
\end{align}
For $\nu=N$ a similar argument holds except there we include the additional, artificially introduced, time weight $\mu^{-\delta}$ in $\mathcal{S}^N$. Hence
\begin{equation}
(1+\tau^2)e^{-2\mu_0 \tau}\mathcal{S}^N \gtrsim 
\begin{dcases}
\|\frac{\overline{C}}{2 \mu}\int_0^\tau \mu^{1-\delta-\sigma} \partial^\nu(\Lambda \mathscr{A}[D \uptheta]y \times \Lambda \eta) d \tau ' \|^2  &\text{if } \ |\nu| \leq N-1  \\
\mu^{-\delta} \| \frac{\overline{C}}{2 \mu}\int_0^\tau \mu^{1-\delta-\sigma} \partial^\nu(\Lambda \mathscr{A}[D \uptheta]y \times \Lambda \eta) d \tau ' \|^2 &\text{if } \ |\nu| = N
\end{dcases}
\end{equation}
For the last term on the right hand side of (\ref{E:CURLVHIGHORDERDERV}), use the compact support of $\beta$ and it's derivatives to bound $y$, and then apply (\ref{E:UBOUNDFIRSTGAMMALEQ5OVER3}) in combination with the $H^2 (\mathbb{R}^3) \hookrightarrow L^\infty (\mathbb{R}^3)$ embedding and $\| \beta \|^2_{H^{N+1,\infty}(\mathbb{R}^3)} \leq \lambda$ to obtain
\begin{equation}
e^{-2\mu_0 \tau} \lambda \gtrsim 
\begin{dcases}
\|\frac{\alpha \overline{C}}{(1+\alpha) \mu}\int_0^{\tau} \mu^{1-\delta-\sigma} \partial^\nu ( (1+\beta)^{-1} (\Lambda \mathscr{A} \nabla \beta \times \Lambda y)) d \tau'\|^2  &\text{if } \ |\nu| \leq N-1  \\
\mu^{-\delta} \|\frac{\alpha \overline{C}}{(1+\alpha) \mu}\int_0^{\tau} \mu^{1-\delta-\sigma} \partial^\nu ( (1+\beta)^{-1} (\Lambda \mathscr{A} \nabla \beta \times \Lambda y)) d \tau'\|^2 &\text{if } \ |\nu| = N
\end{dcases}
\end{equation}
This concludes the proof of (\ref{E:CURLVBOUNDGAMMAGREATER5OVER3}). \\ \\
\textit{Proof of} (\ref{E:CURLTHETABOUNDGAMMAGREATER5OVER3}). Using the Fundamental Theorem of Calculus we rewrite two terms in the equation (\ref{E:CURLTHETAFINAL}) for $\text{Curl}_{\Lambda \mathscr{A}} \uptheta$ in Lemma \ref{L:CURLEQUATIONDERIVATION}. Apply $\partial^\nu$ to the resulting equation
\begin{align}
&\text{Curl}_{\Lambda\mathscr{A}}{\partial^\nu \uptheta}  = -[\partial^\nu,\text{Curl}_{\Lambda \mathscr{A}}]\uptheta +  \frac{1}{1+\alpha} \partial^\nu(\Lambda \mathscr{A} y \times \uptheta) +  \frac{\alpha}{1+\alpha} \partial^\nu((1+\beta)^{-1} \Lambda\mathscr{A} \nabla \beta \times \uptheta) \notag \\
&+\partial^\nu(\text{Curl}_{\Lambda \mathscr{A}}([\uptheta(0)])) - \frac{1}{1+\alpha} \partial^\nu(\Lambda \mathscr{A} y \times \uptheta(0)) - \frac{\alpha}{1+\alpha} \partial^\nu( (1+\beta)^{-1} \Lambda\mathscr{A} \nabla \beta \times \uptheta (0) ) \notag \\
&+\mu(0) \partial^\nu( \text{Curl}_{\Lambda \mathscr{A}} ({\bf V}(0))) \int_0^\tau \frac{1}{\mu(\tau')} d \tau' - \frac{\mu(0) \partial^\nu( \Lambda\mathscr{A} y \times \mathbf{V}(0))}{1+\alpha} \int_0^\tau \frac{1}{\mu(\tau')} d \tau' \notag \\
&- \frac{\alpha \mu(0) \partial^\nu( (1+\beta)^{-1} \Lambda\mathscr{A} \nabla \beta \times \mathbf{V}(0))}{1+\alpha}\int_0^\tau \frac{1}{\mu(\tau')} d \tau' \notag \\
&+ \int_0^\tau \partial^\nu( [\partial_\tau, \text{Curl}_{\Lambda\mathscr{A}}] {\uptheta}) d\tau' + \frac{1}{1+\alpha}\int_0^\tau \partial^\nu( [\partial_\tau,\Lambda \mathscr{A} y \times] \uptheta) \notag \\
&+ \frac{\alpha}{(1+\alpha)} \int_0^\tau \partial^\nu ( (1+\beta)^{-1} [\partial_\tau, \Lambda\mathscr{A} \nabla \beta \times] \uptheta) d\tau' \notag \\
&+\int_0^\tau \frac{1}{\mu(\tau')} \int_0^{\tau'} \mu(\tau'') \partial^\nu([\partial_\tau, \text{Curl}_{\Lambda\mathscr{A}}] {\bf V} ) d\tau'' d \tau' \notag \\
&-\frac{1}{1+\alpha} \int_0^{\tau} \frac{1}{\mu(\tau')} \int_0^{\tau'} \mu(\tau '') \partial^\nu([\partial_\tau,\Lambda \mathscr{A} y \times ] \mathbf{V} ) d\tau '' d \tau '  \notag \\
& - \frac{\alpha}{(1+\alpha)} \int_0^\tau \frac{1}{\mu(\tau')} \int_0^{\tau'} \mu(\tau'') \partial^\nu((1+\beta)^{-1}[\partial_\tau,\Lambda \mathscr{A} \nabla \beta \times ] \mathbf{V} )d \tau'' d \tau' \notag \\
& - \int_0^\tau \frac{2}{\mu(\tau')} \int_0^{\tau '} \mu(\tau'') \, \partial^\nu(\text{Curl}_{\Lambda\mathscr{A}}(\Gamma^\ast{\bf V})) d\tau'' d \tau' \notag \\
& + \frac{2}{1+\alpha} \int_0^{\tau} \frac{1}{\mu(\tau')} \int_0^{\tau '} \mu(\tau'') \, \partial^\nu(\Lambda\mathscr{A} y \times (\Gamma^\ast \mathbf{V}) ) d\tau '' d \tau ' \notag \\
& + \frac{2 \alpha}{1+\alpha} \int_0^\tau \frac{1}{\mu(\tau')} \int_0^{\tau'} \mu(\tau'') \, \partial^\nu((1+\beta)^{-1} \Lambda\mathscr{A} \nabla \beta \times (\Gamma^\ast \mathbf{V})) d\tau'' d \tau' \notag \\
& +\frac{\overline{C}}{1+\alpha} \int_0^\tau \frac{1}{\mu} \int_0^{\tau'} \mu^{1-\delta-\sigma} \partial^\nu( \Lambda y \times  \Lambda \uptheta) d \tau '' d \tau ' \notag \\ 
&-\frac{\overline{C}}{1+\alpha} \int_0^\tau \frac{1}{\mu}\int_0^{\tau'} \mu^{1-\delta-\sigma} \partial^\nu(\Lambda \mathscr{A}[D \uptheta]y \times \Lambda \eta) d \tau '' d \tau ' \notag \\
&+\frac{\alpha \overline{C}}{1+\alpha} \int_0^\tau \frac{1}{\mu} \int_0^{\tau'} \mu^{1-\delta-\sigma} \partial^\nu((1+\beta)^{-1} (\Lambda \nabla \beta \times  \Lambda \uptheta)) d \tau'' d\tau' \notag \\
&-\frac{\alpha \overline{C}}{1+\alpha} \int_0^{\tau} \frac{1}{\mu} \int_0^{\tau'} \mu^{1-\delta-\sigma} \partial^\nu ((1+\beta)^{-1} (\Lambda \mathscr{A} \nabla \beta \times \Lambda y)) d \tau'' d\tau' \label{E:CURLTHETAHIGHORDERDERV}
\end{align}
We take the $ \| \cdot \|^2$ norm of (\ref{E:CURLTHETAHIGHORDERDERV}), and if $|\nu| = N$ then multiply by $[\mu(\tau)]^{-\delta}$, and then estimate the right hand side. As can be seen from (\ref{E:CURLTHETAHIGHORDERDERV}) and (\ref{E:CURLVHIGHORDERDERV}), many terms and hence estimates are similar. Therefore we give the key estimates below and remark similar arguments to these and the proof of (\ref{E:CURLVBOUNDGAMMAGREATER5OVER3}) will hold for the other terms.

For the second term on the right hand side of (\ref{E:CURLTHETAHIGHORDERDERV}), $\frac{1}{1+\alpha} \partial^\nu(\Lambda \mathscr{A} y \times \uptheta)$,
\begin{equation}
\partial^\nu (\Lambda \mathscr{A} y \times \uptheta) = \Lambda y \partial^\nu (\mathscr{A} \uptheta) + \mathscr{R},
\end{equation}
where $\mathscr{R}$ are favorable remainder terms. Note 
\begin{equation}
\partial^\nu (\mathscr{A} \uptheta) = \mathscr{A} \partial^\nu \uptheta + \partial^{\nu-1}((D^2 \uptheta) \uptheta)
\end{equation}
Then
\begin{align}
\| \Lambda y \mathscr{A} \partial^\nu \uptheta \|^2 &\lesssim \| \int_0^\tau y \partial^\nu \partial_\tau \uptheta d \tau' + y \partial^\nu \uptheta(0) \|^2 \notag \\
&\lesssim \int_{\mathbb{R}^3} |\int_0^\tau (1+K \tau') \partial^\nu \partial_\tau \uptheta d \tau' |^2 + |\partial^\nu \uptheta(0)|^2 \, dy \notag \\
&\lesssim \int_{\mathbb{R}^3} |\int_0^\tau (1+K \tau')\mu^{-\sigma/4} \mu^{-\sigma/4} \mu^{\sigma/2} \partial^\nu \partial_\tau \uptheta d \tau' |^2 dy + \mathcal{S}^N(0) \notag \\
& \lesssim \int_{\mathbb{R}^3} (\int_0^\tau (1+K \tau')^2 e^{-\mu_0 \tau'} d \tau')(\int_0^\tau e^{-\mu_0 \tau'} \mu^{\sigma} |\partial^\nu \partial_\tau \uptheta|^2 d \tau') dy + \mathcal{S}^N(0) \notag \\
& \lesssim \int_0^\tau e^{-\mu_0 \tau'} \mathcal{S}^N d \tau' + \mathcal{S}^N(0).
\end{align}
Next,
\begin{align}
\| \Lambda y \partial^{\nu-1}((D^2 \uptheta) \uptheta) \|^2 &\lesssim \mathcal{S}^N \| y \partial^{\nu_\ell} \uptheta \|^2_{L^{\infty}(\mathbb{R}^3)} \text{ where } |\nu_\ell| \leq \frac{N}{2} \notag \\
&\lesssim ( \| y \partial^{\nu_\ell} \uptheta(0) \|_{L^{\infty}}+\int_0^\tau \| y \partial^{\nu_\ell} \partial_\tau \uptheta \|_{L^\infty} )^2 \notag \\
&\lesssim \mathcal{S}^N(0) + (\int_0^\tau (1+K \tau') e^{-\mu_0 \tau'} \sqrt{\mathcal{S}^N})^2 \notag \\
&\lesssim \mathcal{S}^N(0) + (\int_0^\tau (1+K \tau') e^{-\mu_0 \tau'} d \tau')(\int_0^\tau e^{-\mu_0 \tau'} \mathcal{S}^N d \tau') \notag \\
&\lesssim \mathcal{S}^N(0) + \int_0^\tau e^{-\mu_0 \tau'} \mathcal{S}^N d \tau'
\end{align}
Thus, since top order here with $\mu^{-\delta}$ is simple,
\begin{equation}
\mathcal{S}^N(0)+ \int_0^\tau e^{-\mu_0 \tau'} \mathcal{S}^N d \tau'
\gtrsim 
\begin{dcases}
\|\frac{1}{1+\alpha} \partial^\nu(\Lambda \mathscr{A} y \times \uptheta)\|^2  &\text{if } \ |\nu| \leq N-1  \\
\mu^{-\delta} \|\frac{1}{1+\alpha} \partial^\nu(\Lambda \mathscr{A} y \times \uptheta)\|^2 &\text{if } \ |\nu| = N
\end{dcases}
\end{equation}
For the eleventh term on the right hand side of (\ref{E:CURLTHETAHIGHORDERDERV}), $\frac{1}{1+\alpha}\int_0^\tau \partial^\nu( [\partial_\tau,\Lambda \mathscr{A} y \times] \uptheta)$, note
\begin{equation}
\partial^\nu([\partial_\tau,\Lambda \mathscr{A} y \times] \uptheta) = y \partial^\nu (\partial_\tau (\Lambda \mathscr{A}) \uptheta) + \mathscr{R},
\end{equation}
where $\mathscr{R}$ are favorable remainder terms. Note $\partial_\tau(\Lambda \mathscr{A}) = (\partial_\tau \Lambda) \mathscr{A} + \Lambda (\partial_\tau \mathscr{A})$. Also $\partial^\nu (\mathscr{A} \uptheta) = \partial^{\nu-1} ((D^2 \uptheta) \uptheta) + \mathscr{A} \partial^\nu(\uptheta)$. Then
\begin{align}
& \| \int_0^\tau y \partial_{\tau} \Lambda \partial^{\nu-1}((D^2 \uptheta) \uptheta) d \tau' \|^2 \notag \\
& \lesssim \int_{\mathbb{R}^3} | \int_0^\tau (1+K \tau') e^{-\frac{\mu_0}{2} \tau'}  e^{-\frac{\mu_0}{2} \tau'} \partial^{\nu_h} \uptheta \partial^{\nu_\ell} \uptheta d \tau' |^2 dy \quad \left( \text{where } |\nu_h| \geq \frac{N}{2}+1, |\nu_\ell| \leq \frac{N}{2} \right) \notag \\
& \lesssim \|\partial^{\nu_\ell} \uptheta \|^2_{L^\infty} \int_{\mathbb{R}^3} (\int_0^\tau e^{-\mu_0 \tau'} | \partial^{\nu_h} \uptheta | ^2 d \tau') (\int_0^\tau e^{-\mu_0 \tau'} (1+K \tau')^2 d \tau') dy \notag \\
&\lesssim \int_0^\tau e^{-\mu_0 \tau'} \mathcal{S}^N d \tau'.
\end{align} 
Similar for $y \partial_\tau \Lambda \mathscr{A} \partial^\nu \uptheta$, except just bound $\mathscr{A}$. Next,
\begin{equation}
\partial^\nu(\partial_\tau \mathscr{A} \uptheta) = (\partial^\nu (\partial_\tau D \uptheta))\uptheta+\partial^{\nu-1}([\partial_\tau D \uptheta][D \uptheta]).
\end{equation}
For $|\nu| \leq N-1$,
\begin{align}
& \| \int_0^\tau \Lambda y \partial^{\nu}(\partial_\tau D \uptheta) \uptheta d \tau' \|^2  \lesssim \int_{\mathbb{R}^3} | \int_0^\tau (1+K \tau') e^{-\frac{\mu_0}{2} \tau'}  e^{-\frac{\mu_0}{2} \tau'} e^{\mu_0 \tau'}(\partial^{\nu+1} \partial_\tau \uptheta )  \uptheta d \tau' |^2 dy \notag \\
& \lesssim \| \uptheta \|^2_{L^\infty} \int_{\mathbb{R}^3} (\int_0^\tau e^{-\mu_0 \tau'} e^{2 \mu_0 \tau'} |\partial^{\nu+1} \partial_\tau \uptheta|^2 d \tau')(\int_0^\tau e^{-\mu_0 \tau'}(1+K \tau')^2 d \tau') dy \notag \\
& \lesssim \int_0^\tau e^{-\mu_0 \tau'} \mathcal{S}^N d \tau'.
\end{align}
For $\nu = N$ we integrate by parts in $\tau$
\begin{align}
\int_0^\tau \Lambda y \uptheta \partial_\tau (\partial^{\nu+1} \uptheta) d \tau' & = \Lambda y \uptheta \partial^{\nu+1} \uptheta|_0^\tau - \int_0^\tau \Lambda y (\partial_\tau \uptheta) \partial^{\nu+1} \uptheta d \tau' - \int_0^\tau (\partial_\tau \Lambda) y \uptheta \partial^{\nu+1} \uptheta d \tau' \notag \\
&= (I)+(II)+(III).
\end{align}
Note
\begin{align}
\mu^{-\delta} \| (I) \|^2 &\lesssim \mathcal{S}^N(0) + \mu^{-2} \int_{\mathbb{R}^3} | y \uptheta |^2 | \partial^{\nu+1} \uptheta |^2 dy \notag \\
& \lesssim \mathcal{S}^N(0) + \mathcal{S}^N \|y \uptheta \|^2_{L^\infty} \notag \\
& \lesssim \mathcal{S}^N(0) + \mathcal{S}^N ( \| y \uptheta(0) \|_{L^\infty} + \int_0^\tau \| y \partial_\tau \uptheta \|_{L^\infty} d \tau')^2 \notag \\
& \lesssim \mathcal{S}^N(0) + \mathcal{S}^N (\mathcal{S}^N(0) + (\int_0^\tau e^{-\mu_0 \tau'} \sqrt{\mathcal{S}^N} (1+K \tau') d \tau')^2) \notag \\
& \lesssim \mathcal{S}^N(0) + \mathcal{S}^N  (\mathcal{S}^N(0) + (\int_0^\tau e^{-\mu_0 \tau'} (1+K \tau')^2 d \tau')(\int_0^\tau e^{-\mu_0 \tau'} \mathcal{S}^N  d \tau')) \notag \\
& \lesssim \mathcal{S}^N(0) + \int_0^\tau e^{-\mu_0 \tau'} \mathcal{S}^N d \tau'.
\end{align}
For $(II)$, let $\varsigma>0$ be such that $2\varsigma < \mu_0.$ Then,
\begin{align}
\mu^{-\delta} \| (II) \|^2 &=\mu^{-\delta} \int_{\mathbb{R}^3} | \int_0^\tau \Lambda y \partial_\tau \uptheta \partial^{\nu+1} \uptheta d \tau'|^2 \notag \\
&\lesssim \mu^{-\delta} \int_{\mathbb{R}^3} (\int_0^\tau e^{2 \varsigma \tau'} | y \partial_\tau \uptheta |^2 d \tau')(\int_0^\tau e^{-2\varsigma \tau'} |\partial^{\nu+1} \uptheta|^2 d \tau') dy \notag \\
&\lesssim \int_0^\tau e^{2 \varsigma \tau'} \|y \partial_\tau \uptheta\|_{L^\infty}^2 d \tau'  \mathcal{S}^N \int_0^\tau e^{-2 \varsigma \tau'} d \tau' \notag \\
&\lesssim  \int_0^\tau e^{2 \varsigma \tau'} (1+K \tau')^2 e^{-\mu_0 \tau'} e^{-\mu_0 \tau'} \mathcal{S}^N d \tau' \notag \\
&\lesssim  \int_0^\tau e^{-\mu_0 \tau'} \mathcal{S}^N d \tau'.
\end{align}
Now $(III)$ is similar to other estimates as well as top order $\mu^{-\delta}$ inside integral argument. Also $\partial^{\nu-1}([\partial_\tau D \uptheta][D \uptheta])$ estimate is similar to other estimates. Hence
\begin{equation}
\mathcal{S}^N(0)+\int_0^\tau e^{-\mu_0 \tau'} \mathcal{S}^N d \tau'
\gtrsim
\begin{dcases}
\|\frac{1}{1+\alpha}\int_0^\tau \partial^\nu( [\partial_\tau,\Lambda \mathscr{A} y \times] \uptheta)\|^2  &\text{if } \ |\nu| \leq N-1  \\
\mu^{-\delta} \|\frac{1}{1+\alpha}\int_0^\tau \partial^\nu( [\partial_\tau,\Lambda \mathscr{A} y \times] \uptheta)\|^2 &\text{if } \ |\nu| = N
\end{dcases}
\end{equation}
For the fourteenth term on the right hand side of (\ref{E:CURLTHETAHIGHORDERDERV}) which is as follows
$$-\frac{1}{1+\alpha} \int_0^{\tau} \frac{1}{\mu(\tau')} \int_0^{\tau'} \mu(\tau '') \partial^\nu([\partial_\tau,\Lambda \mathscr{A} y \times ] \mathbf{V} ) d\tau '' d \tau ' $$
we first note
\begin{equation}
\partial^\nu([\partial_\tau,\Lambda \mathscr{A} y \times]\partial_\tau \uptheta) = y \partial^\nu (\partial_\tau (\Lambda \mathscr{A}) \partial_\tau \uptheta) + \mathscr{R},
\end{equation}
where $\mathscr{R}$ are favorable remainder terms. Then
\begin{align}
&\| \int_0^\tau \frac{1}{\mu} \int_0^{\tau'} \mu y \partial^\nu (\partial_\tau (\Lambda \mathscr{A}) \partial_\tau \uptheta) d \tau'' d \tau' \|^2 \notag \\
& \lesssim \int_{\mathbb{R}^3} | \int_0^\tau \frac{(1+K \tau')(1+\tau')}{\mu^{1/2}} \frac{1}{(1+\tau')\mu^{1/2}} \int_0^{\tau'} \mu \partial^\nu (\partial_\tau (\Lambda \mathscr{A}) \partial_\tau \uptheta) d \tau'') d \tau' | dy \notag \\
& \lesssim \int_{\mathbb{R}^3} [ (\int_0^\tau \frac{(1+K \tau')^2(1+\tau')^2}{\mu} d \tau') \int_0^\tau \frac{1}{(1+\tau')^2 \mu} (\int_0^{\tau'} \mu \partial^{\nu} (\partial_\tau (\Lambda \mathscr{A}) \partial_\tau \uptheta) d \tau'')^ 2 ] dy \notag \\
& \lesssim \int_0^\tau \frac{1}{(1+\tau')^2 \mu} \| \int_0^\tau \mu \partial^\nu (\partial_\tau (\Lambda \mathscr{A}) \partial_\tau \uptheta) d \tau''\|^2 d \tau'.
\end{align}
Then by a similar argument to other estimates,
\begin{equation}
\int_0^\tau e^{-\mu_0 \tau'} \mathcal{S}^N d \tau'
\gtrsim
\begin{dcases}
\|\frac{1}{1+\alpha} \int_0^{\tau} \frac{1}{\mu(\tau')} \int_0^{\tau'} \mu(\tau '') \partial^\nu([\partial_\tau,\Lambda \mathscr{A} y \times ] \mathbf{V} ) d\tau '' d \tau '\|^2  &\text{if } \ |\nu| \leq N-1,  \\
\mu^{-\delta} \|\frac{1}{1+\alpha} \int_0^{\tau} \frac{1}{\mu(\tau')} \int_0^{\tau'} \mu(\tau '') \partial^\nu([\partial_\tau,\Lambda \mathscr{A} y \times ] \mathbf{V} ) d\tau '' d \tau '\|^2 &\text{if } \ |\nu| = N.
\end{dcases}
\end{equation}
For the nineteenth term on the right hand side of (\ref{E:CURLTHETAHIGHORDERDERV}), $\frac{\overline{C}}{1+\alpha} \int_0^\tau \frac{1}{\mu} \int_0^{\tau'} \mu^{1-\delta-\sigma} \partial^\nu( \Lambda y \times  \Lambda \uptheta) d \tau '' d \tau ' $, we fix $0<X<1-\frac{\sigma}{4}$ sufficiently small so that $\frac{\sigma}{4}+\frac{1}{2}+\frac{\delta}{2}-X > 0$. Then fix $0<Y<\frac{\sigma}{4}+\frac{1}{2}+\frac{\delta}{2}-X$. Then
\begin{align}
&\| \int_0^\tau \frac{\overline{C}}{2 \mu} \int_0^{\tau'} \mu^{1-\delta-\sigma} (\partial^\nu (\Lambda y \times \Lambda \uptheta)) d \tau'' d \tau' \|^2 \notag \\
& \lesssim \int_{\mathbb{R}^3} ( \int_0^\tau \frac{1}{\mu^{2X}} d \tau')\int_0^\tau \frac{1}{\mu^{2-2X}} (\int_0^{\tau'} \mu^{1-\delta-\sigma} (\partial^\nu (\Lambda y \times \Lambda \uptheta)) d \tau'')^2 d\tau' dy \notag \\
& \lesssim \int_0^\tau \frac{1}{\mu^{2-2X}} \| \int_0^{\tau'} \mu^{1-\delta-\sigma} \partial^{\nu}(\Lambda y \times \Lambda \uptheta) d \tau'' \|^2 d \tau' \notag \\
&\lesssim \int_0^\tau e^{-\mu_0 \tau'} \mathcal{S}^N,
\end{align}
where we conclude the final bound by a similar argument to other estimates since
\begin{align}
\frac{1}{\mu^{2-2X}} |\ \int_0^{\tau}  \mu^{1-\delta-\sigma} (1+K \tau)^2 d \tau|^2 & \lesssim e^{ (2X-2) \mu_1 \tau } \Big| \int_0^{\tau} e^{(1-\delta-\sigma+Y)\mu_1 \tau'} d \tau' \Big|^2 \notag \\
&\lesssim e^{(2X-2)\mu_1\tau} (e^{2(1-\delta-\sigma+Y) \mu_1 \tau}+1) \notag \\
& \lesssim e^{- \mu_0\tau}.
\end{align}   
Hence
\begin{equation}
\int_0^\tau e^{-\mu_0 \tau'} \mathcal{S}^N d \tau'
\gtrsim
\begin{dcases}
\| \frac{\overline{C}}{1+\alpha} \int_0^\tau \frac{1}{\mu} \int_0^{\tau'} \mu^{1-\delta-\sigma} \partial^\nu( \Lambda y \times  \Lambda \uptheta) d \tau '' d \tau ' \|^2  &\text{if } \ |\nu| \leq N-1  \\
\mu^{-\delta} \|\frac{\overline{C}}{1+\alpha} \int_0^\tau \frac{1}{\mu} \int_0^{\tau'} \mu^{1-\delta-\sigma} \partial^\nu( \Lambda y \times  \Lambda \uptheta) d \tau '' d \tau ' \|^2 &\text{if } \ |\nu| = N
\end{dcases}
\end{equation}
For the twentieth term on the right hand side of (\ref{E:CURLTHETAHIGHORDERDERV}), $-\frac{\overline{C}}{1+\alpha} \int_0^\tau \frac{1}{\mu}\int_0^{\tau'} \mu^{1-\delta-\sigma} \partial^\nu(\Lambda \mathscr{A}[D \uptheta]y \times \Lambda \eta) d \tau '' d \tau ' $, first
\begin{equation}
\Lambda \mathscr{A} [D \uptheta] y \times \Lambda \eta = \Lambda \mathscr{A} [D \uptheta] y \times \Lambda \uptheta+ \Lambda \mathscr{A} [D \uptheta] y \times \Lambda y.
\end{equation}
Then a similar argument to above holds except at top order $|\nu|=N$. At top order order, $|\nu|=N$, note, with $X$ defined as above,
\begin{align}
&\mu^{-{\delta}} \| \int_0^\tau \frac{\overline{C}}{2\mu}  \int_0^{\tau'} \mu^{1-\delta-\sigma} (\partial^\nu (\Lambda y \times \Lambda \uptheta)) d \tau'' d \tau' \|^2 \notag \\
&\lesssim \mu^{-{\delta}} \int_{\mathbb{R}^3} [\int_0^\tau \frac{1}{\mu(\tau')^{2X}} d \tau' \int_0^\tau \frac{1}{\mu(\tau')^{2-2X}} (\int_0^{\tau'} \mu^{1-\delta-\sigma} (\partial^\nu (\Lambda y \times \Lambda \uptheta)) d \tau'')^2 d \tau' dy \notag \\
&\lesssim \sup_{0 \leq \tau' \leq \tau} [\mu(\tau')^{-\delta}] \int_0^\tau \frac{\mu^{-\delta}}{\mu^{-\delta}\mu^{2-2X}} \| \int_0^{\tau'} \mu^{1-\delta-\sigma} (\partial^\nu (\Lambda y \times \Lambda \uptheta)) d \tau''\|^2 d \tau' \notag \\
&\lesssim \sup_{0 \leq \tau' \leq \tau} [\mu(\tau')^{-\delta}] \sup_{0 \leq \tau' \leq \tau} [ \frac{1}{\mu(\tau')^{-\delta}}] \int_0^{\tau} \frac{1}{\mu^{2-2X}} \mu^{-\delta}  \| \int_0^{\tau'} \mu^{1-\delta-\sigma} (\partial^\nu (\Lambda y \times \Lambda \uptheta)) d \tau''\|^2 d \tau' \notag \\
&=\int_0^{\tau} \frac{1}{\mu^{2-2X}} \mu^{-\delta}  \| \int_0^{\tau'} \mu^{1-\delta-\sigma} (\partial^\nu (\Lambda y \times \Lambda \uptheta)) d \tau''\|^2 d \tau' \notag \\
&\lesssim \int_0^\tau e^{-\mu_0 \tau'} \mathcal{S}^N d \tau',
\end{align}
where the last bound follows from a similar argument to above, where we include $\mu^{-\delta}$ in $\mathcal{S}^N$ if needed. Hence
\begin{equation}
\int_0^\tau e^{-\mu_0 \tau'} \mathcal{S}^N d \tau'
\gtrsim
\begin{dcases}
\|\frac{\overline{C}}{1+\alpha} \int_0^\tau \frac{1}{\mu}\int_0^{\tau'} \mu^{1-\delta-\sigma} \partial^\nu(\Lambda \mathscr{A}[D \uptheta]y \times \Lambda \eta) d \tau '' d \tau '\|^2  &\text{if } \ |\nu| \leq N-1,  \\
\mu^{-\delta} \|\frac{\overline{C}}{1+\alpha} \int_0^\tau \frac{1}{\mu}\int_0^{\tau'} \mu^{1-\delta-\sigma} \partial^\nu(\Lambda \mathscr{A}[D \uptheta]y \times \Lambda \eta) d \tau '' d \tau '\|^2 &\text{if } \ |\nu| = N.
\end{dcases}
\end{equation}
Finally for the last term on the right hand side of (\ref{E:CURLTHETAHIGHORDERDERV}), using a combination of the $H^2 (\mathbb{R}^3) \hookrightarrow L^\infty (\mathbb{R}^3)$ embedding and $\| \beta \|^2_{H^{N+1,\infty}(\mathbb{R}^3)} \leq \lambda$, a similar argument, without applying finite propagation, to other estimates and then (\ref{E:SOURCETERMHIGHORDERENERGYESTIMATE}) gives
\begin{align}
& \lambda + \int_0^\tau e^{-\mu_0 \tau'} \mathcal{S}^N d \tau' \notag \\
& \gtrsim
\begin{dcases}
\|\frac{\alpha \overline{C}}{1+\alpha} \int_0^{\tau} \frac{1}{\mu} \int_0^{\tau'} \mu^{1-\delta-\sigma} \partial^\nu ((1+\beta)^{-1} (\Lambda \mathscr{A} \nabla \beta \times \Lambda y)) d \tau'' d\tau'\|^2  &\text{if } \ |\nu| \leq N-1,  \\
\mu^{-\delta} \|\frac{\alpha \overline{C}}{1+\alpha} \int_0^{\tau} \frac{1}{\mu} \int_0^{\tau'} \mu^{1-\delta-\sigma} \partial^\nu ((1+\beta)^{-1} (\Lambda \mathscr{A} \nabla \beta \times \Lambda y)) d \tau'' d\tau'\|^2 &\text{if } \ |\nu| = N.
\end{dcases}
\end{align}
This concludes the proof of (\ref{E:CURLTHETABOUNDGAMMAGREATER5OVER3}).
\end{proof}

\section{Energy Inequality and Proof of the Main Theorem}\label{S:CONTINUITY}
By the Local Well-Posedness Theorem \ref{T:LWPGAMMALEQ5OVER3} there exists a unique solution to (\ref{E:THETAEQNLINEARENERGYFUNCTION}) on some time interval $[0,T^*], T^*>0.$ From Propositions \ref{P:ENERGYESTIMATE} and \ref{P:CURLESTIMATES} with $\kappa>0$ chosen small enough and by using the equivalence of the norm $\mathcal{S}^N$ and the modified energy $\mathcal{E}^N+\mathcal{C}^{N-1}$ as well as a similar argument to (\ref{E:SOURCETERMHIGHORDERENERGYESTIMATE}), we conclude that there exist universal constants $c_1,c_2,c_3,c_4\geq 1$ such that for any $0 \leq \tau^* \leq \tau \leq T$
\begin{equation}\label{E:ENERGYBOUND}
\mathcal{S}^N(\tau;\tau^*) \leq c_1 \mathcal{S}^N(\tau^*)+c_2 \lambda +c_3 ( \mathcal{S}^N(0)+\mathcal{B}^N[{\bf V}](0) ) + c_4 \int_{\tau^*}^\tau e^{-\frac{\mu_0}{2} \tau'} \mathcal{S}^N (\tau';\tau^*) d \tau'.
\end{equation}
Here $\mathcal{S}^N(\tau;\tau^*)$ denotes the sliced norm of $\mathcal{S}^N$ from $\tau^*$ to $\tau$ with $\sup_{0\le \tau' \le\tau}$ replaced by $\sup_{\tau^*\le \tau' \le\tau}\,$. By a standard well-posedness estimate, we deduce that the time of existence $T$ is inversely proportional to the size of the initial data, i.e.: $T\sim (\mathcal{S}^N(0)+\mathcal{B}^N[{\bf V}](0))^{-1}$. Choose $\varepsilon>0$ so small that the time of existence $T$ satisfies 
\begin{equation}\label{E:LOCALTIME}
e^{-\mu_0 T/4} \leq \frac{\kappa \mu_0}{c_4}, \ \ \sup_{\tau \leq T}\mathcal{S}^N(\tau) \leq c\left(\mathcal{S}^N(0)+\mathcal{B}^N[{\bf V}](0)+\lambda \right)
\end{equation}
where $c>0$ is a universal constant provided by the local-in-time well-posedness theory. Let
\begin{equation}
C_\ast= 3(c_1 c + c_2 + c_3).
\end{equation}
Define 
\begin{equation}
\mathcal T : = \sup_{\tau\geq0}\{\text{ solution to (\ref{E:THETAEQNLINEARENERGYFUNCTION}) exists on $[0,\tau)$ and} \ \mathcal{S}^N(\tau)\leq  C_\ast \left(\mathcal{S}^N(0)+\mathcal{B}^N[{\bf V}](0)+\lambda\right)\} . 
\end{equation}
Observe that $\mathcal T\geq T$ since $c \leq C_\ast$. Letting $\tau^*=\frac T2$ in (\ref{E:ENERGYBOUND}) for any $\tau\in[\frac T2,\mathcal T]$ we obtain
\begin{equation}
\mathcal{S}^N(\tau; \frac T2) \leq c_1 \mathcal{S}^N(\frac T2)+c_2 \lambda + c_3\left(\mathcal{S}^N(0)+\mathcal{B}^N[{\bf V}](0)\right) +  c_4 \int_{\frac T2}^\tau e^{-\frac{\mu_0}{2}\tau'}\mathcal{S}^N(\tau';\frac{T}{2})\,d\tau'.
\end{equation}
Therefore, using (\ref{E:LOCALTIME}) we conclude that for any $\tau\in[\frac T2,\mathcal T]$
\begin{align}
\mathcal{S}^N(\tau; \frac T2) & \leq  c_1\mathcal{S}^N(\frac T2)+c_2 \lambda+c_3\left(\mathcal{S}^N(0)+\mathcal{B}^N[{\bf V}](0)\right)   + \frac{c_4}{\mu_0}e^{-\mu_0 T/4}\mathcal{S}^N(\tau;\frac{T}{2}) \notag \\
& \le c_1\mathcal{S}^N(\frac T2)+c_2 \lambda+c_3\left(\mathcal{S}^N(0)+\mathcal{B}^N[{\bf V}](0)\right)   +\kappa\mathcal{S}^N(\tau;\frac{T}{2}). \label{E:CONT1}
\end{align}
Since by (\ref{E:LOCALTIME}), $\mathcal{S}^N(\frac T2)\leq c\left(\mathcal{S}^N(0)+\mathcal{B}^N[{\bf V}](0)+\lambda \right)$ we conclude from (\ref{E:CONT1}) that
\begin{equation}
\mathcal{S}^N(\tau; \frac T2) \leq c_1 c \left( \mathcal{S}^N(0)+\mathcal{B}^N[{\bf V}](0) + \lambda \right) + c_2 \lambda+c_3\left(\mathcal{S}^N(0)+\mathcal{B}^N[{\bf V}](0)\right) +\kappa\mathcal{S}^N(\tau;\frac{T}{2})
\end{equation}
which for sufficiently small $\kappa$ gives
\begin{equation}
\mathcal{S}^N(\tau; \frac T2) \leq 2(c_1 c + c_2 + c_3) \left( \mathcal{S}^N(0)+\mathcal{B}^N[{\bf V}](0) + \lambda \right) < C_\ast \left(\mathcal{S}^N(0)+\mathcal{B}^N[{\bf V}](0) + \lambda\right),
\end{equation}
and hence
\begin{equation}\label{E:SNCSTARBOUND}
\mathcal{S}^N(\tau)< C_\ast \left(\mathcal{S}^N(0)+\mathcal{B}^N[{\bf V}](0)+ \lambda\right).
\end{equation}
It is now easy to check the a priori bounds in (\ref{E:APRIORI}) are in fact improved. 
For instance, by the fundamental theorem of calculus 
\begin{equation}
\|D\uptheta\|_{W^{1,\infty}} = \|\int_0^\tau D{\bf V}\|_{W^{1,\infty}} \le \int_0^\tau e^{-\mu_0\tau'} \mathcal{S}^N(\tau')\,d\tau' \lesssim \varepsilon <\frac16, \ \ \tau\in[0,\mathcal T) 
\end{equation}
for $\varepsilon>0$ small enough. Similar arguments apply to the remaining a priori assumptions. 

From the continuity of the map $\tau\mapsto \mathcal{S}^N(\tau')$ and the definition of $\mathcal T$ we conclude that $\mathcal T = \infty$ and the solution to (\ref{E:THETAEQNLINEARENERGYFUNCTION}) exists globally-in-time.  Furthermore, also using Proposition \ref{P:CURLESTIMATES}, the global bound (\ref{E:GLOBALBOUND}) follows.

\section*{Acknowledgments}
C. Rickard acknowledges the support of the NSF grant DMS-1608494 and the NSF grant DMS-1613135.

\appendix
\renewcommand{\theequation}{\Alph{section}.\arabic{equation}}
\setcounter{theorem}{0}\renewcommand{\theorem}{\Alph{section}.\??arabic{prop}}

\section{Energy Identities}
We give key identities which will be used in our estimates. First from Lemma 4.3 \cite{1610.01666} we have the following modified energy identity:
\begin{lemma} \label{L:KEYLEMMA}
Recalling the matrix quantities introduced in (\ref{E:LAMBDADECOMP}) and (\ref{E:NNU}), the following identity holds
\begin{equation}
 \Lambda_{\ell j}(\nabla_\eta\partial^\nu\uptheta)^i_j\Lambda^{-1}_{im}(\nabla_\eta\partial^\nu\uptheta)^m_{\ell,\tau}  =  \frac12\frac{d}{d\tau}\left(\sum_{i,j=1}^3 d_id_j^{-1}([\mathscr N_\nu]^j_{i})^2\right) + \mathcal T_{\nu}. \label{E:KEYTWO}
\end{equation}
where the error term $\mathcal T_{\nu}$ is given as follows
\begin{equation}
\mathcal T_{\nu} = -  \frac12\sum_{i,j=1}^3\frac d{d\tau}\left(d_id_j^{-1}\right)([\mathscr N_\nu]^j_i)^2 
 -  \text{{\em Tr}}\left(Q\mathscr N_\nu Q^{-1}\left(\partial_\tau P P^\top \mathscr N_\nu^\top + \mathscr N_\nu^\top P\partial_\tau P^\top\right)\right). \label{E:TNU}
\end{equation}
\end{lemma}
In the next Lemma, we give some useful results concerning our quantities $\mathscr{A}$, $\mathscr{J}$ and $\Lambda$
\begin{lemma}\label{L:USEFULIDENTITIESPREENERGYINEQUALITY}       
For $\mathscr{A}$, $\mathscr{J}$ and $\Lambda$, the following identities hold
\begin{align}
\mathscr{A}^k_j\mathscr{J}^{-\frac1\alpha} - \delta^k_j &= ( \mathscr{A}^k_j-\delta^k_j) \mathscr{J}^{-\frac1\alpha} + \delta^k_j (\mathscr{J}^{-\frac1\alpha} -1),  \label{E:AJIDENTITYENERGY} \\  
\mathscr{A}^k_j-\delta^k_j &= -\mathscr{A}^k_l  [D\uptheta]^l_j,  \label{E:AIDENTITYENERGY}\\
\Lambda_{ij}&=\Lambda_{ip}(\mathscr{A}^j_p +  \mathscr{A}^j_l\uptheta^l,_p), \label{E:LAMBDAIDENTITYENERGY} \\
1-\mathscr{J}^{-\frac1\alpha}&=\frac{1}{\alpha}\text{\emph{Tr}}[D\uptheta ] + O(|D\uptheta|^2), \label{E:JIDENTITYENERGY} \\
\mathscr{A}_j^k \mathscr{J}^{-\frac{1}{\alpha}} -\delta_j^k&=-\mathscr{A}_{\ell}^k[D \uptheta]^{\ell}_j \mathscr{J}^{-\frac{1}{\alpha}}-\delta_j^k(\frac{1}{\alpha}\text{\emph{Tr}}[D \uptheta ] + O(|D\uptheta|^2)). \label{E:EXPANDAJMINUSID}
\end{align}
\end{lemma}
\begin{proof}
First (\ref{E:AJIDENTITYENERGY}) is straightforward to verify by expanding the right-hand side. For (\ref{E:AIDENTITYENERGY}), first note $\mathscr{A}=[D \eta]^{-1}$ and $\eta=y+\uptheta$. We then have
\begin{equation}
\mathscr{A}^k_j-\delta^k_j = \mathscr{A}^k_l  \delta^l_j  -\mathscr{A}^k_l [D\eta]^l_j  = \mathscr{A}^k_l (\delta^l_j - [D  y]^l_j - [D\uptheta]^l_j) = -\mathscr{A}^k_l  [D\uptheta]^l_j.
\end{equation}
Next, (\ref{E:LAMBDAIDENTITYENERGY}) is proven in the following calculation where we recall $A=[D \eta]^{-1}$ and use $\eta=y+\uptheta$,
\begin{equation}  
\Lambda_{ij}=\Lambda_{ip}\delta^j_p = \Lambda_{ip} \mathscr{A}^j_l\eta^l,_p=\Lambda_{ip}(\mathscr{A}^j_p +  \mathscr{A}^j_l\uptheta^l,_p).
\end{equation}
Now (\ref{E:JIDENTITYENERGY}) follows from the calculation
\begin{equation}
\mathscr{J}=\det [D \eta ] = \det [\textbf{Id} + D\uptheta ] = 1 + \text{Tr}[D \uptheta ] + O(|D\uptheta|^2).
\end{equation}
Finally (\ref{E:EXPANDAJMINUSID}) follows from (\ref{E:AJIDENTITYENERGY})-(\ref{E:AIDENTITYENERGY}) and (\ref{E:JIDENTITYENERGY}).
\end{proof}

\section{Time Based Inequalities}
We have the following useful $\tau$ based inequalities, summarized by the Lemma below.
\begin{lemma}\label{L:USEFULTAULEMMAGAMMALEQ5OVER3}
Fix an affine motion $A(t)$ from the set $\mathscr{S}$ under consideration, namely require
\begin{equation}
\det A(t) \sim 1 + t^3, \quad t \geq 0.
\end{equation}  
Let
\begin{equation}
\mu_1:=\lim_{\tau \rightarrow \infty} \frac{\mu_\tau(\tau)}{\mu(\tau)}, \quad \mu_0:=\frac{\sigma}{2}\mu_1.
\end{equation}          
Then we have the following properties
\begin{align}
0<\mu_0&=\mu_0(\gamma)\le\mu_1, \label{E:MU0MU1INEQGAMMALEQ5OVER3}\\
e^{\mu_1\tau} \lesssim & \ \mu(\tau)  \lesssim  e^{\mu_1\tau}, \ \ \tau\ge0, \label{E:EXPMU1MUINEQGAMMALEQ5OVER3}\\
\sum_{a+|\beta|\le N}\|X_r^a \slashed\partial^\beta{\bf V}\|_{1+\alpha+a,\psi e^{\bar{S}}} &+ \sum_{|\nu|\leq N} \|\partial^\nu {\bf V}\|_{1+\alpha,(1-\psi)e^{\bar{S}}} \lesssim e^{-\mu_0\tau}\mathcal{S}^N(\tau)^{\frac12} \label{E:EXPSQRTSNBOUNDGAMMALEQ5OVER3}\\
\| \Lambda_\tau \| \lesssim e^{-\mu_1 \tau}, & \quad  \| \Lambda \| + \| \Lambda^{-1} \| \le C,  \label{E:LAMBDABOUNDSGAMMALEQ5OVER3}\\
\sum_{i=1}^3\left(d_i+\frac{1}{d_i}\right) &\le C \label{E:EIGENVALUESBOUNDGAMMALEQ5OVER3} \\
\sum_{i=1}^3|\partial_\tau d_i| + \|\partial_\tau P\| &\lesssim e^{-\mu_1\tau} \label{E:EIGENVALUEPLUSPBOUNDGAMMALEQ5OVER3}\\
|{\bf w}|^2 \lesssim & \ \langle \Lambda^{-1}{\bf w}, {\bf w}\rangle  \lesssim |{\bf w}|^2 , \ {\bf w}\in \mathbb R^3, \label{E:LAMBDAINVERSEIPINEQGAMMALEQ5OVER3}
\end{align}
for $C > 0$.       
\end{lemma} 
\begin{proof}
The result (\ref{E:MU0MU1INEQGAMMALEQ5OVER3}) is clear from the definition of $\mu_0$. For inequalities (\ref{E:EXPMU1MUINEQGAMMALEQ5OVER3}) and (\ref{E:LAMBDABOUNDSGAMMALEQ5OVER3}) through (\ref{E:LAMBDAINVERSEIPINEQGAMMALEQ5OVER3}) we first note that by Lemma \ref{L:AASYMPTOTICS}, there exist matrices $A_0,A_1,M(t)$ such that          
\begin{align}  
A(t) = A_0 + t A_1 + M(t), \quad t \geq 0.
\end{align} 
where $A_0,A_1$ are time-independent and $M(t)$ satisfies the bounds
\begin{align} 
\|M(t)\| = o_{t \rightarrow \infty}(1+t), \ \ \|\partial_t M(t)\| \lesssim (1+t)^{3-3\gamma}.
\end{align}          
We also note $\det A(t) \sim 1 + t^3$. Then inequalities (\ref{E:EXPMU1MUINEQGAMMALEQ5OVER3}) and (\ref{E:LAMBDABOUNDSGAMMALEQ5OVER3}) through (\ref{E:LAMBDAINVERSEIPINEQGAMMALEQ5OVER3}) follow from Lemma A.1 \cite{1610.01666}. Finally, (\ref{E:EXPSQRTSNBOUNDGAMMALEQ5OVER3}) follows from the definition of $\mathcal{S}^N$ (\ref{E:SNNORMGAMMALEQ5OVER3}) and properties (\ref{E:MU0MU1INEQGAMMALEQ5OVER3})-(\ref{E:EXPMU1MUINEQGAMMALEQ5OVER3}) above.
\end{proof}

\section{Local Well Posedness}\label{A:LWP}
We construct a local solution for our original Euler system (\ref{E:CONTINUITY})-(\ref{E:EEOS}) from generic initial data. First write the equations (\ref{E:CONTINUITY})-(\ref{E:EE}) in terms of $(\rho,\mathbf{u},T)$ as follows where we use the material derivative $\frac{D}{Dt}=\partial_t + \mathbf{u} \cdot \nabla$ and we have multiplied by $\text{diag}(T^2, \rho^2 T I_3,\alpha \rho^2)$ where $I_3$ is the $3 \times 3$ identity matrix,
\begin{alignat}{2}
T^2 \frac{D \rho}{D t} + \rho T^2 \,\text{div}(\mathbf{u}) & = 0 \label{E:RHOSYMMETRIC1} \\
\rho^2 T \frac{D \mathbf{u}}{D t} + \rho T^2 \, \nabla \rho +  \rho^2 T  \, \nabla T & = 0 \label{E:USYMMETRIC1} \\
\alpha \rho^2 \frac{D T}{D t} + \rho^2 T \, \text{div}(\mathbf{u}) & = 0 \label{E:TSYMMETRIC1}
\end{alignat}
Let $(*)$ denote the symmetric $5 \times 5$ system of equations (\ref{E:RHOSYMMETRIC1})-(\ref{E:TSYMMETRIC1}).

We first note a fixed affine solution $(\rho_A,\mathbf{u}_A,T_A)$ solves $(*)$ with initial data 
\begin{equation}
((\rho_A)_0,(\mathbf{u}_A)_0,(T_A)_0)=(\rho_A(0,\cdot),\mathbf{u}_A(0,\cdot),T_A(0,\cdot)).
\end{equation}

Next we construct initial data, using generic initial data $(\rho_0,\mathbf{u}_0,T_0)$ and our affine initial data $((\rho_A)_0,(\mathbf{u}_A)_0,(T_A)_0)$, for a modified system which allows to us to avoid the fact that we want to have $\rho_0 \rightarrow 0$ for large $x$. To this end choose $(\sigma_0^{\rho},\sigma_0^{\mathbf{u}},\sigma_0^{T})$ as follows
\begin{align}
\sigma_0^{\rho}&=\varphi (\rho_0-C_1) + (1-\varphi)\psi ((\rho_A)_0-C_1)+(1-\psi) \rho_0,  \\
\sigma_0^{\mathbf{u}}&=\varphi \mathbf{u}_0 + (1-\varphi) (\mathbf{u}_A)_0,  \\
\sigma_0^{T}&=\varphi T_0 + (1-\varphi) (T_A)_0,
\end{align}
where $\varphi, \psi  \in C_c^\infty (\mathbf{R}^3)$ are such that, for fixed $R>0$ and $\eta >0$,  $\varphi=1 \text{ on } B(0,\frac{R}{2})$, $\varphi = 0 \text{ on } \mathbb{R}^3 \setminus B(0,R)$, $\psi=1 \text{ on } B(0,R+2\eta)$ and $\psi = 0 \text{ on } \mathbb{R}^3 \setminus B(0,R+3\eta)$, and $C_1 >0$ is such that $\| \rho_0 \|_{L^\infty(\mathbb{R}^3)} \leq \frac{C_1}{2}$. Note the existence of $C_1>0$ will be guaranteed through the regularity of $\rho_0$. Then by construction we will have $\inf_{\mathbb{R}^3} \{ \sigma_0^{\rho} + C_1 \} > 0$. Now we consider $(\sigma_0^{\rho},\sigma_0^{\mathbf{u}},\sigma_0^{T})$ as initial data for the modified system
\begin{alignat}{2}
T^2 \frac{D \rho}{D t} + (\rho+C_1) T^2 \,\text{div}(\mathbf{u}) & = 0 \label{E:RHOSYMMETRIC1MOD} \\
(\rho+C_1)^2 T \frac{D \mathbf{u}}{D t} + (\rho+C_1) T^2 \, \nabla \rho +  (\rho+C_1)^2 T  \, \nabla T & = 0 \label{E:USYMMETRIC1MOD} \\
\alpha (\rho+C_1)^2 \frac{D T}{D t} + (\rho+C_1)^2 T \, \text{div}(\mathbf{u}) & = 0 \label{E:TSYMMETRIC1MOD}
\end{alignat}
Let $(\sigma \text{-} *)$ denote the symmetric $5 \times 5$ system of equations (\ref{E:RHOSYMMETRIC1MOD})-(\ref{E:TSYMMETRIC1MOD}). Then, with sufficiently regular $(\rho_0,\mathbf{u}_0,T_0)$ which will be specified by the Lagrangian formulation, by Theorem II \cite{kato1975cauchy} there exists $\hat{T} > 0$ such that $(\sigma^{\rho},\sigma^{\mathbf{u}},\sigma^{T})$ is a solution to $(\sigma\text{-}*)$ with initial data $(\sigma_0^{\rho},\sigma_0^{\mathbf{u}},\sigma_0^{T})$. Now let
\begin{equation}
(\rho_B,\mathbf{u}_B,T_B)=(\sigma^{\rho}+C_1,\sigma^{\mathbf{u}},\sigma^{T}).
\end{equation}
Since $(\sigma^{\rho},\sigma^{\mathbf{u}},\sigma^{T})$ solve $(\sigma \text{-} *)$ we have that $(\rho_B,\mathbf{u}_B,T_B)$ solve $(*)$ with initial data $((\rho_B)_0,(\mathbf{u}_B)_0,(T_B)_0)=(\sigma_0^\rho+C_1,\sigma_0^{\mathbf{u}},\sigma_0^T)$.

Next let
\begin{align}
K&=\{(x,t) \, | \, 0 \leq t \leq T^1, x \in B(0,R+\eta+Mt) \} \notag \\
M&=\tfrac{3}{c} (\sup_{0 \leq t \leq T^1-\kappa} \{ \| \rho_B (T_B)^2  \|_{L^\infty(B(0,R+2\eta))} + \| (\rho_B)^2 T_B \mathbf{u}_B \|_{L^\infty(B(0,R+2\eta))} \notag \\
&\qquad \qquad \qquad + \| (\rho_{B})^2 T_{B} \|_{L^\infty(B(0,R+2\eta))}  + \| \rho_A (T_A)^2  \|_{L^\infty(B(0,R+2\eta))} \notag \\
&\qquad \qquad \qquad + \| (\rho_A)^2 T_A \mathbf{u}_A \|_{L^\infty(B(0,R+2\eta))} + \| (\rho_A)^2 T_A \|_{L^\infty(B(0,R+2\eta))} \}) \notag \\
c&>0 \text{ is such that } c \leq \sup_{0 \leq t \leq T^*-\kappa} \{ \|(T_{B})^2  \|_{L^\infty(\mathbb{R}^3)} + \| (\rho_{B})^2 T_{B} \|_{L^\infty(\mathbb{R}^3)} + \| \alpha (\rho_{B})^2  \|_{L^\infty(\mathbb{R}^3)} \}, \notag \\
T^1&=\min (\hat{T}-\kappa, \frac{\eta}{2 M} - \kappa), \notag \\
\kappa &> 0 \text{ is sufficiently small.}
\end{align}
Now we take
\begin{equation}
(\rho,\mathbf{u},T)=\begin{cases} (\rho_{B},\mathbf{u}_{B},T_{B}) & \text{in } K, \\
(\rho_A,\mathbf{u}_A,T_A) & \text{outside } K.
\end{cases}
\end{equation}
Then $(\rho,\mathbf{u},T)$ is a solution of $(*)$ on $\mathbb{R}^3 \times [0,T^1]$ with the following initial data
\begin{align}
\rho_0 &=\varphi (\rho_B)_0+ (1-\varphi) (\rho_A)_0 ,   \notag \\
\mathbf{u}_0 &=\varphi (\mathbf{u}_B)_0 + (1-\varphi) (\mathbf{u}_A)_0, \notag \\
T_0 &=\varphi (T_B)_0 + (1-\varphi)(T_A)_0,
\end{align}
since $(\rho,\mathbf{u},T)$ is a solution in $K$ and outside $K$, and applying a classical property of local uniqueness of solutions to $(*)$ to get that $(\rho,\mathbf{u},T)$ is continuous across $\partial K$.

\section{Curl Equations Derivation}\label{A:CURLEQS}
Here we give the derivations of the equations satisfied by the modified curl of our velocity and perturbation which will be used for the purpose of our estimates. 
\begin{lemma}\label{L:CURLEQUATIONDERIVATION}
Let $(\uptheta, {\bf V}):\Omega \rightarrow \mathbb R^3\times \mathbb R^3$ be a unique local solution to (\ref{E:THETAEQNLINEARENERGYFUNCTION})-(\ref{E:THETAICGAMMALEQ5OVER3}) on $[0,T^*]$ for $T^*>0$ fixed. Then for all $\tau \in [0,T]$ the curl matrices $\text{\em Curl}_{\Lambda\mathscr{A}}{\bf V}$ and $\text{\em Curl}_{\Lambda\mathscr{A}}{\bf \uptheta}$ satisfy the equations
\begin{align}\label{E:CURLVFINAL}
&\text{\em Curl}_{\Lambda\mathscr{A}}{\bf V}  = \frac{1}{1+\alpha}\Lambda\mathscr{A} y \times \mathbf{V} + \frac{\alpha}{(1+\alpha)(1+\beta)} \Lambda\mathscr{A} \nabla \beta \times \mathbf{V} \notag \\
&+  \frac{\mu(0) \text{\em Curl}_{\Lambda \mathscr{A}} ({\bf V}(0))}{\mu} - \frac{\mu(0) \Lambda\mathscr{A} y \times \mathbf{V}(0)}{(1+\alpha) \mu} - \frac{\alpha \mu(0) \Lambda\mathscr{A} \nabla \beta \times \mathbf{V}(0)}{(1+\alpha)(1+\beta) \mu} \notag \\
& + \frac{1}{\mu}\int_0^{\tau} \mu [\partial_{\tau}, \text{\em Curl}_{\Lambda\mathscr{A}}] {\bf V} d\tau' - \frac{1}{(1+\alpha) \mu} \int_0^{\tau} \mu [\partial_{\tau},\Lambda \mathscr{A} y \times ] \mathbf{V} d \tau' \notag \\
&- \frac{\alpha}{(1+\alpha)(1+\beta) \mu} \int_0^{\tau} \mu [\partial_{\tau},\Lambda \mathscr{A} \nabla \beta \times ] \mathbf{V} d \tau' \notag \\
& - \frac{2}{\mu} \int_0^{\tau} \mu \, \text{\em Curl}_{\Lambda\mathscr{A}}(\Gamma^\ast{\bf V}) d\tau' + \frac{2}{(1+\alpha) \mu} \int_0^{\tau} \mu \, \Lambda\mathscr{A} y \times (\Gamma^\ast \mathbf{V}) d \tau'  \notag \\
&+ \frac{2 \alpha}{(1+\alpha)(1+\beta) \mu} \int_0^{\tau} \mu \, \Lambda\mathscr{A} \nabla \beta \times (\Gamma^\ast \mathbf{V}) d \tau' \notag \\
&+\frac{\overline{C}}{(1+\alpha) \mu}\int_0^{\tau} \mu^{1-\delta-\sigma} \Lambda y \times  \Lambda \uptheta d \tau' -\frac{\overline{C}}{(1+\alpha) \mu} \int_0^{\tau} \mu^{1-\delta-\sigma} \Lambda \mathscr{A}[D \uptheta]y \times \Lambda \eta d \tau' \notag \\
&+\frac{\alpha \overline{C}}{(1+\alpha)(1+\beta) \mu}\int_0^{\tau} \mu^{1-\delta-\sigma} \Lambda \mathscr{A} \nabla \beta \times \Lambda \uptheta d \tau' -\frac{\alpha \overline{C}}{(1+\alpha)(1+\beta) \mu}\int_0^{\tau} \mu^{1-\delta-\sigma} \Lambda \mathscr{A} \nabla \beta \times \Lambda y d \tau'.
\end{align}
and          
\begin{align}\label{E:CURLTHETAFINAL}
&\text{\em Curl}_{\Lambda\mathscr{A}}{\bf \uptheta}  = \text{\em Curl}_{\Lambda \mathscr{A}}([\uptheta(0)]) \notag \\
&+\mu(0) \text{\em Curl}_{\Lambda \mathscr{A}} ({\bf V}(0)) \int_0^\tau \frac{1}{\mu(\tau')} d \tau' - \frac{\mu(0) \Lambda\mathscr{A} y \times \mathbf{V}(0)}{1+\alpha} \int_0^\tau \frac{1}{\mu(\tau')} d \tau' \notag \\
& - \frac{\alpha \mu(0) \Lambda\mathscr{A} \nabla \beta \times \mathbf{V}(0)}{(1+\alpha)(1+\beta)}\int_0^{\tau} \frac{1}{\mu(\tau')} d \tau' + \int_0^\tau  [\partial_{\tau}, \text{\em Curl}_{\Lambda\mathscr{A}}] {\uptheta} d\tau' \notag \\
& + \frac{1}{1+\alpha} \int_0^\tau \Lambda\mathscr{A} y \times \mathbf{V} d\tau' + \frac{\alpha}{(1+\alpha)(1+\beta)} \int_0^\tau \Lambda\mathscr{A} \nabla \beta \times \mathbf{V} d\tau' \notag \\
&+\int_0^\tau \frac{1}{\mu(\tau')} \int_0^{\tau'} \mu(\tau'') [\partial_{\tau}, \text{\em Curl}_{\Lambda\mathscr{A}}] {\bf V} d\tau'' d \tau' -\frac{1}{1+\alpha} \int_0^{\tau} \frac{1}{\mu(\tau')} \int_0^{\tau'} \mu(\tau '') [\partial_{\tau},\Lambda \mathscr{A} y \times ] \mathbf{V} d\tau '' d \tau ' \notag \\
&- \frac{\alpha}{(1+\alpha)(1+\beta)} \int_0^{\tau} \frac{1}{\mu(\tau')} \int_0^{\tau'} \mu(\tau '') [\partial_{\tau},\Lambda \mathscr{A} \nabla \beta \times ] \mathbf{V} d \tau'' \notag \\
& - \int_0^\tau \frac{2}{\mu(\tau')} \int_0^{\tau '} \mu(\tau'') \, \text{\em Curl}_{\Lambda\mathscr{A}}(\Gamma^\ast{\bf V}) d\tau'' d \tau' + \frac{2}{1+\alpha} \int_0^{\tau} \frac{1}{\mu(\tau')} \int_0^{\tau '} \mu(\tau'') \, \Lambda\mathscr{A} x \times (\Gamma^\ast \mathbf{V}) d\tau '' d \tau ' \notag \\
&+ \frac{2 \alpha}{(1+\alpha)(1+\beta)} \int_0^{\tau} \frac{1}{\mu(\tau')} \int_0^{\tau '} \mu(\tau'') \, \Lambda\mathscr{A} \nabla \beta \times (\Gamma^\ast \mathbf{V}) d\tau '' d \tau ' \notag \\
& +\frac{\overline{C}}{1+\alpha} \int_0^\tau \frac{1}{\mu(\tau')} \int_0^{\tau'} \mu(\tau'')^{1-\delta-\sigma}  \Lambda x \times  \Lambda \uptheta d \tau '' d \tau ' \notag \\
& -\frac{\overline{C}}{1+\alpha} \int_0^\tau \frac{1}{\mu(\tau')} \int_0^{\tau'} \mu(\tau '')^{1-\delta-\sigma} \Lambda \mathscr{A}[D \uptheta]y \times \Lambda \eta d \tau '' d \tau ' \notag \\
&+\frac{\alpha \overline{C}}{(1+\alpha)(1+\beta)} \int_0^\tau \frac{1}{\mu} \int_0^{\tau'} \mu^{1-\delta-\sigma} \Lambda \mathscr{A} \nabla \beta \times \Lambda \uptheta d \tau'' d\tau' \notag \\
& -\frac{\alpha \overline{C}}{(1+\alpha)(1+\beta)} \int_0^{\tau} \frac{1}{\mu} \int_0^{\tau'} \mu^{1-\delta-\sigma} \Lambda \mathscr{A} \nabla \beta \times \Lambda y  d \tau'' d\tau'.
\end{align}
\end{lemma}
\begin{proof}
Writing (\ref{E:THETAEQNLINEARENERGYFUNCTION}) without the source term outside the nonlinearity, we have
\begin{equation}
\mu^{\sigma} \partial_{\tau \tau} \uptheta_i + \mu_{\tau} \mu^{-1+\sigma} \partial_\tau \uptheta_i + 2 \mu^{\sigma} \Gamma^*_{ij} \partial_{\tau} \uptheta_j + \overline{C} \mu^{-\delta} \Lambda_{i \ell} \uptheta_\ell + \frac{\overline{C} \mu^{-\delta}}{w} (w \Lambda_{ij} ( (1+\beta) \mathscr{A}_j^k \mathscr{J}^{-\frac{1}{\alpha}} - \delta_j^k))_{,k}=0.
\end{equation}
Return back to $\eta$ via $\eta=\uptheta+y$,
\begin{equation}
\mu^{\sigma} \partial_{\tau \tau} \eta_i + \mu_{\tau} \mu^{-1+\sigma} \partial_\tau \eta_i + 2 \mu^{\sigma} \Gamma^*_{ij} \partial_{\tau} \eta_j + \overline{C} \mu^{-\delta} \Lambda_{i \ell} \eta_\ell + \frac{\overline{C} \mu^{-\delta}}{w} (w \Lambda_{ij} (1+\beta) \mathscr{A}_j^k \mathscr{J}^{-\frac{1}{\alpha}})_{,k}=0.
\end{equation}
Multiply by $w^{\frac{1}{1+\alpha}}(1+\beta)^{-\frac{\alpha}{1+\alpha}}$
\begin{align}
&w^{\frac{1}{1+\alpha}} (1+\beta)^{-\frac{\alpha}{1+\alpha}} (\mu^{\sigma} \partial_{\tau \tau} \eta_i + \mu_{\tau} \mu^{-1+\sigma} \partial_\tau \eta_i + 2 \mu^{\sigma} \Gamma^*_{ij} \partial_{\tau} \eta_j + \overline{C} \mu^{-\delta} \Lambda_{i \ell} \eta_\ell) \notag \\
&+  w^{\frac{1}{1+\alpha}-1} (1+\beta)^{-\frac{\alpha}{1+\alpha}} \overline{C} \mu^{-\delta} (w \Lambda_{ij} \mathscr{A}^k_j \mathscr{J}^{-\frac{1}{\alpha}}),_k=0
\end{align}
Note
\begin{equation}
w^{\frac{1}{1+\alpha}-1} (1+\beta)^{-\frac{\alpha}{1+\alpha}} \overline{C} \mu^{-\delta} (w \Lambda_{ij} \mathscr{A}^k_j \mathscr{J}^{-\frac{1}{\alpha}}),_k = \overline{C} \mu^{-\delta} (1+\alpha) \Lambda_{i j} \mathscr{A}^k_j (w^{\frac{1}{1+\alpha}} (1+\beta)^{\frac{1}{1+\alpha}} \mathscr{J}^{-\frac{1}{\alpha}}),_k.
\end{equation}
Moving away from coordinates we then have
\begin{align}\label{E:AWAYFROMCOORDINATES}
&w^{\frac{1}{1+\alpha}} (1+\beta)^{-\frac{\alpha}{1+\alpha}}  (\partial_{\tau \tau} \uptheta + \mu_{\tau} \mu^{-1+\sigma} \uptheta + 2 \mu^{\sigma} \Gamma^* \partial_{\tau} \uptheta + \overline{C} \mu^{-\delta} \Lambda \eta) \notag \\
&+ \overline{C} \mu^{-\delta} (1+\alpha) \Lambda \mathscr{A}^T \nabla(w^{\frac{1}{1+\alpha}} (1+\beta)^{\frac{1}{1+\alpha}} \mathscr{J}^{-\frac{1}{\alpha}})=0.
\end{align}
Note
\begin{equation}
\overline{C} \mu^{-\delta} (1+\alpha) \Lambda \mathscr{A}^T \nabla(w^{\frac{1}{1+\alpha}} (1+\beta)^{\frac{1}{1+\alpha}} \mathscr{J}^{-\frac{1}{\alpha}})=\overline{C} \mu^{-\delta} (1+\alpha) \Lambda \nabla_{\eta} (w^{\frac{1}{1+\alpha}} (1+\beta)^{\frac{1}{1+\alpha}} \mathscr{J}^{-\frac{1}{\alpha}}).
\end{equation}
Since $\text{Curl}_{\Lambda \mathscr{A}} (\Lambda \nabla_\eta f) =0$, apply $\text{Curl}_{\Lambda \mathscr{A}}$ to (\ref{E:AWAYFROMCOORDINATES}),
\begin{equation}\label{E:PREDIVIDESQUAREROOTW}
\text{Curl}_{\Lambda \mathscr{A}} \left( w^{\frac{1}{1+\alpha}} (1+\beta)^{-\frac{\alpha}{1+\alpha}} (\mu^\sigma \partial_{\tau \tau} \uptheta + \mu^{-1+\sigma} \mu_{\tau} \partial_\tau \uptheta + 2 \mu^\sigma \Gamma^{*} \partial_\tau \uptheta + \overline{C} \mu^{-\delta} \Lambda \eta) \right) =0.
\end{equation}
Now note
\begin{equation}
\text{Curl}_{\Lambda \mathscr{A}}(w^{\frac{1}{1+\alpha}} (1+\beta)^{-\frac{\alpha}{1+\alpha}} \mathbf{F})=w^{\frac{1}{1+\alpha}} (1+\beta)^{-\frac{\alpha}{1+\alpha}} \, \text{Curl}_{\Lambda \mathscr{A}} \mathbf{F} + \Lambda \mathscr{A} \nabla(w^{\frac{1}{1+\alpha}} (1+\beta)^{-\frac{\alpha}{1+\alpha}}) \times \mathbf{F}.
\end{equation}
Also
\begin{align}
(w^{\frac{1}{1+\alpha}} (1+\beta)^{-\frac{\alpha}{1+\alpha}}),_s  &= \frac{1}{1+\alpha} w^{\frac{1}{1+\alpha}-1} w,_s (1+\beta)^{-\frac{\alpha}{1+\alpha}} -\frac{\alpha}{1+\alpha} (1+\beta)^{-\frac{\alpha}{1+\alpha}-1} \beta,_s w^{\frac{1}{1+\alpha}}  \notag \\
&= -\frac{1}{1+\alpha} w^{\frac{1}{1+\alpha}} (1+\beta)^{-\frac{\alpha}{1+\alpha}} y_s -\frac{\alpha}{1+\alpha} (1+\beta)^{-\frac{1+2\alpha}{1+\alpha}} \beta,_s w^{\frac{1}{1+\alpha}}
\end{align}
since $w,_s = - y_s w$. So
\begin{equation}
\Lambda \mathscr{A} \nabla(w^{\frac{1}{1+\alpha}} (1+\beta)^{-\frac{\alpha}{1+\alpha}}) \times \mathbf{F} = - \frac{w^{\frac{1}{1+\alpha}} (1+\beta)^{-\frac{\alpha}{1+\alpha}}}{1+\alpha} \Lambda \mathscr{A} y \times \mathbf{F} - \frac{\alpha(1+\beta)^{-\frac{1+2\alpha}{1+\alpha}} w^{\frac{1}{1+\alpha}}}{1+\alpha} \Lambda \mathscr{A} \nabla \beta \times \mathbf{F}
\end{equation}
where
\begin{equation}
[\Lambda \mathscr{A} y \times \mathbf{F}]^i_j := \Lambda_{jm} \mathscr{A}^s_m x_s \mathbf{F}^i - \Lambda_{im} \mathscr{A}^s_m y_s \mathbf{F}^j.
\end{equation}
So multiplying (\ref{E:PREDIVIDESQUAREROOTW}) by $w^{-\frac{1}{1+\alpha}}(1+\beta)^{\frac{\alpha}{1+\alpha}}$ we have and using $\text{Curl}_{\Lambda \mathscr{A}} (\Lambda \eta)=0$ we have
\begin{align}
&\mu^\sigma \text{Curl}_{\Lambda \mathscr{A}} (\partial_{\tau \tau} \uptheta) + \mu^{-1+\sigma} \mu_{\tau} \text{Curl}_{\Lambda \mathscr{A}}(\partial_\tau \uptheta) + 2 \mu^\sigma \text{Curl}_{\Lambda \mathscr{A}}(\Gamma^{*} \partial_\tau \uptheta) \notag \\
&-\frac{1}{1+\alpha} \Lambda \mathscr{A} x \times \left( \mu^\sigma \partial_{\tau \tau} \uptheta + \mu^{-1+\sigma} \mu_{\tau} \partial_\tau \uptheta + 2 \mu^\sigma \Gamma^{*} \partial_\tau \uptheta + \overline{C} \mu^{-\delta} \Lambda \eta\right) \notag \\
&-\frac{\alpha}{(1+\alpha)(1+\beta)} \Lambda \mathscr{A} \nabla \beta \times \left( \mu^\sigma \partial_{\tau \tau} \uptheta + \mu^{-1+\sigma} \mu_{\tau} \partial_\tau \uptheta + 2 \mu^\sigma \Gamma^{*} \partial_\tau \uptheta + \overline{C} \mu^{-\delta} \Lambda \eta\right)=0.
\end{align}
Divide by $\mu^{\sigma}$
\begin{align}
&\text{Curl}_{\Lambda \mathscr{A}} (\partial_{\tau \tau} \uptheta) + \mu^{-1} \mu_{\tau} \text{Curl}_{\Lambda \mathscr{A}}(\partial_\tau \uptheta) + 2 \text{Curl}_{\Lambda \mathscr{A}}(\Gamma^{*} \partial_\tau \uptheta) \notag \\
&-\frac{1}{1+\alpha} \Lambda \mathscr{A} x \times \left( \partial_{\tau \tau} \uptheta + \mu^{-1} \mu_{\tau} \partial_\tau \uptheta + 2  \Gamma^{*} \partial_\tau \uptheta + \overline{C} \mu^{-\delta-\sigma} \Lambda \eta\right) \notag \\
&-\frac{\alpha}{(1+\alpha)(1+\beta)} \Lambda \mathscr{A} \nabla \beta \times \left( \partial_{\tau \tau} \uptheta + \mu^{-1} \mu_{\tau} \partial_\tau \uptheta + 2  \Gamma^{*} \partial_\tau \uptheta + \overline{C} \mu^{-\delta-\sigma} \Lambda \eta\right)=0.
\end{align}
Note
\begin{equation}
\text{Curl}_{\Lambda \mathscr{A}} \left(\mathbf{V}_ \tau\right)=\partial_\tau \left( \, \text{Curl}_{\Lambda\mathscr{A}}\left({\bf V}\right)   \right)-[\partial_\tau, \text{Curl}_{\Lambda\mathscr{A}}] \left({\bf V}\right),
\end{equation}
where
\begin{equation}\label{E:CURLCOMMUTATOR}
[\partial_\tau, \text{Curl}_{\Lambda\mathscr{A}}] \mathbf{F}^i_j := \partial_\tau \left(\Lambda_{jm}\mathscr{A}^s_m\right) \mathbf{F},_s^i - \partial_\tau \left(\Lambda_{im}\mathscr{A}^s_m\right) \mathbf{F},_s^j. 
\end{equation}
Then
\begin{align}\label{E:CURLPREINTGAMMALEQ5OVER3}
&\partial_\tau \left(\mu\, \text{Curl}_{\Lambda\mathscr{A}}\left({\bf V}\right)   \right) = \mu [\partial_\tau, \text{Curl}_{\Lambda\mathscr{A}}] \left({\bf V}\right)  - 2 \mu \, \text{Curl}_{\Lambda\mathscr{A}}\left(\Gamma^\ast{\bf V}\right) \notag \\
&+\frac{1}{1+\alpha} \Lambda \mathscr{A} y \times \left( \mu \partial_{\tau \tau} \uptheta + \mu_{\tau} \partial_\tau \uptheta + 2 \mu \Gamma^{*} \partial_\tau \uptheta + \overline{C} \mu^{1-\delta-\sigma} \Lambda \eta\right) \notag \\
&+\frac{\alpha}{(1+\alpha)(1+\beta)} \Lambda \mathscr{A} \nabla \beta \times \left( \mu \partial_{\tau \tau} \uptheta + \mu_{\tau} \partial_\tau \uptheta + 2 \mu \Gamma^{*} \partial_\tau \uptheta + \overline{C} \mu^{1-\delta-\sigma} \Lambda \eta\right).
\end{align}
Integrate from $0$ to $\tau'$, where $\tau' \in [0,\tau]$,
\begin{align}\label{E:CURLEXPMULTVGAMMALEQ5OVER3}
\text{Curl}_{\Lambda\mathscr{A}}\left({\bf V}\right)  &= \frac{\mu(0) \text{Curl}_{\Lambda \mathscr{A}} \left([{\bf V}(0)]\right)}{\mu}+\frac{1}{\mu}\int_0^{\tau'} \mu [\partial_{\tau}, \text{Curl}_{\Lambda\mathscr{A}}] \left({\bf V}\right) d\tau'' - \frac{2}{\mu} \int_0^{\tau'} \mu \, \text{Curl}_{\Lambda\mathscr{A}}\left(\Gamma^\ast{\bf V}\right) d\tau'' \notag \\
&+\frac{1}{(1+\alpha) \mu} \int_0^{\tau'}  \Lambda \mathscr{A} x \times \left( \mu \partial_{\tau \tau} \uptheta + \mu_{\tau} \partial_{\tau} \uptheta + 2 \mu \Gamma^{*} \partial_{\tau} \uptheta + \overline{C} \mu^{1-\delta-\sigma} \Lambda \eta\right) d \tau'' \notag \\
&+\frac{\alpha}{(1+\alpha)(1+\beta)\mu} \int_0^{\tau'}  \Lambda \mathscr{A} \nabla \beta \times \left( \mu \partial_{\tau \tau} \uptheta + \mu_{\tau} \partial_{\tau} \uptheta + 2 \mu \Gamma^{*} \partial_{\tau} \uptheta + \overline{C} \mu^{1-\delta-\sigma} \Lambda \eta\right) d \tau''.
\end{align}
Note
\begin{equation}
\mu (\Lambda \mathscr{A} x \times \partial_{\tau \tau} \uptheta) = \partial_{\tau} ( \mu \Lambda \mathscr{A} x \times \partial_{\tau} \uptheta ) - \mu_{\tau} (\Lambda \mathscr{A} x \times \partial_{\tau} \uptheta) - \mu [\partial_{\tau},\Lambda \mathscr{A} x \times]\partial_{\tau} \uptheta,
\end{equation}
where
\begin{equation} 
[\partial_{\tau},\Lambda \mathscr{A} x \times ] \mathbf{F}^i_j := \partial_{\tau} (\Lambda_{jm}\mathscr{A}^s_m)x_s \mathbf{F}^i - \partial_{\tau}(\Lambda_{im} \mathscr{A}^s_m)y_s\mathbf{F}^j.
\end{equation}
Using a similar result for $\mu (\Lambda \mathscr{A} \nabla \beta \times \partial_{\tau \tau} \uptheta)$, we have
\begin{align}\label{E:CURLVPOSTCURLEXPAND2GAMMALEQ5OVER3}
&\text{Curl}_{\Lambda\mathscr{A}}{\bf V}  = \frac{1}{1+\alpha}\Lambda\mathscr{A} y \times \mathbf{V} + \frac{\alpha}{(1+\alpha)(1+\beta)} \Lambda\mathscr{A} \nabla \beta \times \mathbf{V} \notag \\
&+  \frac{\mu(0) \text{Curl}_{\Lambda \mathscr{A}} ({\bf V}(0))}{\mu} - \frac{\mu(0) \Lambda\mathscr{A} x \times \mathbf{V}(0)}{(1+\alpha) \mu} - \frac{\alpha \mu(0) \Lambda\mathscr{A} \nabla \beta \times \mathbf{V}(0)}{(1+\alpha)(1+\beta) \mu} \notag \\
& + \frac{1}{\mu}\int_0^{\tau'} \mu [\partial_{\tau}, \text{Curl}_{\Lambda\mathscr{A}}] {\bf V} d\tau'' - \frac{1}{(1+\alpha) \mu} \int_0^{\tau'} \mu [\partial_{\tau},\Lambda \mathscr{A} x \times ] \mathbf{V} d \tau'' \notag \\
&- \frac{\alpha}{(1+\alpha)(1+\beta) \mu} \int_0^{\tau'} \mu [\partial_{\tau},\Lambda \mathscr{A} \nabla \beta \times ] \mathbf{V} d \tau'' \notag \\
& - \frac{2}{\mu} \int_0^{\tau'} \mu \, \text{Curl}_{\Lambda\mathscr{A}}(\Gamma^\ast{\bf V}) d\tau'' + \frac{2}{(1+\alpha) \mu} \int_0^{\tau'} \mu \, \Lambda\mathscr{A} x \times (\Gamma^\ast \mathbf{V}) d \tau'' \notag \\
& + \frac{2\alpha}{(1+\alpha)(1+\beta) \mu} \int_0^{\tau'} \mu \, \Lambda\mathscr{A} \nabla \beta \times (\Gamma^\ast \mathbf{V}) d \tau'' \notag \\
&+\frac{\overline{C}}{(1+\alpha) \mu}\int_0^{\tau'} \mu^{1-\delta-\sigma} \Lambda\mathscr{A} y \times (\Lambda \eta) d \tau''+\frac{\alpha \overline{C}}{(1+\alpha)(1+\beta) \mu}\int_0^{\tau'} \mu^{1-\delta-\sigma} \Lambda\mathscr{A} \nabla \beta \times (\Lambda \eta) d \tau''.
\end{align}
Now
\begin{align}\label{E:CURLLASTTERMIDENTITYGAMMALEQ5OVER3}
&[\Lambda\mathscr{A} x \times (\Lambda \eta)]_j^i = \Lambda_{jm}(\delta_m^s - \mathscr{A}_\ell^s [D \uptheta]_m^\ell) x_s \Lambda_{ik} \eta^k - \Lambda_{im}(\delta_m^s - \mathscr{A}_\ell^s [D \uptheta]_m^\ell ) y_s \Lambda_{jk} \eta^k \notag \\
&= \Lambda_{js} y_s \Lambda_{ik}\uptheta^k - \Lambda_{is} y_s \Lambda_{jk} \uptheta^k+\Lambda_{js} y_s\Lambda_{ik}y^k - \Lambda_{is} y_s \Lambda_{jk} y^k \notag \\
&\qquad \qquad \qquad -(\Lambda_{jm}\mathscr{A}_\ell^s [D \uptheta]_m^\ell y_s \Lambda_{ik} \eta^k - \Lambda_{im} \mathscr{A}_\ell^s [D \uptheta]_m^\ell y_s \Lambda_{jk} \eta^k) \notag \\
&=(\Lambda_{js} y_s \Lambda_{ik}\uptheta^k - \Lambda_{is} y_s \Lambda_{jk} \uptheta^k) -(\Lambda_{jm}\mathscr{A}_\ell^s [D \uptheta]_m^\ell x_s \Lambda_{ik} \eta^k - \Lambda_{im} \mathscr{A}_\ell^s [D \uptheta]_m^\ell y_s \Lambda_{jk} \eta^k) \notag \\
&:=[\Lambda y \times  \Lambda \uptheta]_j^i- [\Lambda \mathscr{A}[D \uptheta]y \times \Lambda \eta]_j^i.
\end{align}
Second notice that
\begin{align}\label{E:CURLLASTTERMIDENTITYGAMMALEQ5OVER3}
[\Lambda\mathscr{A} \nabla \beta \times (\Lambda \eta)]_j^i &= \Lambda_{jm}\mathscr{A}^s_m \beta_{,s} \Lambda_{ik} \uptheta^k - \Lambda_{im}\mathscr{A}^s_m \beta_{,s} \Lambda_{jk} \uptheta^k +\Lambda_{jm}\mathscr{A}^s_m \beta_{,s} \Lambda_{ik} y^k - \Lambda_{im}\mathscr{A}^s_m \beta_{,s} \Lambda_{jk} y^k \notag \\
&=[\Lambda \mathscr{A} \nabla \beta \times \Lambda \uptheta]_j^i + [\Lambda \mathscr{A} \nabla \beta \times \Lambda y]_j^i.
\end{align}
Hence we have (\ref{E:CURLVFINAL}).
Now for $\text{Curl}_{\Lambda \mathscr{A}} \uptheta$, first note
\begin{equation}\label{E:CURLTHETADERIVATIVE}
\partial_{\tau} (\text{Curl}_{\Lambda \mathscr{A}} \uptheta) = \text{Curl}_{\Lambda \mathscr{A}} (\partial_{\tau} \uptheta) + [\partial_{\tau},\text{Curl}_{\Lambda \mathscr{A}}] \uptheta.
\end{equation}
So integrating (\ref{E:CURLTHETADERIVATIVE}) from $0$ to $\tau$ via (\ref{E:CURLVPOSTCURLEXPAND2GAMMALEQ5OVER3}) we have (\ref{E:CURLTHETAFINAL}).
\end{proof}

\section{Coercivity Estimates}\label{A:COERCIVITY}
We give a useful result which will allow us to overcome the time weights with negative powers which arise from our equation structure.

\begin{lemma}[Coercivity Estimates]\label{L:COERCIVITY}
Let $(\uptheta, {\bf V}):\Omega \rightarrow \mathbb R^3\times \mathbb R^3$ be a unique local solution to (\ref{E:THETAEQNLINEARENERGYFUNCTION})-(\ref{E:THETAICGAMMALEQ5OVER3}) on $[0,T^*]$ for $T^*>0$ fixed with $\text{supp} \,\uptheta_0 \subseteq B_1(\mathbf{0})$, $\text{supp}\,\mathbf{V}_0 \subseteq B_1(\mathbf{0})$ and assume $(\uptheta, {\bf V})$ satisfies the a priori assumptions (\ref{E:APRIORI}). Fix $N\geq 4$.  Suppose $\beta$ in (\ref{E:THETAEQNLINEARENERGYFUNCTION}) satisfies $\| \beta \|^2_{H^{N+1}(\mathbb{R}^3)} \leq \lambda$ and $\text{supp} \, \beta \subseteq B_1(\mathbf{0})$ where  $\lambda > 0$ is fixed.  Fix $\nu$ with $0 \leq |\nu| \leq N-1$. Then for all $\tau \in [0,T^*]$, we have the following inequalities
\begin{align}
\| \partial^\nu \uptheta\|^2 &\lesssim \sup_{0 \leq \tau \leq \tau'}  \{ \mu^\sigma \| \partial^\nu \mathbf{V}\|^2 \} + \| \partial^\nu \uptheta (0)\|^2 \label{E:COERCIVITY1} \\
\| \nabla_{\eta} \partial^\nu \uptheta\|^2 &\lesssim \sup_{0 \leq \tau \leq \tau'}  \{ \sum_{|\nu'| = |\nu|+1} \mu^\sigma \| \partial^{\nu'} \mathbf{V}\|^2 \} + \| \nabla_{\eta} \partial^\nu \uptheta (0)\|^2 \label{E:COERCIVITY2} \\
\| \text{\em div}_{\eta} \partial^\nu \uptheta\|^2 &\lesssim \sup_{0 \leq \tau \leq \tau'}  \{ \sum_{|\nu'| = |\nu|+1} \mu^\sigma \| \partial^{\nu'} \mathbf{V}\|^2 \} + \| \text{\em div}_{\eta} \partial^\nu \uptheta (0)\|^2. \label{E:COERCIVITY3}
\end{align}
\end{lemma} 
\begin{proof}
\textit{Proof of} (\ref{E:COERCIVITY1}). By the fundamental theorem of calculus, and the exponential boundedness of $\mu$ (\ref{E:EXPMU1MUINEQGAMMALEQ5OVER3}) and therefore time integrability of negative powers of $\mu$,
\begin{align}\label{E:COERCIVITYPROOF1A}
\partial^\nu \uptheta = \int_0^\tau \partial^\nu \mathbf{V} d \tau' + \partial^\nu  \uptheta (0) &= \int_0^\tau \mu^{-\tfrac{\sigma}{2}} \mu^{\tfrac{\sigma}{2}} \partial^\nu \mathbf{V} d \tau' + \partial^\nu \uptheta (0) \notag \\
& \lesssim \sup_{0 \leq \tau \leq \tau'} \{ \mu^{\tfrac{\sigma}{2}} \partial^\nu \mathbf{V} \} +\partial^\nu \uptheta (0).
\end{align}      
Therefore applying Cauchy's inequality ($ab \lesssim a^2 + b^2,$ $a,b \in \mathbb{R}$) 
\begin{equation}\label{E:COERCIVITYPROOF1B}
\| \partial^\nu \uptheta\|^2 \lesssim \sup_{0 \leq \tau \leq \tau'} \{ \mu^\sigma \| \partial^\nu \mathbf{V}\|^2 \} + \| \partial^\nu \uptheta (0)\|^2.    
\end{equation}        
\textit{Proof of} (\ref{E:COERCIVITY2}). By a similar coercivity estimate to (\ref{E:COERCIVITYPROOF1A})-(\ref{E:COERCIVITYPROOF1B})
\begin{equation}\label{E:COERCIVITYPROOF2A}
\| \nabla_{\eta}  \partial^\nu \uptheta\|^2 \lesssim \sup_{0 \leq \tau \leq \tau'} \{ \mu^\sigma \| \nabla_{\eta}  \partial^\nu \mathbf{V}\|^2 \} + \| \nabla_{\eta}  \partial^\nu \uptheta (0)\|^2. 
\end{equation}
Now using our a priori bounds (\ref{E:APRIORI}), we have
\begin{equation}\label{E:COERCIVITYPROOF2B}
\sup_{0 \leq \tau \leq \tau'} \mu^\sigma \| \nabla_{\eta} \partial^\nu \mathbf{V}\|^2 \lesssim \sup_{0 \leq \tau \leq \tau'} \{  \sum_{|\nu'| = |\nu|+1} \mu^\sigma \| \partial^{\nu'} \mathbf{V}\|^2 \} . 
\end{equation}      
Then (\ref{E:COERCIVITYPROOF2A})-(\ref{E:COERCIVITYPROOF2B}) imply (\ref{E:COERCIVITY2}). \\ \\ 
\textit{Proof of} (\ref{E:COERCIVITY3}). Finally the proof of (\ref{E:COERCIVITY3}) is similar to the proof of (\ref{E:COERCIVITY2}).
\end{proof}  

\end{document}